\documentclass[12pt]{amsart}
\usepackage[letterpaper,margin=1in,centering,marginparwidth=.95in,marginparsep=1mm]{geometry}
\usepackage{amsmath,amssymb,thmtools}
\usepackage{xcolor}\def\red{\color{red}}
\usepackage{lineno}
\usepackage{enumitem}
\usepackage{lipsum}
\usepackage[normalem]{ulem}
\usepackage[noadjust]{cite}
\usepackage{yfonts}
\usepackage{thmtools}
\usepackage{mathtools}
\usepackage{comment}
\usepackage{soul}
\usepackage{graphicx}
\usepackage[pagebackref,hidelinks]{hyperref}
\usepackage{tikz}
\usetikzlibrary{arrows.meta}
\usetikzlibrary{intersections}
\usetikzlibrary{decorations.markings, arrows, arrows.meta}
\usetikzlibrary{decorations.pathreplacing,angles,quotes}
\graphicspath{{build/tikz/}}
\usetikzlibrary{decorations.markings,arrows.meta,angles,quotes}
\graphicspath{ {./images/} }

\usepackage[capitalize,nameinlink,noabbrev]{cleveref}

\newtheorem{theorem}{Theorem}[section]
\newtheorem{lemma}[theorem]{Lemma}
\newtheorem{corollary}[theorem]{Corollary}
\newtheorem{proposition}[theorem]{Proposition}
\newtheorem{conjecture}[theorem]{Conjecture}
\theoremstyle{definition}
\newtheorem{definition}[theorem]{Definition}
\newtheorem{remark}[theorem]{Remark}



\newif\ifcomments
\newcounter{commentlabel}
  \newlength{\commentlift}
\makeatletter\@mparswitchfalse\def\commentfl@g{%
    \vbox to\z@{%
      \setlength{\fboxrule}{0.75pt}%
      \setlength{\fboxsep}{0.75pt}%
      \setlength{\commentlift}{1ex}%
      \addtolength{\commentlift}{\fboxrule}%
      \addtolength{\commentlift}{\fboxsep}%
\vss\color{blue}\rlap{\rlap{\vrule\@height\commentlift\@width\fboxrule}\raise \commentlift%
      \hbox{\fcolorbox{blue}{blue!3}{\normalfont\tiny\bfseries\thecommentlabel}}}}}
  \DeclareRobustCommand{\COMMENT}[1]{\unskip\global\commentstrue\stepcounter{commentlabel}%
\hypertarget{\the\value{section}-\the\value{commentlabel}}{}%
\edef\WRITECOM##1{\noexpand\write\COM{\thecommentlabel, S.##1}}\WRITECOM{\the\value{section}, \noexpand\hyperlink{\the\value{section}-\the\value{commentlabel}}{p.\the\value{page}}: {#1}}\@bsphack%
    \commentfl@g%
    \marginpar{\fontfamily{EBGaramond-TLF}\selectfont\noindent\lineskip=1pt\lineskiplimit=\maxdimen\raggedright\fontfamily{put}\selectfont %
    \Tiny\def\Hrule{\vphantom{p}\hrule height 1pt\vphantom{!}}\textbf{\color{blue}\llap{\textperiodcentered}\thecommentlabel:}{\color{blue}\thinspace#1\endgraf}}%
  \@esphack}%
  \DeclareRobustCommand{\COMMENTINLINE}[1]{\commentstrue\stepcounter{commentlabel}%
    \write\COM{\thecommentlabel, S.\the\value{section}, p.\the\value{page}: {#1}}\@bsphack%
      \noindent\setlength{\fboxsep}{0pt}\newline\fcolorbox{blue}{lightgray!20}{\begin{minipage}{\textwidth}\tiny\textbf{\color{red}\thecommentlabel}:\thinspace\Tiny\bfseries#1\par\end{minipage}}\allowbreak%
  \@esphack}
\newwrite\COM
\immediate\openout\COM=comments
\AtEndDocument{\ifcomments\def\red{\color{red}}\section*{Marginal Comments}\label{commentstex}{\let\tiny\relax\let\Tiny\relax\let\Hrule\newline\def\pullout#1, S.{\strut\llap{\textbf{#1: }}\S\ }\obeylines\parindent0pt\everypar{\normalsize\normalfont\pullout}\immediate\closeout\COM\def\MT_extended_cref:n {\cref}\def\math@bb{\mathbb }\input comments }\fi}
\let\Hrule\relax
\makeatother

\newcommand{\done}{\newline\textbf{\color{red}---Made small changes, will think much more carefully.}}
\newcommand{\DONE}{\newline\textbf{\color{red}---This comment has been addressed. It should now be removed.}}
\newcommand{\wtilde}[1]{\,\widetilde{\!#1}}
\newcommand{\oline}[1]{\,\overline{\!#1}}
\newcommand{\dfn}{\ensuremath{\mathbin{{:}{=}}}}
\newcommand{\nfd}{\ensuremath{\mathbin{{=}{:}}}}

\newcommand \eps{\epsilon}
\newcommand \re{{\mathbb R}}

\newcommand \gav{\ga_v}
\newcommand \la{\lambda}
\newcommand \ga{\gamma}
\newcommand \las{\la^s}
\newcommand \sm{T^1\Sigma}
\newcommand \lau{\la^u}
\newcommand \lav{\la(v)}
\newcommand \de{\delta}

\newcommand \ftv{f_tv}

\newcommand \ws{W^s}
\newcommand \wu{W^u}

\newcommand \fs{\mathcal{F}^s}

\newcommand \M{\tilde \Sigma}
\newcommand \Minfty{\M(\infty)}
\newcommand \om{\,\overline{\Sigma}}
\newcommand \fmu{f^\mu}
\newcommand \kmu{K_\mu}
\newcommand \m{\ensuremath\mu-magnetic}
\newcommand{\PPP}{\mathcal{P}}
\newcommand{\SSS}{\mathcal{S}}
\newcommand{\GGG}{\mathcal{G}}
\newcommand{\BBB}{\mathcal{B}}
\newcommand{\CCC}{\mathcal{C}}

\newcommand{\Reg}{\mathrm{Reg}}
\newcommand{\Sing}{\mathrm{Sing}}
\newcommand{\td}{\tilde d}
\newcommand{\dmu}{d_\mu}
\newcommand{\barc}{\bar c}

\DeclareMathOperator{\emu}{\mu exp}
\newcommand{\h}{\text{If }(\Sigma,\mu) \text{ satisfies }\ref{H1} \text{\red\bfseries\upshape\large[Bad macro, replace]}}

\renewcommand{\h}{{\red If \ensuremath{(\Sigma,\mu)} satisfies \ref{H1}}}

\newcommand{\smm}{T^1\M}
\newcommand{\bmu}{B^\mu}
\newcommand{\Per}{\mathrm{Per}}
\newcommand{\f}{f^\mu}
\DeclareMathOperator{\rank}{rank_\mu}
\let\emptyset\varnothing
\let\implies\Rightarrow\let\iff\Leftrightarrow
\begin{document}
\title{Surfaces with nonpositive magnetic curvature}
\author{Boris Hasselblatt}
\address{Department of Mathematics, Tufts University, Medford, MA 02155, USA}
\email{boris.hasselblatt@tufts.edu}
\author{Jincheng Wang}
\address{Department of Mathematics, Tufts University, Medford, MA 02155, USA}
\email{JinCheng.Wang@tufts.edu}
\keywords{Magnetic flows, hyperbolic dynamics, equilibrium states}
\subjclass{37D40,37D35}
\begin{abstract}

For a negatively curved surface, magnetic (or twisted geodesic) flows with nonpositive ``magnetic curvature'' have a large-scale geometry on the universal cover that retains some features associated with geodesic flows: magnetic boundary and visibility, possibly magnetically flat strips. This has dynamical consequences: orbit-equivalence to the geodesic flow, density of periodic points, transitivity, entropy-expansivity, existence of equilibrium states.

\end{abstract}
\date{\today}
\maketitle
\tableofcontents
\section{Introduction}\label{section:intro}
The unit-speed geodesic flow of a (here negatively curved) surface can be modified by the addition of a Lorentz force modeling the effect of a magnetic field on a charged particle to produce what is called a magnetic (or twisted geodesic) flow. For weak magnetic fields, this retains the uniform hyperbolicity of the geodesic flow, and we study the edge case when the magnetic flow is not uniformly hyperbolic but nonuniformly so, or, more precisely, where the ``magnetic curvature'' is nonpositive but not everywhere negative. The aim is to understand the topological dynamics and ergodic theory, and the approach centers on understanding the large-scale geometry on the universal cover, following the blueprint for the study of the geodesic flow. Thus, the subject of this work is geometry, mainly, but the object is dynamics.

\subsection{Main results}\label{SECMainReults}
We produce a boundary at infinity, the visibility property, an abundance of closed orbits, topological transitivity, expansivity, orbit-equivalence to the geodesic flow, and equilibrium states. However, the most natural definition of a distance function does not yield a distance due to irreversibility: symmetry and the triangle inequality both fail. This  stands in the way of the most obvious approach to natural further developments.

Let $\Sigma$ be a closed Riemannian surface with Gaussian curvature $K$ and denote by $\sm \dfn \{v\in T\Sigma \mid |v|=1\}$ its unit tangent bundle. Given $\mu\in\re$, the \m\ flow $\f\colon \sm \to\sm$ (\cref{def:mag flow}) is an Anosov flow if the \emph{\m\ curvature} $\kmu(p)\coloneqq K(p)+\mu^2$ is negative for all $p\in\Sigma$ (i.e., $\ref{H4}\implies\ref{HA}$ below). 
\begin{definition}[Hyperbolic magnetic flow]\label{DEFHypMag}Let $\Sigma$ be a closed surface with Gaussian curvature $K$ and let $\mu\in\re$. Cases for  $(\Sigma,\mu)=(\Sigma,g,\Omega_g,\mu)$ or $\f$:
\begin{enumerate}[label=\textbf{\upshape{(H\arabic{*})}}]\setcounter{enumi}{-1}
   \item\label{H0}$K<0$ (a standing assumption);
   \item\label{H1} nonpositive magnetic curvature: \ref{H0} and $K_{\mu}\leq0$;
   \item\label{H2}(``Not flat'') \ref{H1} and $K_{\mu}(p)<0$ for some $p\in \Sigma$ (see \cref{prop:magnetic-rank-one-implies-h2});
   \item\label{H3} (``No flat strips'') \ref{H2} and for any two distinct points $x,y\in\Minfty$, there exists at most one \m\ geodesic from $x$ to $y$ (see \cref{thm:ptildeqtilde-magnetic-visibility-property}), i.e., $(\Sigma,\mu)$ has no magnetically flat strips (\cref{def:magnetic-flat-strip});
   \item\label{HA}$\f$ is Anosov;   \item\label{H4}$K_\mu<0$.
\end{enumerate}
If \ref{H2}, then we say that $\f$ is a \emph{hyperbolic magnetic flow} and that $\Sigma$ is a \emph{\m\  rank-1} surface, which means \ref{H1} and that it has a \m\  rank-1 \m\  geodesic, that is, a \m\  geodesic that  admits only tangential parallel \m\  Jacobi fields (\cref{def:magnetic-rank-one,prop:magnetic-rank-one-implies-h2}).

We say that $\Sigma$ is \emph{$\mu$-magnetically flat} if \ref{H1} but not \ref{H2}, i.e., $K_\mu\equiv0$, hence  $K\equiv -\mu^2$.
\end{definition} 
If $\Sigma$ is \emph{$\mu$-magnetically flat}, then $\f$ is the horocycle flow; accordingly, we should implicitly think of $K$ (hence $K_\mu$) as nonconstant throughout; indeed, otherwise, $\ref{H2} \implies \ref{H4}$. 

\ref{HA} gives the full thermodynamic formalism \cite{Franco}, and we provide tools that should help establish a ``magnetic counterpart'' to \cite{BCFT} assuming only \ref{H2} (or \ref{H3}) \cite{ColesHasselblattWu}.

Our main results assume \ref{H1} throughout, unless stated otherwise:
\begin{description}
    \item[Magnetic boundary (\cref{prop:boundary-homeomorphic})]The \m\ boundary $\M_\mu(\infty)$, defined as equivalence classes of  asymptotic \m\ geodesics (\cref{def:asymptotic-maggeo-mag-boundary}), is homeomorphic to $\Minfty$ (\cref{def:geo-boundary-point-asymp-class}) in a natural way. 
    \item[Endpoints and flatness (\cref{prop:h2-has-no-closed-orbit})]$\Sigma$ is $\mu$-magnetically flat if and only if there is a \m\  geodesic with the same forward and backward endpoints on $\Minfty$; in this case, all \m\ geodesics have this property. 
    \item[Magnetic visibility (\cref{thm:ptildeqtilde-magnetic-visibility-property})] $\ref{H2}\implies\;$for any two distinct points $x,y\in \Minfty$, there exists a \m\  geodesic $\ga$ from $x$ to $y$. (Uniqueness$\;\implies\ref{H3}$.)
    \item[Magnetic flat strips (\cref{prop:magnetic-flat-strip-is-connected,lemma:magnetic-flat-strip-parallel})] For $x\neq y\in\Minfty$, the union of all \m\ geodesics from $x$ to $y$ is a either a \m\ geodesic or a \emph{\m\  flat strip} $S\colon[0,C]\times \re\to\M$, a parallel \m\ variation (\cref{def:stable-unstable-magJF}) consisting of points with zero magnetic curvature, where $S(c,\cdot)$ are \m\ geodesics with $d(S(c_1,t),S(c_2,\re))$ constant in $t\in \re$ for any $c_1,c_2\in [0,C]$.
    
    \item[Orbit-equivalence (\cref{thm:orbit-equivalence})] A \m\ flow is orbit-equivalent to the geodesic flow of $\Sigma$ iff \ref{H3}. $\ref{H2}\implies\;$the geodesic flow is an orbit factor (\cref{remark:orbit-factor}).
    \item[Expansivity (\cref{prop:expansivity-of-mag-flows,prop:mag-flow-is-h-expansive,thm:kinematic-expansivity})] A \m\ flow is kinematic-expansive, so its time-1 map is entropy-expansive. It is expansive iff \ref{H3}.
    \item[Periodic orbits (\cref{prop:periodic-orbits-vectors,cor:nw-is-TM})] $\ref{H2}\implies\;$there are infinitely many periodic orbits, periodic vectors are dense, and every vector is nonwandering.
    \item[Topological transitivity (\cref{thm:topologically-transitivity})] A magnetic flow is topologically transitive (\cref{def:topo_transitivity}). $\ref{H2}\implies\;$any dense orbit is regular (\cref{def:sing-reg-mag-vectors}).
    \item[Equilibrium states (\cref{thm:existence-of-es-equidistribution})] A magnetic flow has an equilibrium state for any Bowen-bounded potential function, in particular, a measure of maximal entropy. 
    \item[Magnetic triangle inequality (\cref{lemma:partial-magnetic-triangle-inequality,lemma:magnetic-triangle-inequality-fails})]Defining the distance between $p,q\in\M$ as the length of the unique \m\  geodesic $\ga_{pq}$ from $p$ to $q$ (\cref{def:mag-distance}) is asymmetric (\cref{cor:td-is-not-metric}), and a magnetic triangle inequality only holds on one side of $\ga_{pq}$ (\cref{fig:mag tri ineq,fig:magtriineqfailed}).
    \item[Magnetic \emph{orthosphere} (\cref{prop:c2-regularity-of-magnetic-horosphere})] Horospheres are both limit spheres and orthogonal curves to asymptotic geodesics. The preceding issue causes difficulties with magnetic horospheres as limit spheres, so we define magnetic \emph{orthospheres}  (\cref{def:F-F-t-magnetic-busemann-function}) and show that these are $C^2$.
    \end{description}

\subsection{Organization of this article}
We now discuss the main results and their interrelations, as well as the organization of this article. 

For a magnetic geodesic flow, we define and study the \m\ boundary $\M_{\mu}(\infty)$ of the universal cover. For negative magnetic curvature, the magnetic flow is Anosov and, in fact, a quasigeodesic flow \cite{Grognet}. Thus having bounded distance from a geodesic, a magnetic orbit (\cref{def:angle-for-pqtilde-end-points}) therefore has the same (Riemannian) endpoints on the (Riemannian) boundary $\Minfty$ (\cref{def:geo-boundary-point-asymp-class}) as that geodesic. 

This is not so if one only assumes nonpositive \m\ curvature, as made most starkly clear by the horocycle flow (\cref{Halfplanechords}). Therefore, we instead 
define the \m\ boundary $\M_{\mu}(\infty)$ from scratch as asymptote classes of \m\ geodesics and then construct a homeomorphism to $\Minfty$ (\cref{prop:boundary-homeomorphic}).
In the process, we note that \m\ Jacobi fields are stable if and only if the corresponding \m\ geodesics are asymptotic  (\cref{prop:stable-mag-jf-same-endpoint}). 
One application is a characterization of magnetic flatness by the existence of a magnetic geodesic with a single endpoint (\cref{prop:h2-has-no-closed-orbit}). (Thus, any orbit of a hyperbolic magnetic flow has two distinct endpoints on $\M_\mu(\infty)$.) 
On the other hand, we adapt the arguments in \cite{visibility} for geodesic flows to prove \m\ visibility (\cref{thm:ptildeqtilde-magnetic-visibility-property}; any two distinct points on $\M_\mu(\infty)$ are connected by a \m\ geodesic) and deduce topological transitivity of magnetic flows (\cref{thm:topologically-transitivity}) as well as orbit-equivalence to the geodesic flow of the underlying metric (\cref{thm:orbit-equivalence}).

The question of uniqueness of the connecting magnetic geodesic in the magnetic visibility property, which may fail for nonpositive magnetic curvature, leads to magnetic flat strips in \cref{section:mag-axes-and-flat-strip}; this mirrors geodesic flows, where flat strips are isometrically embedded Euclidean strips and associated with axial isometries (\cref{prop:magnetic-flat-strip-is-connected,lemma:magnetic-flat-strip-parallel}; in either context, the existence of flat strips is incompatible with uniform hyperbolicity and expansivity). Axial isometries come with magnetic axes---(periodic) magnetic geodesics preserved by the respective isometry---and this implies density of periodic vectors (\cref{prop:periodic-orbits-vectors}), hence that every vector is nonwandering  (\cref{cor:nw-is-TM}).

The results so far mirror those for geodesic flows. However, magnetic flat strips illustrate a distinction between geodesic and magnetic flows that has significant consequences. This is literally a twist. A (Riemannian) flat strip can be parametrized by a geodesic variation for which all corresponding Jacobi fields are parallel and orthogonal; in this sense, the geodesic flow acts on flat strips by translation. A magnetic flat strip can be parametrized by a magnetic variation for which all magnetic Jacobi fields are parallel, but this now only means that the orthogonal component is constant, while the tangential component is linear in time (see \eqref{eqn:magJac1}). Accordingly, the magnetic flow acts on a magnetic flat strip by linear twist maps (or shears). This is consonant with the terminology ``twisted geodesic flow''; see \cref{rmk:shearing-effect} for more details. The consequences of this are explored in \cref{SHorospheres}; they primarily affect attempts to build a theory of magnetic horospheres modeled on that for geodesic flows and the strong (un)stable subbundles (\cref{rmk:obstacle-sing-strong-subspaces}).

But first, we see how this affects expansivity (\cref{ssection:expansivity}). As is the case with geodesic flow, expansivity fails (only) in the presence of flat strips (\cref{prop:expansivity-of-mag-flows}). Moreover, a magnetic flow is orbit-equivalent to the geodesic flow of the underlying metric if and only if there are no magnetic flat strips (\cref{thm:orbit-equivalence}).

However, kinematic-expansivity, a weaker and possibly more natural notion than expansivity (points, rather than orbits, separate) holds even in the presence of magnetic flat strips because the twisting separates points (see \cref{thm:kinematic-expansivity}). This implies entropy-expansivity and hence the existence of equilibrium states (\cref{prop:mag-flow-is-h-expansive,thm:existence-of-es-equidistribution}). 

\cref{SHorospheres} begins with the most basic consequence of the twisting: the definition of a magnetic distance function based on lengths of connecting magnetic orbit (\cref{def:mag-distance}) does not yield a magnetic triangle inequality. (Indeed, that distance function is not even symmetric---this is to be expected from irreversibility of the magnetic flow and proved in \cref{cor:td-is-not-metric}.) In fact, we can describe the ``intermediate'' points for which the triangle inequality fails (\cref{lemma:partial-magnetic-triangle-inequality,lemma:magnetic-triangle-inequality-fails}).

Geometrically, the problem affecting the triangle inequality is that magnetic spheres are not orthogonal to radial magnetic orbits; \cref{fig:magtriineqfailed} illustrates the connection. As noted before, the two equivalent descriptions of horospheres (as limit spheres or orthogonal curves) disaggregate for magnetic flows, and indeed, there are difficulties with the notion of Busemann functions associated with the aforementioned distance (see \cref{rmk:obstacles-defining-magnetic-busemann-in-obvious-way}). 

Accordingly, we introduce magnetic orthospheres (orthogonal curves) and B-functions (see \cref{def:F-F-t-magnetic-busemann-function,rmk:name-of-B-function}) and connect these to the ``magnetic spray'':
magnetic orthospheres are $C^2$, and the geodesic curvature of orthospheres at $\pi v$ reflects the hyperbolicity of the magnetic flow at $v$ (\cref{prop:c2-regularity-of-magnetic-horosphere,rmk:perspective-on-horospheres}). This produces strong (un)stable spaces on the regular set (\cref{prop:strong_stable_spaces} and \cref{def:strong-stable-space}).

\subsection{Magnetic intensity versus speed of the flow}
As is the case with geodesic flows, magnetic flows conserve energy, i.e., they move at constant speed (because the magnetic force is perpendicular to the direction of motion). In the case of geodesic flows, this is the reason they are usually studied on the unit tangent bundle---because the geodesics are the same regardless of speed, one may as well consider unit speed. We do the same for magnetic flows. It is well to note, however, that magnetic flows at different speeds do not trace out the same curves on the manifold. A fast trajectory is less curved by a given magnetic field. Accordingly, while we consider different possible magnetic intensities and focus on the one at the edge of hyperbolicity, we could instead consider a given magnetic field and choose various speeds: fast motion is uniformly hyperbolic, and extremely fast motion will look very much like the geodesic flow. That is, the low-magnetic-intensity limit corresponds to the high-speed limit and vice versa. See \cref{rmk:observation} and \cite[Remark 2.8]{PeyerimhoffSiburg} for more details. 

We will restrict the study of magnetic flows to unit speed. 
Note that \m\ flows are Anosov (uniform hyperbolic) when $\mu$ is small by the stability of Anosov flows since the 0-magnetic flow is Anosov. 
We are interested in the threshold magnetic intensity
where hyperbolicity is no longer uniform.

\subsection{Some prior works on geodesic and magnetic flows}
A century ago, geodesic flows of negatively curved manifolds began to play a central role in what was to become smooth ergodic theory. The Boltzmann ergodic hypothesis from statistical mechanics and Poincar\'e's discovery of dynamical complexity had prompted a search for ergodic mechanical (i.e., Hamiltonian) systems, and Artin produced one in the form of the geodesic flow of the modular surface, even before ergodicity was a fully developed concept \cite{Artin}. This topic was developed extensively, but relied on the algebraic nature of these systems (Artin reduced the question to one about continued fractions) until Hopf found a fundamentally new mechanism for ergodicity \cite{Hopf}. However, its full application became possible only when Anosov axiomatized the dynamical properties produced by negative curvature and addressed the central technical point (absolute continuity) that had eluded Hopf \cite{Anosov}. Flows for which the entire underlying manifold is a uniformly hyperbolic set have since been known as Anosov flows. And geodesic flows of closed manifolds with negative sectional curvature have retained their status as prototypical Anosov flows, which exhibit structural stability, a complicated orbit structure, and ``chaotic'' dynamics.

Since the 1970s, the scope of these investigations has been widened by studying on one hand, which conditions beyond negative sectional curvature suffice for the Anosov property, and on the other hand, which features of hyperbolic flows persist when the hyperbolicity is not uniform. Both strands center on geodesic flows of manifolds with nonpositive curvature (and slightly beyond) \cite{Eberlein-II-1973, visibility,Pesin1,Pesin2, HI, Ballmann}. While Pesin theory provides a foundation, the presence of useful geometric structures is central to these investigations. 

Notably, the interplay between deck transformations and the geodesic (and now, magnetic) flow on the universal cover is essential, and this involves the ideal boundary, that is, equivalence classes of asymptotic (and now, magnetic) geodesics.

\cref{ssection:geometry-geoflow} contains some of the commonalities and differences between having negative versus nonpositive curvature. Here are a few highlights. First, the definition of the ideal boundary extends from negative to nonpositive sectional curvature, and the action of the deck group on the universal cover $\M$ of $\Sigma$ illuminates large-scale behaviors of geodesics on $\M$, such as, the visibility property, topological transitivity, topological mixing, regularity of horospherical foliations and so on; see \cite{visibility, Eberlein-II-1973, Ballmann, HI}. Ergodic theory provides further insights, such as, the distribution and density of periodic orbits as well as the existence and uniqueness of equilibrium states; see \cite{Eberlein-II-1973, Knieper, BCFT}. Notably, the latter transcends the question of ergodicity of the Liouville measure and instead aims to select from the wealth of invariant probability measures of particular dynamical interest.

Complementary to the exploration beyond negative curvature, there is the aforementioned generalization of geodesic flows to magnetic flows; see \eqref{eqn:general-mag} below. Instead of following geodesics, that is, curves with geodesic curvature 0, as orbits of the geodesic flow do, a unit-speed \m\ flow traces curves with geodesic curvature $\mu\in\re$. In physics, a \m\ flow models the motion of a unit-charge particle in the presence of a magnetic field ``orthogonal'' to the surface and with \emph{magnetic intensity} $\mu$. (Here $\mu$ can be a scalar function on the unit tangent bundle of the surface, which models a magnetic field with nonconstant magnetic intensity at different points; see \cref{ssection:magflow} for details.) For now we note that the Lorentz force is orthogonal to the velocity and hence does not affect the speed. (Magnetic flows can be defined for manifolds of higher dimension, but this work confines itself to surfaces.) Besides the geodesic flow (zero magnetic intensity), magnetic flows include the horocycle flow as the magnetic flow $\f$ for the closed surface with constant negative curvature $-\mu^2$. 

We mention some salient works from the prior literature about magnetic flows. There may be earlier occurrences in the physics literature, but we take the sole joint mathematics paper of Anosov and Sinai as the starting point \cite[p.~112]{AnosovSinai}. Its main concern is ergodicity of Anosov systems, and it is the most readable account from the time period (these are lecture notes related to the fabled Khumsan summer school on ergodic theory \cite{Humsan}) but magnetic and Finsler geodesic flows appear as examples. Those expository notes are related to the fact that the study of magnetic flows from a dynamical and geometric point of view had been started by Anosov: by the stability of uniform hyperbolicity, magnetic flows $\f$ for a closed surface with nonconstant negative curvature are uniformly hyperbolic when $\mu$ is sufficiently small. Moreover, his book \cite{Anosov} contained the computation of magnetic Jacobi fields of magnetic flows with constant magnetic field for surfaces with nonconstant curvature  and showed that uniform hyperbolicity follows from having negative magnetic curvature, i.e., $\ref{H4}\Rightarrow\ref{HA}$ (see also \cref{REMNegCurvImpliesAnosov}, and \ref{H4} has since shown not to be necessary \cite[\S 7]{Burns-Paternain}.) The name ``magnetic curvature" comes from the fact that $\kmu$ plays the same role in the \m\ Jacobi equation \eqref{def:eqn:orthogonal-component-of magnetic-Jacobi-field} as the curvature $K$ does in the Jacobi equation for geodesic flows. Magnetic flows, magnetic intensity, and magnetic curvature have all been defined in higher dimension later on.

There are articles from the physics point of view, investigating magnetic flows of constant magnetic field for surfaces with constant curvature, such as \cite{ComtetHouston,Comtet,Sunada}.

Starting in the 1990s, investigations moved on from purely algebraic systems. \cite{Paternain-Paternain} consider manifolds of higher dimension with nonconstant curvature. They investigate magnetic flows of nonconstant magnetic fields and their magnetic Jacobi fields. They parametrize a class of magnetic flows $\f$ with parameter $\mu$ as perturbations of an Anosov geodesic flow ($\mu=0$), and show that the topological entropy strictly decreases with $\mu$ (he entropy is a strictly decreasing function on $[0,\mu_c]$ \cite{Paternain-Paternain-1997}), and there is a maximal open interval $I_c=(-\mu_c,\mu_c)$ such that for $\mu\in I_c$, the $\mu$-magnetic flow is Anosov; the regularity of the strong stable and unstable subbundles is discussed in \cite{Paternain-regularity}. 

\cite{Adachi1} studies nonconstant magnetic fields for Hadamard surfaces with a view to the unboundedness of magnetic geodesics and their endpoints on $\Minfty$. For surfaces with negative magnetic curvature, the angle subtended by a magnetic geodesic segment and its geodesic chord (\cref{def:chord}) is uniformly less than $\pi$. This was strengthened to $\pi/2$ in \cite{Grognet} and appears in the proof of \cref{Lemma:angle_<_pi/2}.

\cite{Adachi2} studies magnetic flows with nonconstant magnetic field for K\"ahler manifolds and for Riemannian surfaces. Comparison theorems imply that the magnetic exponential map is a covering map for a Hadamard surface $\Sigma$ for constant magnetic fields and nonpositive magnetic curvature (hence \ref{H1}) \cite[Theorem 3]{Adachi2}. This leads to the existence of magnetic geodesics between a given point on $\M$ and a given point on $\Minfty$; see \cite[Corollary, p.~305]{Adachi2}. Furthermore, \ref{H4} yields a homeomorphism between $T^1_p\M$ and $\Minfty$ with respect to the cone topology (\cite[p.~306]{Adachi2}) for any $p\in \M$, magnetic visibility property (\cite[Theorem 4]{Adachi2}), and continuity of the magnetic Busemann functions (\cite[Remark, p.~308]{Adachi2}).

For nonconstant magnetic fields for compact Riemannian manifolds with negative sectional curvature, \cite{Gouda,Grognet} give sufficient conditions for the magnetic flow to be Anosov. Moreover, \cite{Grognet} gives restrictions on the norm of the magnetic field under which the magnetic flow is a quasigeodesic flow
(\cite[Proposition 3.1]{Grognet}), and is orbit-equivalent to the underlying geodesic flow (\cite[Theorem 3.1]{Grognet}), respectively. For a surface with negative curvature, Grognet discusses the class of magnetic flows that have the same marked length spectrum as the geodesic flow (\cite[Theorem 7.3]{Grognet}). \cite{Grognet-entropy} estimates topological entropy and gives a lower bound for the Liouville entropy under the sufficient condition from \cite{Grognet} for magnetic flows to be Anosov. Moreover, for surfaces, there is an upper bound for the Liouville entropy and rigidity when the topological and Liouville entropies coincide.

\cite{Wojtkowski} further extends the sufficient condition of \cite{Gouda, Grognet} for \ref{HA} to Gaussian thermostats, a generalization of magnetic flows (in arbitrary dimension), and \cite{MRS24} gives necessary and sufficient conditions for a magnetic flow of a surface to be Anosov.

\cite{Burns-Paternain} show that the Ma\~n\'e critical value $c(g,\Omega)$ of given magnetic structure $(g,\Omega)$ (\cref{def:mag flow}), the infimum of energies $k$ for which $A_{L+k}(\ga)\geq 0$ for every absolutely continuous closed curve $\ga\colon [a,b]\to \widetilde{M}$ where $L$ is the Lagrangian and $A_L(\ga)\dfn \int_a^b L(\ga(t), \dot{\ga}(t)) \, dt$, is related to the Anosov property: if the \m\ flow of a manifold with magnetic structure $(g,\Omega)$ is Anosov, then $\mu^2<\frac{1}{2c(g,\Omega)}$. They relate volume entropy (the exponential growth rate of the average volume of certain balls) to topological entropy and construct surfaces of negative curvature to show that as $\mu$ increases, the \m\ geodesic flow can exit and reenter the set of Anosov magnetic flows arbitrarily many times.

\cite{PaternainRigidity} proves rigidity results for magnetic flows with nonconstant magnetic field for surfaces and gives characterizations of horocycle flows. In particular, a uniquely ergodic Ma\~n\'e-critical magnetic flow is a horocycle flow.

\cite{PeyerimhoffSiburg} generalizes the quasigeodesic result of \cite{Grognet}, which requires pinched negative curvature, to magnetic flows with energy levels above the Ma\~n\'e critical value: minimal magnetic geodesics are Riemannian $(A,0)$-quasigeodesics with $A \to 1$ as the energy tends to infinity, which on negatively curved manifolds implies that magnetic geodesics lie in tubes around Riemannian geodesics.

\cite{BurnsMatveev} proves the following rigidity theorem: if two magnetic structures $(g,\Omega)$ and $(\bar g, \bar\Omega)$ on the same manifold share the same magnetic geodesics up to reparametrization at corresponding energy levels $h$ and $\bar h$, then either both are rescalings of one another, or $\Omega = 0 = \bar\Omega$ and $g$, $\bar g$ are metrics with the same geodesics. Moreover, if some magnetic geodesic of $(g, \Omega)$ has tangent vectors dense in its energy level, then the only magnetic structures sharing the same magnetic geodesics as $(g, \Omega)$ are its rescalings.

\cite{GomesRuggiero} establishes smooth rigidity: if the magnetic flow for a surface preserves a highly smooth codimension-one foliation, then it is the magnetic flow of a constant magnetic field for a surface with constant nonpositive curvature.

\cite{GLP25} proves  marked length spectrum rigidity for Anosov surfaces: two Riemannian metrics on a closed oriented surface of genus at least two with Anosov geodesic flow and sharing the same marked length spectrum are isometric via an isometry isotopic to the identity. \cite{MarkedLength} (based on \cite{Echevarria}) extends this to magnetic flows: for a closed, connected, oriented surface carrying an Anosov magnetic structure, any volume-preserving conjugacy isotopic to the identity, under which the 2-forms of magnetic structures lie in the same cohomology class, forces the underlying metrics to be isometric. In higher dimension, a similar but local rigidity result is in \cite{beaufort2026markedmagneticactionrigidity}, as is a global one in a conformal class. For more about the related literature, we recommend the bibliography of \cite{beaufort2026tensortomographyframeflow}.

Two threads of investigation we mentioned---extending the sectional curvature assumption of the manifolds from negative to nonpositive, and generalizing geodesic flows---involve, respectively, the study of geodesic flows beyond strictly negative curvature, and the study of nonuniformly hyperbolic flows beyond geodesic flows.
Thus, the present work goes beyond both geodesic flows and uniform hyperbolicity. We study \emph{magnetic flows} for a surface with negative curvature and with \emph{nonpositive} magnetic curvature. Significant properties of geodesic flows in negative curvature also hold for such magnetic flows (topological transitivity, kinematic-expansivity, and the existence of equilibrium states), but some other straightforward adaptations of arguments are precluded by the combination of the magnetic twist and the nonuniformity of hyperbolicity. Notably, a natural theory of magnetic horospheres seems elusive for now.

\section{Background}\label{section:background}
In this section we develop background on magnetic flows. We denote by $M$ an oriented, connected and closed Riemannian manifold of arbitrary dimension, and we instead use $\Sigma$ when the manifold is 2-dimensional. Let $\pi\colon TM\to M$ be the footpoint projection.
\subsection{Magnetic flows}\label{ssection:magflow}
A \emph{magnetic structure} on $M$ is a pair $(g,\Omega)$, where $g$ is a Riemannian metric and $\Omega$ is a closed 2-form on $M$. Let $\omega_g$ be the symplectic form on $TM$ obtained by pulling back the canonical symplectic form on $T^*M$ via the Riemannian metric $g$. Consider the symplectic form \[\omega_{mag}\dfn\omega_g+\pi^*\Omega\] and the Hamiltonian function \[E(v)=\frac{1}{2}g(v,v).\]
The magnetic flow of  $(g,\Omega)$ is the Hamiltonian flow of $E$ with respect to $\omega_{mag}$, see \cite[p.~282]{Burns-Paternain}.
This models the motion of a particle of unit mass and charge in a magnetic field, whose Lorentz force $Y\colon TM\to TM$ is defined by 
\begin{equation}\label{eqn:lorentz}
    \Omega_x(u,v)=g_x(Y_x(u),v)
\end{equation} for all $x\in M$ and all $u, v\in TM$. In other words, the curve \[t\mapsto (\ga(t),\dot\ga(t))\in TM\]
is an orbit of this flow if and only if 
\begin{equation}\label{eqn:general-mag}
    \frac{D\dot\ga}{dt}=Y_\ga(\dot\ga),
\end{equation}
where $D$ is the covariant derivative of $g$; $\ga$ will be called a \emph{magnetic geodesic} of $(g,\Omega)$. The flow $\{f_t\}_{t\in\re}$ with $f_t\colon TM\to TM$, $\dot\ga(t_0)\mapsto \dot\ga(t+t_0)$ for any $t_0\in\re$ is the \emph{magnetic flow} of $(g,\Omega)$. The magnetic flow of $(g,0)$ is the geodesic flow.
\begin{remark}\label{rmk:observation}Magnetic flows preserve energy (hence speed) and volume.
\begin{enumerate}
    \item \cref{eqn:general-mag} shows that magnetic geodesics have constant energy:
\[\frac{d}{dt}\frac{1}{2}|\dot\ga|^2=\langle \frac{D\dot\ga}{dt},\dot\ga \rangle=\langle Y_\ga(\dot\ga),\dot\ga \rangle=g_x(Y_\ga(\dot\ga), \dot\ga) = \Omega_x(\dot\ga,\dot\ga)=0.\]
The (hence invariant) energy level $L=1/2$ is  $T^1M$.

\item \label{item:obs} Fix the energy level $L$. Let $T={\dot\ga}/{|\dot\ga|}$ be the unit tangent vector to a magnetic geodesic $\ga$ of $(g,\Omega)$. 
The geodesic curvature of $\ga$ satisfies
\[|k_{g,\Omega}(\ga)|= |\nabla_T T|=\frac{1}{|\dot\ga|^2}\frac{D\dot\ga}{dt}= \frac{1}{|\dot\ga|^2}|Y_\ga(\dot\ga)|.\]
From \eqref{eqn:lorentz}, scaling $\Omega$ by $\mu$ rescales the Lorentz force $Y$ by $\mu$. Therefore,
\[|k_{g,\mu\Omega}(\gamma)| = |\mu\, k_{g,\Omega}(\gamma)|.\]
In particular, geodesics arise as the $\mu\to 0$ limit of unit-speed magnetic geodesics of $(g,\mu\Omega)$ and also as the $L\to\infty$ (or $|\dot\ga|\to\infty$) limit of (reparametrized) magnetic geodesics of $(g,\Omega)$. 
Furthermore, studying the magnetic flow of $(g,\Omega)$ at the energy level $L = {1}/{2\mu^2}$, i.e., at speed $|\dot \ga|=1/\mu$, is equivalent to working with the magnetic flow of $(g,\mu\Omega)$ restricted to $T^1M$ (unit-speed, hence with energy level $L=1/2$). Accordingly, we restrict to unit speed henceforth.

\item Magnetic flows are Hamiltonian, so  by the Liouville theorem \cite[Proposition 5.5.12]{Katok-Hasselblatt} they preserve a natural normalized volume, the \emph{magnetic Liouville measure} $\mathrm{vol}_{\widetilde\Omega}$, induced by the restriction to $T^1\Sigma$ of the symplectic form $\widetilde\Omega = d\lambda + \pi^*\Omega$ on $T\Sigma$, where $\lambda$ is the Liouville 1-form and $\pi\colon T^1\Sigma\to \Sigma$ is the footpoint projection. In fact, this coincides with the Liouville measure $\mathrm{vol}_{d\lambda}$ for the geodesic flow. To see this, note that $\pi^*\Omega$ is a 2-form on the $3$-dimensional manifold $T^1\Sigma$ that vanishes whenever one of its arguments is tangent to a fiber of $\pi$, that is, tangent to the unit circle $\pi^{-1}(x)=T^1_x\Sigma$ for some $x \in \Sigma$; see \cite[Chapter ~1]{Paternainbook}. Since $T^1\Sigma$ is 3-dimensional with $1$-dimensional fibers, any $3$-form on $T^1\Sigma$ must involve the fiber direction, and hence $\pi^*\Omega$ contributes nothing to the volume form. Therefore,
\[
\mathrm{vol}_{\widetilde\Omega}\big|_{T^1\Sigma} = \mathrm{vol}_{d\lambda}\big|_{T^1\Sigma}.
\]
\end{enumerate}
\end{remark}


\begin{remark}\label{remark:notation_mag_geo}
Henceforth, we consider unit-speed \m\ flows for $(g,\Omega_g)$. Only a few basic notions and results will be presented without assuming $h\equiv1$. From \cref{sec:hyp-mag-flow} onward, we will take $h\equiv1$.
\end{remark}

In the surface case, let $(p,v)\in T^1\Sigma$ and $iv$ be the vector in $T_{p}^1\Sigma$ such that $\{v,iv\}$ is a positively oriented orthonormal basis of $T_{p}^1\Sigma$. The area form $\Omega_g$ is given by 
\begin{equation}\label{eqn:2-norm-curvature}
    \Omega_g(u,v)=g(iu,v).
\end{equation}
Since any closed 2-form $\Omega$ can be written as 
\begin{equation}\label{eqn:area_form}
    \Omega=h\Omega_g 
\end{equation} for some smooth function $h\colon \Sigma\to\re$, the Lorentz force $Y$ associated with $\Omega$ is given by
\begin{equation}\label{eq:lorentz-force}
Y_p(v)=h(p)iv.
\end{equation} 


\begin{definition}[$\mu$-magnetic geodesic and flow]\label{def:mag flow}
For $\mu\in\re$, a curve $\ga$ on $M$ 
is said to be a \m\  geodesic if it is a magnetic geodesic of the magnetic structure $(g,\mu\Omega_g)$, i.e.,\begin{equation}\label{eqn:general_lor}
    \frac{D\dot\ga}{dt}=\mu Y_\ga(\dot\ga),
\end{equation} (see \eqref{eqn:general-mag}). For a surface, $Y_\ga(\dot\ga)= i\dot\ga$, so a \m\ geodesic $\ga$ satisfies 
\begin{equation}\label{eqn:mag-geo-flow}
    \frac{D\dot\ga}{dt}=\mu i\dot\ga.
\end{equation}
For $v \in T_p^1M$, let $\gamma_v$ denote the $\mu$-magnetic geodesic (\cref{def:mag flow}) with ($\gamma_v(0) = p$ and) $\dot\gamma_v(0) = v$. Likewise, $g_v$ is the geodesic with $\dot g_v(0)=v$.

The flow $\f\colon\re\times T^1M \rightarrow T^1M$ defined by $\f(t,v)\dfn\f_t(v)\dfn\big(\ga_v(t),\dot{\ga}_v(t)\big)$ 
is called the \emph{$\mu$-magnetic flow of $M$}. 

 \end{definition}

\begin{remark}[Basic  observations]\label{rmk:general-behaviors-easy-observation}
    \begin{enumerate} 
        \item \cite[Remark 1.1.30]{Fisher-Hasselblatt} $\f$ is reversible, that is, $-f^0_{-t}(-v)=f^0_t(v)$, iff $\mu=0$. 
        \item In the surface case,  \label{item:mu-mu-relation}
        \begin{enumerate}
            \item $\ga$ is a \m\  geodesic on $\Sigma$ if and only if $\ga$ has constant geodesic curvature $\mu$. The \m\ flow models the motion of a particle on $\Sigma$ with unit charge and mass in a magnetic field that is orthogonal to $\Sigma$ and of constant \emph{magnetic intensity} $\mu$, or equivalently, the motion of a unit-mass particle with charge $\mu$ in a magnetic field that is orthogonal to $\Sigma$ and of constant \emph{magnetic intensity} $1$. 
            
            If $\mu>0$ (resp.\ $\mu<0$), then the particle drifts to the left (resp.\ right) of the tangent direction; and $\mu=0$ implies that the trajectories of particles are geodesics.

            Put differently, if $\mu>0$ (resp.\ $\mu<0$), then locally the magnetic geodesic tangent to a given vector $v$ lies to the left (resp.\ right) of the geodesic tangent to $v$. (This can also be seen from \eqref{eqn:mag-geo-flow}.)
            
\item\label{LeftRightTangency}Thus, a \m\ geodesic bounds a (strictly geodesically convex) intersection of (left or right) half-spaces defined by its tangent geodesics. So, unless $\mu=0$, a \m\ geodesic intersects any geodesic at most twice, and for a \m\ geodesic segment, if $\mu>0$ (resp.\ $\mu<0$), then its geodesic chords other than tangent lines (\cref{def:chord}) lie to its left (resp.\ right), with both intersections transverse.
        \item\label{ReverseVersion}$\ga$ is a \m\ geodesic if and only if $\tilde \ga\colon t\mapsto\ga(-t)$ is a $(-\mu)$-magnetic geodesic on $\Sigma$ with magnetic structure $(g,\Omega)$, or equivalently, $\tilde \ga$ is a \m\ geodesic on $\Sigma$ with magnetic structure $(g,-\Omega)$.
        \end{enumerate}
    \end{enumerate}
\end{remark}

\subsection{Magnetic Jacobi fields and invariant subbundles}
\subsubsection{Magnetic Jacobi fields}\label{ssection:magnetic_jacobi_field}
We now introduce two equivalent definitions of magnetic Jacobi fields, one as the infinitesimal change of a magnetic variation, a smooth family of magnetic geodesics, and the other as the solution of a second-order differential equation. \cite{Paternain-Paternain} shows that these are equivalent. 

\begin{definition}[Magnetic variation and magnetic Jacobi field] \label{def:magnetic-jacobi-variation-field}
Given a magnetic structure $(g,\Omega)$ on $M$, a magnetic variation $\Gamma$ is a smooth mapping \[(s,t)\in(a,b)\times \re\to \Gamma(s,t)\in M\]such that 
$\ga_s\dfn \Gamma(s,\cdot)$ is a magnetic geodesic of $(g,\Omega)$ for any $s\in (a,b)$, and 
$J_{s_0}(\cdot)\coloneqq \frac{\partial \Gamma (s,\cdot)}{\partial s}|_{s=s_0}$ is called the magnetic  Jacobi field corresponding to $\Gamma$ along $\ga_{s_0}$.
For $\Omega=\mu\Omega_g$ we  use the terms \emph{\m\ variation} and \emph{\m\ Jacobi field}.
\end{definition}
Since here $\Gamma$ is smooth, so is $(s,t)\mapsto J_s(t)$.

\begin{proposition}[{\cite[Equation(3)]{Paternain-Paternain}}]
Given a magnetic structure $(g,\Omega)$ on $M$ and $\mu\in\re$, let $J$ be a magnetic Jacobi field along a magnetic geodesic $\ga=\gav$ of the magnetic structure $(g,\mu\Omega)$. Then $J$ satisfies the following \emph{magnetic Jacobi equation}:
\begin{equation}\label{eqn:magJF}
\ddot J+R(\dot\ga,J)\dot\ga-\mu[Y(\dot J)+(\nabla_{J} Y)(\dot \ga)]=0,
\end{equation}
where $R$ is the curvature tensor of $g$. 
 
Furthermore, let $f$ be the magnetic flow and $\xi\in T_vTM$ such that $J(t)=d_v(\pi\circ \fmu_t)\xi$ (note that the energy level $L$ is not necessarily $1/2$). 
By identifying $T_vTM$ with $T_{\pi v}M\oplus T_{\pi v}M$ using horizontal and vertical subbundles, we have 
\begin{equation}\label{eqn:differential-of-fmu}
    d\f_t(\xi)=(J(t),\dot J(t)).
\end{equation}
We call $J$ the \emph{magnetic Jacobi field} along $\ga$ with initial condition $(J(0),\dot J(0))$.

If $\Sigma$ is a surface with Gauss curvature $K$ and $h\equiv1$, then \cite[\S 5]{Paternain-Paternain} gives
\begin{equation}\label{def:eqn:orthogonal-component-of magnetic-Jacobi-field}
    J(t)=x(t)\dot{\ga}(t)+y(t)i\dot{\ga}(t)
\end{equation}
with
\begin{align}
&\ddot{x}=\mu \dot{y},   \label{eqn:general-magJac11}\\
&\ddot{y}+\mu\dot x+Ky=0.
\label{eqn:general-magJac22}
\end{align}
\end{proposition}
\begin{remark}\label{rmk:tangential-magnetic-JF}
$J(t)=C\dot\ga(t)$ is a tangential magnetic Jacobi field (i.e., $y\equiv0$); it arises from the variation by reparametrizations $t\mapsto t+Cs$. 
This is it, except  for geodesic flows:
\end{remark}

\begin{proposition}\label{prop:dim-of-tang-JF}
In the context of \eqref{eqn:general-magJac11} and \eqref{eqn:general-magJac22}, the space $V$ of tangential magnetic  Jacobi fields along a \m\ geodesic $\ga$ satisfies
\[\dim V=\begin{cases}
    2&\text{if }\mu=0,\\1&\text{if }\mu\neq0.
\end{cases}\]    
\end{proposition}
\begin{proof}
If $\mu=0$, then  $x$ is linear in $t$ by \eqref{eqn:general-magJac11}, so $\{\dot\ga(t),t\dot\ga(t)\}$ spans $V$.
If $\mu\neq0$ and $y\equiv 0$, then \eqref{eqn:general-magJac22} gives $\dot x\equiv0$, so $x$ is constant, and $\dot\ga(t)$ spans $V$.
\end{proof}


We briefly return to the generality of \eqref{eqn:area_form} 
with the Lorentz force in \eqref{eq:lorentz-force}.
Given $v\in T^1\Sigma$, $\mu\in\re$,  and a magnetic Jacobi field $J$ along $\ga=\gav$ of the magnetic structure $(g,\mu h\Omega_g)$, 
by \eqref{eqn:area_form}, \eqref{eq:lorentz-force}, \eqref{eqn:general_lor}, \eqref{eqn:magJF} and \eqref{def:eqn:orthogonal-component-of magnetic-Jacobi-field},  we have
\begin{align}
&\dot{x}=\mu h(\ga)\,y,   \label{eqn:general-magJac1}\\
&\ddot{y}+\left[K(\ga)-\mu\langle\nabla h(\ga),i\dot\ga\rangle+\mu^2h^2(\ga)\right]y=0,
\label{eqn:general-magJac2}
\end{align}
where $K$ is the Gaussian curvature; see \cite[\S 22]{Anosov} and \cite[\S 6]{Burns-Paternain} for details.

\begin{definition}[Magnetic curvature]\label{def:y-of-J-and-general-magnetic-curvature}
Given a surface $\Sigma$ with magnetic structure $(g,h\Omega_g)$, $\mu\in\re$, and a magnetic Jacobi field $J$ along a magnetic geodesic $\ga$ of the magnetic structure $(g,\mu h\Omega_g)$, $y$ in \eqref{def:eqn:orthogonal-component-of magnetic-Jacobi-field} (and therefore satisfying \eqref{eqn:general-magJac2}) is called the \emph{orthogonal component of $J$}, and 
    \begin{equation}\label{eqn:general-magnetic-curvature}
        K_\mu(p,v)\dfn K(p)-\mu\langle\nabla h(p),iv\rangle+\mu^2h^2(p)
    \end{equation} is called the \emph{\m\ curvature of $(p,v)\in\sm$} with respect to $(g,h\Omega_g)$. 
    
    With our standing convention $h\equiv1$, this is $K_\mu(p)=K(p)+\mu^2$, and \eqref{eqn:general-magJac2} becomes \eqref{eqn:magJac2}.
\end{definition}
The term ``magnetic curvature'' comes from the fact that replacing $\kmu$ by the sectional curvature $K$, \eqref{eqn:general-magJac2} is satisfied by the orthogonal Jacobi field for geodesic flows. 

From now, we focus on  surfaces with $h\equiv1$ (except for some background results), and any occurrence of the term ``\m'' will emphasize this. 
\begin{definition}[(Un)stable, parallel magnetic Jacobi field and magnetic variation]\label{def:stable-unstable-magJF}
A \m\ Jacobi field $J$ is said to be \emph{stable} (resp.\ \emph{unstable}) if its \emph{orthogonal component} $y(t)$ in \eqref{eqn:magJac2} is 
bounded on $[0,\infty)$ (resp. $[-\infty,0]$). Denote by $\mathcal{J}^s(\ga)$ (resp.\ $\mathcal{J}^u(\ga)$) the linear subspace of stable (resp. unstable) magnetic Jacobi fields along a magnetic  geodesic $\ga$. $J$ is \emph{parallel at $t$} if $\dot y(t)=0$. 
$J$ is \emph{parallel} if it is parallel at all $t\in \re$ (i.e., $y$ is constant).


\end{definition}


The following result is immediate. 

\begin{lemma}\label{lemma:sing-vector-has-constant-y}
    Given a surface $\Sigma$ with magnetic structure $(g,h\Omega_g)$ and $\mu\in\re$, a magnetic  Jacobi field $J$ of $(g,\mu h\Omega_g)$ is parallel if and only if $y$ is constant, and in this case, $J$ is both stable and unstable.
\end{lemma}
\begin{definition}[Singular and regular vectors]\label{def:sing-reg-mag-vectors}
Given a surface $\Sigma$ and $\mu\in\re$, we say $v\in\sm$, or $\ga_v$, is \emph{singular}, or $v\in \Sing$, if there exists a nontangential parallel \m\  Jacobi field, that is, a \m\ Jacobi field with nonzero constant $y$ along the \m\  geodesic $\ga_v(t)$ for all $t\in \re$. Otherwise, $v$, or $\gav$ is said to be \emph{regular}, or $v\in\Reg$.
\end{definition}

Clearly, $\Sing$ is an obstruction to uniform hyperbolicity (\cref{prop:dim-of-subbundle}). It can also be characterized in different ways; see \cref{prop:reg-sing-dim-of-subbundle,lemma:sing-zero-0-magnetic-curvature,lemma:la-vanish-on-sing}.

\subsubsection{Invariant subbundles}\label{ssec:invsubbundle}

Given a surface $\Sigma$ and $\mu\in\re$, let $\f$ be the \m\ flow. 
Similar to the case of geodesic flows, there are three $\f$-invariant subbundles $E^u$, $E^s$, and $E^c$ of $TT^1\Sigma$, where $E^c$ is spanned by the vector field that generates $\f$, and 
\begin{align}
&E^s(v)=\{\xi\in T_vT^1\Sigma\mid J_\xi \text{ is nontangential and stable along } \ga_v\}, \label{eqn:stable-space} \\
&E^u(v)=\{\xi\in T_vT^1\Sigma\mid J_\xi \text{ is nontangential and unstable along } \ga_v\}. \label{eqn:unstable-space}
\end{align}
We write $E^{cu}=E^c\oplus E^u$ and $E^{cs}=E^c\oplus E^s$. \cref{def:sing-reg-mag-vectors} gives:
\begin{proposition} \label{prop:reg-sing-dim-of-subbundle}
  \begin{enumerate}
   \item $E^u$ and $E^s$ are both transverse to $E^c$.
   \item $\Reg=\{v\in T^1\Sigma\mid E^{cs}(v)\cap E^{cu}(v)= E^c(v)\}$.
   \item $\Sing=\{v\in T^1\Sigma \mid E^{cs}(v)\cap E^{cu}(v) \neq E^c(v)\}$.
    \end{enumerate}
\end{proposition}

\subsection{Hyperbolic magnetic flows}\label{sec:hyp-mag-flow}
In the rest of this work, we fix $h\equiv 1$ in \eqref{eqn:area_form} and assume \ref{H1} from \cref{DEFHypMag}.

Equations \eqref{eqn:general-magJac1}, \eqref{eqn:general-magJac2}, \eqref{eqn:general-magnetic-curvature} thus become 
\begin{align}
&\dot{x}=\mu y,   \label{eqn:magJac1}\\
&\ddot{y}+\big(K(\gav)+\mu^2\big)y=0, \label{eqn:magJac2}\\
&\kmu(p)\coloneqq \kmu(p,v)=K(p)+\mu^2,\label{semi-general-mag-curv}
\end{align}
the \emph{\m\  curvature}. \cref{eqn:magJac2,semi-general-mag-curv} give
\begin{equation}\label{MagneticJacobiEquation}
    \ddot{y}+\kmu(\gav)y=0.
\end{equation}
Differentiating \eqref{def:eqn:orthogonal-component-of magnetic-Jacobi-field} using \eqref{eqn:mag-geo-flow} and \eqref{eqn:magJac1} gives
\begin{equation}
    J'=(\mu x+\dot y)i\dot\ga. \label{eqn:J'}
\end{equation}

\begin{remark}[Convexity]\label{rmk:comment-on-mag-JE}
If \ref{H1} and $J(t)=x(t)\dot\ga(t)+y(t)i\dot\ga(t)$ is a \m\  Jacobi field along a \m\  geodesic $\ga(t)$, then
$y$ is a convex function by \eqref{MagneticJacobiEquation}. This is key to understanding the exponential divergence of magnetic geodesics as seen in \cref{remark:riccati-equation} and \cref{sec:mag-jf}, mirroring the behavior of geodesics on manifolds with nonpositive sectional curvature.
\end{remark}

\begin{definition}[Magnetic rank]\label{def:magnetic-rank-one}
The \emph{\m\  rank} $\rank (v)$ of $v\in\sm$ is the dimension of the space of parallel \m\  Jacobi fields along $\gav$. The \emph{\m\  rank} $\rank(\Sigma)$ of $\Sigma$ is the minimum rank over all vectors in $\sm$.
\end{definition}

\begin{remark}\label{rmk:singular-vector-has-higher-rank}
By \cref{def:sing-reg-mag-vectors}, $\Sing$ is the set of vectors with \m\  rank larger than 1, and $\Reg$ is the set of vectors  with \m\  rank 1. 
\end{remark}

\begin{proposition}\label{prop:magnetic-rank-one-implies-h2}
If \ref{H1}, then $\rank(\Sigma)=1$ if and only if there is a $p\in \Sigma$ with $\kmu(p)<0$. That is, for surfaces, \ref{H2} is equivalent to having both \ref{H1} and $\rank(\Sigma)=1$.
\end{proposition}

\begin{proof}
For $v\in T^1_p\Sigma$ consider a parallel \m\  Jacobi field along $\gav$. Then its orthogonal component $y$ is constant, so $\ddot y(0)\equiv\dot y(0)\equiv 0$, hence also $y(0)=0$ by \eqref{MagneticJacobiEquation} since $\kmu(p)<0$. Now, $y(0)=0=\dot y(0)$ implies $y\equiv0$ because \eqref{MagneticJacobiEquation} is linear and homogeneous, so the Jacobi field is tangential. 
The (contrapositive of the) other direction follows from \cref{rmk:singular-vector-has-higher-rank}.
\end{proof}

\begin{remark}[Twist property/shearing]\label{rmk:shearing-effect}
The twist property (or shearing effect) of the magnetic flow arises from the coupling in \eqref{eqn:magJac1}. Let $\Sigma$ be a Riemannian surface (with no curvature assumption yet). Consider a \m\ variation $\Gamma\colon [0,C]\times \re\to\Sigma$ such that the corresponding \m\ Jacobi field $J_c(t) = x_c(t)\dot\ga_c(t) + y_c(t)i\dot\ga_c(t)$ along $\ga_c(t)\dfn \Gamma(c,t)$ has $y_c \equiv 1$ for any $c\in[0,C]$. Equivalently, $\Gamma$ parametrizes a \m\ flat strip; see \cref{def:magnetic-flat-strip}. Then \eqref{eqn:magJac1} gives $\dot x_c \equiv \mu$ and $x_c(t)= \mu t + x_c(0)$ for any $c\in[0,C]$.

If $\mu\neq0$, in the $(\dot\ga_c,i\dot\ga_c)$-coordinates of the \m\ Jacobi field along $\ga_c$, the derivative of the time-$t$ \m\ flow is given by the shear matrix \[ D\f_t = \begin{pmatrix} 1 & \mu t \\ 0 & 1 \end{pmatrix}, \] so $(x_c(0),y_c(0)) \mapsto (x_c(0)+\mu t\, y_c(0),\, y_c(0))$: the \m\ flow acts on the \m\ flat strip by a twist (shear) map whose shear rate is $\mu t$, producing a shearing effect along the flow direction. A notable instance of this phenomenon is the \m\ flow of $\Sigma$ with constant curvature $-\mu^2$: the \m\ flow coincides with the horocycle flow of $\Sigma$, which moves at unit speed with respect to the hyperbolic metric, while its Euclidean speed in the upper half-plane varies with height. See \cite[Chapter I.2.1.c]{Fisher-Hasselblatt} for more details.


If $\mu = 0$ and $\Sigma$ has nonpositive curvature, the same computation gives $\dot x_c \equiv 0$, so in the same coordinates \[ Df_t^{0} = \begin{pmatrix} 1 & 0 \\ 0 & 1 \end{pmatrix} \] is the identity: the geodesic flow acts on the flat strip by translation of the base-point, with no shear on the transverse $(x,y)$-coordinates. Here the variation can be reparametrized to produce orthogonal Jacobi fields with $x_c\equiv 0$ for all $c\in[0,C]$, and kinematic-expansivity fails (an infinite Bowen ball \eqref{def:bowen_ball} for the geodesic flow might correspond to a small flat rectangle); see \cref{def:k-expansivity}.

    
This twist for $\mu\neq 0$ has direct dynamical consequences. It entails kinematic-expansivity of the magnetic flow (an infinite Bowen ball for the magnetic flow is a short orbit segment, \cref{thm:kinematic-expansivity}).
On the other hand, strong stable and unstable directions of a vector $v$ exist  only when $\int_0^\infty \dot x\,dt=\mu\int_0^\infty y\,dt$ and $\int_{-\infty}^0 \dot x\,dt=\mu\int_{-\infty}^0 y\,dt$ are finite, respectively, for stable and unstable \m\ Jacobi fields $x\dot\ga+yi\dot\ga$ along $\gav$. Finiteness is obvious only if the integrand decays exponentially; see \cref{prop:strong_stable_spaces}. 
Horocycle flows are the borderline case where neither integral converges and these directions do not exist.

\end{remark}

\begin{lemma}\label{lemma:sing-zero-0-magnetic-curvature}
If \ref{H1}, then $\Sing=\{v\in \sm\mid K_\mu(\gav(t))=0,\forall t\in \re\}$.
\end{lemma}
\begin{proof}
For $v \in \Sing$, by \cref{lemma:sing-vector-has-constant-y} there exists a nontangential parallel \m\  Jacobi field along $\ga_v$ with $y(t)\equiv C\neq0$ for all $t\in\re$. Therefore, $K_\mu(\gav(t))=0$ for all $t\in\re$ by \eqref{MagneticJacobiEquation}. Conversely, if $K_\mu(\gav(t))=0$ for all $t\in\re$, then $J(t)\coloneqq \mu t\dot \ga_v(t)+i\dot\ga_v(t) $ is a nontangential parallel \m\  Jacobi field along $\gav$, so $v\in \Sing$.
\end{proof}

\begin{corollary}\label{prop:singular-vector-orbit-closed-invariant}
If \ref{H1}, then $\Sing$ is $\f$-invariant and closed.
\end{corollary}

\begin{proof}
Closed: 
\(v\mapsto g(v, t) = \kmu(\gav(t)) = \kmu \circ \pi \circ \f_t(v)\) 
is continuous, so
\[
\Sing = \bigcap_{t \in \re} \big\{ v \in T^1\Sigma \mid \kmu(\gamma_v(t)) = 0 \big\} = \bigcap_{t \in \mathbb{R}} g(\,\cdot\,, t)^{-1}(\{0\})
\]
is an intersection of closed sets.
\cref{lemma:sing-zero-0-magnetic-curvature} gives invariance. 
\end{proof}
\begin{remark}\label{prop:dim-of-subbundle}
If \ref{H1} and $v\in\sm$, then 
$\dim(E^c(v))=1$, and $\dim(E^u)= \dim(E^s)$. For $v\in \Reg$, $\dim(E^u(v))= \dim(E^s(v))=1$. 
\end{remark}

\begin{remark}[{\cite[p.~300]{Burns-Paternain}}]\label{remark:riccati-equation}
Analogously to geodesic flows, if  $y(t)$ is a solution of \eqref{MagneticJacobiEquation} along a \m\ geodesic $\ga$, then (by the quotient rule) $u(t) = \dot y(t)/y(t)$ is a solution of the \m\  Riccati equation
\begin{equation}\label{eqn:riccati}
    \dot u(t)+u^2(t)+\kmu(\ga(t))=0.
\end{equation}
Geometrically, $u(t)$ is the geodesic curvature at $\ga(t)$ of the curve through $\ga(t)$ orthogonal to the \m\  variation corresponding to the \m\  Jacobi field $x(t)\dot\ga(t)+y(t)i\dot\ga(t)$. In the geodesic flow case $\{g_t\colon \sm\to\sm\}_{t\in\re}$, in particular, the curve through $\gav(t)$ orthogonal to the (un)stable variation containing $\gav$ is the (un)stable horosphere of $g_t(v)$, and $u(t)$ measures the hyperbolicity of $g_t(v)$. \emph{Orthospheres} of magnetic flows can be defined using this notion of orthogonality (\cref{def:F-F-t-magnetic-busemann-function,rmk:perspective-on-horospheres}).
\end{remark}

\subsection{Geodesic flows: isometries, ideal boundary and visibility}\label{ssection:geometry-geoflow}
In this section, we review some properties of a closed surface $\Sigma$ with curvature $K<0$ and its geodesic flow, which is not only an instance of magnetic flows of $\Sigma$ but also an important tool for us to later study magnetic flows.  All of these results generalize to higher-dimensional manifolds with negative sectional curvature (and under mild restrictions, to manifolds with nonpositive curvature; see \cite{Eberlein-book}).
We also highlight some commonalities and differences between surfaces (and their corresponding geodesic flows) with negative versus nonpositive curvature.

\subsubsection{Isometries}
\begin{definition}[Isometry]\label{def:isometry}
    Let $M_1$ and $M_2$ be Riemannian manifolds.  A diffeomorphism $\phi\colon M_1\to M_2$ is said to be an \emph{isometry} if
    $\langle u,v\rangle_p=\langle d\phi_p(u),d\phi_p(v)\rangle_{\phi(p)}$ for all $p\in M_1$ and $u,v\in T_pM_1$. Equivalently,  $d_2(\phi (x),\phi(y))=d_1(x,y)$ for all $x,y\in M_1$ \cite{Myers-Steenrod-theorem}.
\end{definition}
The following result will be used in  \cref{ssection:magnetic-visibility,section:mag-axes-and-flat-strip}.
\begin{proposition}\label{cor:isometry-preserves-geodesic-curvature}
    An isometry $\phi$ of a Riemannian surface preserves the geodesic curvature, and therefore sends \m\  geodesics to \m\  geodesics.
\end{proposition}
\begin{proof}
If $\ga$ is a \m\  geodesic, then by \cref{rmk:general-behaviors-easy-observation}(\ref{item:mu-mu-relation}), $\tilde \ga(t)\coloneqq \phi (\ga(t))$ is a \m\ geodesic  since
$\frac{D \dot{\tilde \ga}}{dt}= 
\frac{D}{dt}\big(d\phi(\dot\ga)\big)
= d\phi(\frac{D\dot\ga}{dt})
= d\phi(\mu i \dot\ga  )
=\mu i d\phi (\dot \ga)
=\mu i \dot {\tilde\ga}.
$
Here the second equation comes from the uniqueness of the Levi-Civit\`a connection \cite[p.~71]{John-Lee-Riemannian-MFD}.
\end{proof}

\begin{definition}[Axial isometry]\label{def:axial-iso-geodesic-flow}
An isometry $\phi\colon M\to M$ of a simply connected manifold $M$ is \emph{axial} if there exists a geodesic $a$ on $M$, called the \emph{axis} of $\phi$ preserved by $\phi$, that is, $\phi(a(\re))=a(\re)$ (or equivalently, $\phi(a(t))=a(t+\omega) \text{ for all }t \in \mathbb{R}$; we usually take $\omega>0$).
\end{definition}

\begin{definition}[Universal cover and deck transformations]\label{def:deck-trans}
    The \emph{universal cover} $\wtilde{M}$ of $M$ is a simply connected manifold together with a covering projection $\operatorname{proj} \colon \wtilde{M} \to M$. A 
    \emph{deck transformation} is an isometry $\phi \colon \wtilde{M} \to \wtilde{M}$ 
    satisfying $\operatorname{proj} \circ \phi = \operatorname{proj}$. The group of all deck transformations is 
    denoted $D$, and we write $M = \wtilde{M} / D$.
\end{definition}

Every nontrivial deck transformation is an axial isometry (\cref{def:axial-iso-geodesic-flow}) whose axis projects to a closed geodesic on $M$; see \cite[Proposition 10.15]{Bishop-Oneil}, \cite[p.~133]{Ballmann}, and \cite[p.~258]{Bridson-Haefliger}.

\subsubsection{Ideal boundary and cone topology}\label{ssec:ideal-boundary-cone-topo}
The following consequence of the Cartan--Hadamard Theorem (\cite[Chapter 7, Theorem~3.1]{DoCarmo-Riem-Geo}) will be used in \cref{def:chord}. 
\begin{lemma}\label{Lemma:connectivity-by-geodesic}
For $p,q\in \wtilde{M}$, there is a unique geodesic (up to reparametrization) connecting $p$ and $q$.
\end{lemma}
Compactness of $M$ implies (geodesic) completeness, which allows us to define asymptotes.
\begin{definition}[Asymptotic geodesics]\label{def:geoflow_asymp}
Two geodesics $g$ and $h$ on $\wtilde{M}$ are said to be (forward) \emph{asymptotic} if $t\mapsto d(g(t),h(t))$ is bounded on $[0,\infty)$. Two unit vectors $v,w\in T^1M$ are said to be \emph{(forward) asymptotic} if their corresponding geodesics $g_v$ and $g_w$ are so.
\end{definition}
This is an equivalence relation on the set of geodesics, and we later introduce a magnetic counterpart (\cref{def:asymptotic-mag-geodesics,prop:asymptoticity-equivalence-relation}).

\begin{definition}[Ideal boundary]\label{def:geo-boundary-point-asymp-class}
A point at infinity for $\wtilde{M}$ is an equivalence class of asymptotic geodesics of $\wtilde{M}$. The set of all points at infinity of $\wtilde{M}$ is denoted by $\wtilde{M}(\infty)$. The equivalence class represented by a geodesic $g$ is denoted by $g(\infty)$, and the equivalence class represented by the oppositely oriented geodesic $g^{-1}\colon t\to g(-t)$ is denoted by $g(-\infty)$.
\end{definition}

\begin{proposition}\label{prop:connectivity-geoflow-MtoMbdry}Let $M$ be a closed Riemannian manifold with negative sectional curvature.
\begin{enumerate}
   \item For $p\in \wtilde{M}$ and $x\in\wtilde{M}(\infty)$ there is a unique geodesic $g_{px}\in x$ with $g_{px}(0)=p$ \cite[Proposition 1.2]{visibility}.
   
         \item (Visibility, \cite[Lemma 9.10]{Bishop-Oneil}, \cite[Corollary 5.2]{visibility}) For any two distinct points $x, y\in\wtilde{M}(\infty)$, there exists a unique (up to reparametrization) geodesic $g_{xy}$ on $\wtilde{M}$ such that $g_{xy}(-\infty)=x$ and $g_{xy}(\infty)=y$. \label{lemma:connectivity_geoflow_nega_curv}
\end{enumerate}
\end{proposition}

\begin{definition}[Visibility]\label{def:visibility}
An Hadamard manifold $N$ is said to have the \emph{visibility property} if for $x\neq y \in N(\infty)$ there is a geodesic $g_{xy}$ with $g_{xy}(-\infty)=x$ and $g_{xy}(\infty)=y$. A manifold of nonpositive curvature is said to have the visibility property if its universal cover does. 
\end{definition}
\begin{remark}\label{remark:visibility_nonpositive_curvature}
Uniqueness in \cref{prop:connectivity-geoflow-MtoMbdry}(\ref{lemma:connectivity_geoflow_nega_curv}) fails in the presence of flat strips; see \cite[Proposition 5.1]{visibility}.
\end{remark}

The visibility property was first studied for manifolds with nonpositive curvature, 
and there are several equivalent formulations; see \cite{visibility}. We will show that in the magnetic setting, \ref{H2}, implies that $\Sigma$ has \m\  visibility; see \cref{thm:ptildeqtilde-magnetic-visibility-property}.

We introduce the cone topology on $\oline{M}\dfn\wtilde{M}\cup \wtilde{M}(\infty)$ following \cite{visibility,Eberlein-book}. 
The magnetic counterpart is \cref{def:magentic-cone-topology}.

\begin{definition}\label{def:truncated_cone}
Let $M$ be a Riemannian manifold and let $p\in\wtilde{M}$.  For $v\in T_p^1\wtilde{M}$ and $\epsilon,r>0$, the truncated cone with vertex $p=\pi v$, axis $v$, and angle $\epsilon$ is \[T(v,\epsilon,r)\dfn \{q\in \oline{M} \mid \measuredangle_p(g_v(\infty),q)<\epsilon\}\smallsetminus\{q\in\wtilde{M}\mid d(p,q)\leq r)\},\]where $g_v$ is the geodesic induced by vector $v$ as in \cref{remark:notation_mag_geo}, and  
$\measuredangle_p(x,y)\dfn \measuredangle(\dot g_{px}(0),\dot g_{py}(0))$.
\end{definition}

\begin{proposition}[{\cite[Proposition 2.3, Proposition 2.9, Theorem 2.10]{visibility}}]\label{ConeTopology}
    If $M$ is a closed manifold with negative curvature, then there is a unique topology $\tau$ on $\oline{M}$ such that
    \begin{enumerate}
        \item for $v\in T_p^1\wtilde{M}$, the truncated cones $T(v,\epsilon,r)$ for $\epsilon,r>0$ form a neighborhood basis for $\tau$ at $g_v(\infty)$;
        \item the topology on $\wtilde{M}$ induced by $\tau$ is the original topology of $\wtilde{M}$.
    \end{enumerate}
        \noindent As a consequence,
    \begin{enumerate}[resume]
        \item $\oline{M}$ is homeomorphic to the closed unit disk;\item $f_p\colon T_p^1\wtilde{M} \to \wtilde{M}(\infty)$, $v\mapsto g_v(\infty)$ is a homeomorphism for every $p\in\wtilde{M}$;
        \item $\wtilde{M}$ is a open dense in $\oline{M}$.
    \end{enumerate}
This topology is called the \emph{cone topology}.
\end{proposition}

Axial isometries have north-south dynamics on $\oline{M}$:
\begin{proposition}[{\cite[Chapter~III, Lemma~3.3]{Ballmann_lecture}}]\label{prop:endpoints-converge-to-axis-endpoints}
    Suppose $M$ is an oriented closed manifold with negative curvature. If $\phi$ is an axial isometry with axis $a \colon \re \to \wtilde{M}$, then
    \[a(\pm\infty)\in U_\pm\subset\oline{M}\text{ open}\implies \exists N \in \mathbb{N}\ \forall n \geq N\quad \phi^{\pm n}(\oline{M} \setminus U_\mp) \subset U_\pm.
    \]
\end{proposition}
Pairs of endpoints of axes are dense, and \cref{prop:periodic-orbits-vectors} is a magnetic counterpart.
\begin{proposition}[{\cite[Theorem 2.13]{Ballmann}}]\label{prop:axial-endpoints-pairs-are-dense}
    If $M$ is an oriented closed manifold with negative curvature, then the set of pairs $(x, y) \in \wtilde{M}(\infty) \times \wtilde{M}(\infty)$ that are endpoints at $\wtilde{M}(\infty)$ of the axis of some axial isometry is dense in $\wtilde{M}(\infty) \times \wtilde{M}(\infty)$.
\end{proposition} 

\begin{proposition}\label{prop:cone-nbhd-contains-fundamental-domain}
    If $M$ is a closed manifold with negative curvature and $x \in \wtilde{M}(\infty)$, then any cone topology neighborhood of $x$ in $\oline{M}$ contains a translate $\phi D$ of a fundamental domain $D$ for some axial isometry $\phi$ of $M$.
\end{proposition}
\begin{proof}

For a (cone topology) neighborhood $N$ of $x \in \wtilde{M}(\infty)$,  by \cref{prop:axial-endpoints-pairs-are-dense} there exists an $x'\in N\cap\wtilde{M}(\infty)$ which is the attracting fixed point of an axial isometry $\varphi$. $x'$ is a global attractor on the complement of the repelling fixed point. Thus, by \cref{prop:endpoints-converge-to-axis-endpoints}, for the compact set $D$, there is an $n$ such that $\varphi^n(D)\subset N$. Take $\phi=\varphi^n$.
\end{proof}


\subsection{Dynamical properties of the geodesic flow}\label{SECDynPropGeodFlow}

\begin{proposition}[\cite{Eberlein-geodesic-flow-negative-curvature-I}]\label{prop:geo-periodic} If $M$ is an oriented closed manifold with negative curvature, then
\begin{enumerate}
    \item tangent vectors of axes are dense in $T^1\tilde M$;
    \item periodic unit vectors are dense in $T^1M$;
    \item there are infinitely many distinct periodic orbits in $T^1M$;
    \item the geodesic flow has a dense orbit in $T^1M$.
\end{enumerate}
\end{proposition}
\cref{prop:periodic-orbits-vectors,thm:topologically-transitivity} are magnetic counterparts of \cref{prop:geo-periodic} in the surface case. Their proofs use
\cref{def:limit-set-non-wandering-set,prop:duality-equivalence}.
\begin{definition}\label{def:limit-set-non-wandering-set}
For the group $D$ of deck transformations of a manifold $M$ (\cref{def:deck-trans}),
\begin{enumerate}
    \item define the \emph{limit set} $L(D)$ of $D$ as the set of accumulation points in $\oline{M}$ of $Dp$ for some (hence all) $p\in \wtilde{M}$; 
    \item a vector $v\in T^1\wtilde{M}$ is said to be \emph{nonwandering} with respect to $D$ (and the flow $f$) if there exist sequences $\{\phi_n\}$ in $D$, $\{t_n\}$ in $\re$, and $\{v_n\}$ in $T^1\wtilde{M}$ such that $t_n\to\infty$, $v_n\to v$, and $\phi_n(f_{t_n}(v_n))\to v$. Define the \emph{nonwandering set} $\Omega(D)$ as the collection of nonwandering vectors (see \cref{def:nw-set} for an equivalent definition).
\end{enumerate}
\end{definition}
\begin{remark}\label{rmk:nonwandering-set-is-entire-sm}
    A closed manifold $M$ with negative curvature has finite volume, and so does $T^1 M$; see \cite[p.~505]{Eberlein-geodesic-flow-negative-curvature-I}.
\end{remark}

\begin{proposition}\label{prop:duality-equivalence}
    Let $M$ be a closed Riemannian manifold with negative curvature. The following are equivalent.
    \begin{enumerate}
        \item $\Omega(D)=T^1\wtilde{M}$; \label{item11}
        \item for any two points $x,y\in \wtilde{M}(\infty)$ and any open sets $U$, $V$ containing $x$, $y$ respectively, there exists $\phi\in D$ such that $\phi(\oline{M}\smallsetminus U)\subset V$; \label{item12}
        \item for any two points $x,y \in \oline{M}(\infty)$, there exists a sequence $\{\phi_n\}\subset D$ such that for any point $p\in \wtilde{M}$, $\phi^{-1}_n(p)\to x$ and $\phi_n(p)\to y$.\label{item13}
    \end{enumerate}
\end{proposition}
\begin{proof}
    (\ref{item11}) $\iff$ (\ref{item12}) is from \cite{SSChen-Eberlein}, and (\ref{item12}) $\iff$ (\ref{item13}) is \cite[Proposition 2.5]{Eberlein-non-conjugate}.
\end{proof}
\subsection{Dynamics preliminaries}
For use in \cref{ssection:expansivity}, we review definitions and facts about topological pressure and equilibrium states for flows and homeomorphisms \cite[\S 4.3]{Fisher-Hasselblatt}.

\subsubsection{Topological pressure and equilibrium states}\label{ssection:topological_entropy_and_MME}

Let $X$ be a compact metric space, $\varphi\colon X\to \mathbb{R}$ continuous, $\mathcal{F}\colon \mathbb{R}\times X \to X$, $(t,x)\mapsto f_t(x)$ a \emph{continuous flow}, and 
$\mathcal{M}(\mathcal{F})$ the space of $\mathcal{F}$-invariant probability measures on $X$, i.e., $\mu \in \mathcal{M}(\mathcal{F})$ if and only if $(f_t)_*\mu = \mu$ for every $t \in \mathbb{R}$. We recall the definition of the topological pressure of $\varphi$ with respect to $\mathcal{F}$ \cite{Walters, Fisher-Hasselblatt}. 

For $t,\epsilon>0$, the \emph{Bowen ball} centered at $x\in X$ of radius $\epsilon$ and order $t$ is \begin{equation}\label{def:bowen_ball} B_t(x,\epsilon)\dfn\{y\in X\mid d(f_sx,f_sy)<\epsilon, \ \forall s\in [0,t]\}. \end{equation}

A set $E\subset X$ is \emph{$(t,\epsilon)$-separated in $F\subset X$} if $E\subset F$ and for all distinct $x,y\in F$, we have $y\notin \overline{B_t(x,\epsilon)}$. When $F=X$ we simply say $E$ is $(t,\epsilon)$-separated. 
Set $S_t\varphi \dfn \int_{0}^{t} \varphi\circ f_s\,ds$ and
\begin{equation}\label{eqn:Lambda-sep}
N_d(\varphi,\epsilon, t) \dfn \sup
\Big\{ \sum_{x\in E} e^{S_t\varphi(x)} \mid E\subset X \text{ is $(t,\epsilon)$-separated} \Big\},
\end{equation}
\[V(\mathcal{F})\dfn\big\{\varphi\colon X\to \re\mid\exists K,\eps>0, \; \forall t>0, x\in B_t(y,\eps)\colon\lvert S_t\varphi(x)-S_t\varphi(y) \rvert<K\big\},\]
the set of \emph{Bowen-bounded} functions. Here, $\varphi$ is called a \emph{potential function}.

\begin{definition}[{\cite[Definition 4.3.2, Remark 4.3.3]{Fisher-Hasselblatt}}]
\label{def:topological-pressue-potential-function-bowen-bounded}
    The \emph{topological pressure of $\varphi\colon X\to \re$ with respect to $\mathcal{F}$} is 
\[
P(\varphi) \dfn P(\mathcal{F}, \varphi)= \lim_{\epsilon\to 0} \limsup_{t\to\infty} \frac 1t \log N_d(\varphi,\epsilon, t).
\]
In particular, $P(0)$ is called the \emph{topological entropy}.
\end{definition}
 
The \emph{variational principle} \cite[Theorem 4.3.8]{Fisher-Hasselblatt} states that
\[
P(\varphi)=\sup_{\mu\in \mathcal{M}(\mathcal{F})}\Big\{ h_{\mu}(\mathcal{F}) +\int \varphi \,d\mu\Big\},
\]
where $h_\mu(\mathcal{F})$ is the measure-theoretic entropy of $f_1$ \cite[Definition 4.3.9]{Katok-Hasselblatt}.

A measure achieving the supremum is called an \emph{equilibrium state for $\varphi$} (or a \emph{measure of maximal entropy} when $\varphi=0$). 

\begin{theorem}[{\cite[Proposition 3.7, Proposition 4.15]{Climenhaga-Thompson-advances}}, Existence of equilibrium states]\label{thm:FH19-existence-ES-construction}
    For a flow $\mathcal{F} =\{f_t\}_{t\in\re}$ on a compact metric space $X$ such that the time-1 map $f_1$ of $\mathcal{F}$ is entropy-expansive (\cref{def:h-expansivity}), if $\varphi\in V(\mathcal{F})$ and $P(\varphi)<\infty$, then every $\text{weak}^{\star}$-accumulation point $\nu$ of $\nu_n$ is an equilibrium state for $\varphi$, where
    \[\nu_n\dfn \frac{1}{n}\int_0^n (f_s)_*\mu_nds,\qquad\mu_n\dfn \big(\sum_{x\in E_n}e^{S_n\varphi(x)}\big)^{-1}\sum_{x\in E_n}e^{S_n\varphi(x)}\delta_x,\]
    $\delta_x$ is the Dirac measure at $x$, and $E_n\subset X$ is an $(n,\eps)$-separated set such that for some $\de>0$, with $N_d$ defined in \eqref{eqn:Lambda-sep},
    \begin{equation}\label{eqn:separated-set-card}
        \sum_{x\in E_n}e^{S_n\varphi(x)}\geq N_d(\varphi,\eps,n)-\delta.
    \end{equation} 
    \end{theorem}
In \cref{ssection:expansivity}, we will use the \emph{kinematic-expansivity} property (\cref{def:k-expansivity}) of a magnetic flow to prove entropy-expansivity of its time-1 map, and then apply \cref{thm:FH19-existence-ES-construction} to show the existence of equilibrium states of the magnetic flow. Whether such an equilibrium state is unique in our context remains an open question.

\subsubsection{Topological entropy and entropy-expansivity of a homeomorphism}
Since \cref{thm:FH19-existence-ES-construction} assumes entropy-expansivity, we need to define topological entropy. 
Entropy-expansivity (or $h$-expansivity) was first studied by Bowen \cite{Bowen-h-expansive} and later by Knieper \cite[\S 3]{Knieper} for the time-1 map of geodesic flows of manifolds with nonpositive curvature. 
We discuss in \cref{ssection:expansivity} the time-1 map of magnetic flows (\cref{prop:mag-flow-is-h-expansive}) and here review notations and definitions; see \cite[\S 3]{Knieper} for more details.

    Let $X$ be a compact metric space, $f\colon X\to X$ be a homeomorphism, and $E\subset X$. A set $N=N(n,\eps,E)\subset X$ is said to be an $(n,\eps)$-spanning set of $E$ with respect to $f$ if for each $y\in E$, there exists $x\in N$ such that $d(f^kx,f^ky)\leq \eps$ for all $0\leq k<n$. Let 
    \begin{equation}\label{def:r_n}
        r_n(E,\eps)\dfn\min\big\{\operatorname{card} N \mid N \text{ is an } (n,\eps)\text{-spanning set of } E\big\}.
    \end{equation}
    Since $X$ is compact, $r_n(E,\eps)<\infty$. Define 
    \[r(E,\eps)=\limsup_{n\to\infty}\frac{1}{n}\log r_n(E,\eps)\]
    and \begin{equation}\label{eqn:topological-entropy-on-a-set}
        h(f,E)=\lim_{\eps\to0}r(E,\eps).
    \end{equation}

\begin{definition}[Topological entropy of a homeomorphism]\label{def:topo-entropy-homeo}
    $h(f)=h(f,X)$ from \eqref{eqn:topological-entropy-on-a-set} is called the \emph{topological entropy} of $f$ (it is independent of the metric generating the topology).
\end{definition}
\begin{definition}[Entropy-expansivity]\label{def:h-expansivity}
For a homeomorphism $f\colon X\to X$, $x\in X$, $\eps_0>0$, define the infinite-order Bowen ball by \[Z_{\eps_0}(x)=\bigcap_{n\in\mathbb{Z}}f^{-n}B(f^nx,\eps_0)=\{y\in X\mid d(f^nx,f^ny)\leq\eps_0 \text{ for all } n\in\mathbb{Z}\},\]  
    where $B(y,\eps_0)$ is the $\eps_0$-ball centered at $y$ with respect to $d$. Referring to \eqref{eqn:topological-entropy-on-a-set}, define 
    \begin{equation}\label{eqn:h-*}
        h^*(f,\eps_0)\dfn\sup_{x\in X}h(f,Z_{\eps_0}(x)).
    \end{equation}
$f$ is said to be \emph{entropy-expansive} (or \emph{$h$-expansive}) if there is an $\eps_0>0$ such that $h^*(f,\eps_0)=0$. Such $\eps_0$ is called an \emph{entropy-expansivity} constant for $f$.
\end{definition}

\section{Magnetic flows with nonpositive magnetic curvature}\label{section:geometry_of_mag_flow}

Recall that we assume the following throughout: $(\Sigma,g,\Omega_g,\mu)$ is an oriented closed Riemannian surface satisfying \ref{H1}, i.e., with  area form $\Omega_g$, magnetic structure $(g,\Omega_g)$, Gauss curvature $K<0$, and $\mu\in\re$ such that the \m\ curvature $K_\mu\le0$. Notably, we henceforth assume this (unless otherwise noted) whenever we use the terms ``magnetic flow'', ``magnetic geodesic'', or ``\m\  geodesic''. By a \emph{hyperbolic magnetic flow} we mean a magnetic flow that further satisfies \ref{H2}. By contrast, we say that $(\Sigma,g,\Omega_g,\mu)$ is \emph{$\mu$-magnetically flat} if $K_\mu\equiv0$.

In \cref{sec:mag-jf}, we study magnetic Jacobi fields, magnetic conjugate points and magnetic exponential maps, for use in the following sections. In \cref{ssection:connectivityI}, we define and study asymptotic magnetic geodesics and the magnetic boundary of $\M$ (\cref{def:asymptotic-mag-geodesics,def:asymptotic-maggeo-mag-boundary}) as a basis for investigating the large-scale geometry of magnetic flows. To do this, \cref{prop:boundary-bijection,prop:boundary-homeomorphic} relate asymptotic geodesics and the ideal boundary of $\M$ to asymptotic magnetic geodesics and the magnetic boundary. This also helps characterize magnetically flat surfaces in \cref{sec:mag-flat}: a surface is magnetically flat if and only if there exists a magnetic geodesic with the same forward and backward endpoint (\cref{prop:h2-has-no-closed-orbit}). A corollary of this result is that any orbit of a hyperbolic magnetic flow has distinct forward and backward endpoints, whose complement, called the magnetic visibility property of $\M$, is established in \cref{ssection:magnetic-visibility}. That is, any two distinct and ordered points on the magnetic boundary determine a hyperbolic magnetic geodesic. This property allows us to apply classic arguments to study the density of periodic orbits in \cref{section:mag-axes-and-flat-strip}, and topological transitivity and the nonwandering set in \cref{sec:topo-transitivity}. Uniqueness of magnetic geodesics in the magnetic visibility property fails in the presence of magnetic flat strips (\cref{def:magnetic-flat-strip}), whose dynamics is considered in \cref{section:mag-axes-and-flat-strip}.
In \cref{ssection:expansivity}, we study various notions of expansivity for magnetic flows. One consequence is a condition under which magnetic flows are orbit-equivalent to the geodesic flow of $\Sigma$ (\cref{thm:orbit-equivalence}). Moreover, the kinematic-expansivity of magnetic flows (\cref{thm:kinematic-expansivity}) guarantees the existence of equilibrium states (\cref{thm:existence-of-es-equidistribution}). 

\subsection{Magnetic Jacobi fields II}\label{sec:mag-jf}
Magnetic Jacobi fields encode the infinitesimal change of adjacent magnetic geodesics and play a similarly central role as Jacobi fields do in the study of geodesic flows (see \cref{ssection:magnetic_jacobi_field} and \cite[\S 5.5]{Burns-Gidea}).
\begin{lemma}\label{lemma:stable-y-is-monotonic}
    If $(\Sigma,\mu)$ satisfies \ref{H1}, then the perpendicular size $\lvert y(t)\rvert$ of a stable (resp. unstable) \m\ Jacobi field $J$ in \eqref{def:eqn:orthogonal-component-of magnetic-Jacobi-field} is nonincreasing (resp. nondecreasing) on $\re$. Moreover, $y(t_0)=0$ for some $t_0\in\re$ if and only if $J$ is tangential, i.e., $y\equiv0$. 
\end{lemma}
\begin{proof}
To show monotonicity of $\lvert y(t)\rvert$, we first show that $F(t)\dfn y^2(t)$ is convex on $\re$, so $F'$ is nondecreasing: $F''=(\lvert y\rvert^2)''=(y^2)''=2(y')^2+2yy''=2(y')^2-2\kmu y^2\geq0$. Next, $F'\le0$ for stable $J$: If there exists $t_0$ such that $F'(t_0)>0$, then $F'(t)\geq F'(t_0)$  for any $t\geq t_0$, hence $F(t)\geq F(t_0)+(t-t_0)F'(t_0)\xrightarrow{t\to\infty} \infty$, and $J$ is not stable. We have shown that $yy'\le0$.

Now, if $y(t_0)>0$ for some $t_0\in\re$, then $y'(t_0)\leq0$ and $|y(t)|$ is nonincreasing at $t_0$. Similarly, if $y(t_0)<0$ for some $t_0\in\re$, then $y'(t_0)\geq0$ and $|y(t)|$ is nonincreasing at $t_0$. 

Finally, suppose $y(t_0)=0$ for some $t_0\in\re$ to conclude the proof of monotonicity and also establish the second conclusion.
\begin{enumerate}
    \item If $y'(t_0)>0$, then by the contuniuity of $y(t)$ and $y'(t)$, there exists $t_1>t_0$ close to $t_0$ such that $y(t_1)>0$ and $y'(t_1)>0$, a contradiction.
    \item If $y'(t_0)<0$, then similarly, there exists $t_2>t_0$ close to $t_0$ such that $y(t_2)<0$ and $y'(t_2)<0$, a contradiction.
    \item If $y'(t_0)=0$, then by the uniqueness of the solution to the homogeneous second-order ODE \eqref{MagneticJacobiEquation}, $y(t)\equiv0$, which is monotone and implies that $J$ is tangential.\qedhere
\end{enumerate}

\end{proof}

\begin{lemma}\label{lem:parallel-iff-stable-and-unstable}
   If \ref{H1} and $J$ is a \m\  Jacobi field, then the following are equivalent.
   \begin{enumerate}
   \item $J$ is parallel;\label{p:item1}
   \item $J$ is both stable and unstable;\label{p:item2}
   \item the orthogonal component $y$ of $J$ is bounded on $\re$.\label{p:item3}
   \end{enumerate}
\end{lemma}

\begin{proof}
$\eqref{p:item1}\iff\eqref{p:item2}$ is \cref{lemma:sing-vector-has-constant-y}. $\eqref{p:item1}\implies\dot y\equiv 0\implies\ y$ is constant $\implies\eqref{p:item3}$. To see $\eqref{p:item3}\implies \eqref{p:item1}$, note that $y$ is convex and bounded on $\re$, hence constant.
\end{proof}

\begin{remark}\label{rmk:norm-of-mag-jf-convexity}
If \ref{H1}, then unlike in the case of the geodesic flow of manifolds with nonpositive curvature, the usual convexity argument for the norm of a Jacobi field does not directly extend to the norm of a \m\  Jacobi field due to the shearing effect (\cref{rmk:shearing-effect}).
\end{remark}

We now give two applications of what \cite{DoCarmo-Riem-Geo} calls Sturm's Theorem to magnetic Jacobi fields. The corresponding results for geodesic flows of manifolds with nonpositive curvature were established using the Rauch Comparison Theorem instead \cite[\S 2]{HI}. Later, we will use Sturm's Theorem again to prove $C^2$-regularity of orthospheres (\cref{prop:c2-regularity-of-magnetic-horosphere}).

\begin{lemma}[``Sturm's Theorem'' in {\cite[Chapter 10, \S 1]{DoCarmo-Riem-Geo}}]\label{thm:sturm}
For $i=1,2$, let $x_i$ be solutions of  $x''+p_ix=0$ with initial conditions $x_i(0)=0$  and $x_1'(0)=x_2'(0)$, $p_i$  continuous, and $p_1\leq p_2\le0$ on $[0,T]$. If $x_2>0$ on $(0,T]$, then $x_1\geq x_2$ on $[0,T]$.
\end{lemma}
\begin{proof}
$x_1\geq x_2$ at $t=0$, so it suffices to show that $x_1'(\tau)\ge x_2'(\tau)$ if $x_1\ge x_2$ on $[0,\tau]$. To that end, note that $z\dfn x_1'-x_2'\ge0$ at 0 and \(z'=x_1''-x_2''=-p_1x_1+p_2x_2\ge0\) on \([0,\tau]\).
\end{proof}
\begin{lemma}\label{thm:sturm2}
If $x_i$ for $i=1,2$, are solutions of  $x''+p_ix=0$ with boundary values $x_2(0)\ge x_1(0)\ge0=x_i(\infty)$, $p_i$  continuous, and $p_1\leq p_2\le0$ on $[0,\infty]$, then $0\le x_1\leq x_2$ on $[0,\infty]$.
\end{lemma}
\begin{proof}
(Suggested by ChatGPT.) We use a maximum principle: $u\ge0$ if $u''+pu\le0$ with nonnegative boundary values and $p\le0$. (Proof sketch for negative $p$: otherwise, there is a negative minimum, at which $0\le u''\le-pu$, so $p\ge0$.)
Then, $x_i\geq0$ and $w\dfn x_2-x_1\ge0$ since $w(0)\ge0=w(\infty)$ and 
\(w''+p_2w
 = x_2''-x_1''+p_2(x_2-x_1)
 = (p_1-p_2)x_1
 \leq0
\).
\end{proof}
\begin{proposition}
If \ref{H1}, then for any \m\  geodesic $\ga$ and any $v\in T_{\ga(0)}\Sigma$, there is a unique stable \m\  Jacobi field $J$ along $\ga$ with $J(0)=v$.  
\end{proposition}
\begin{proof}
Uniqueness: Suppose \m\  Jacobi fields $J_1$ and $J_2$ along $\ga$ are stable and $J_1(0)=J_2(0)=v$. Then $J\dfn J_1-J_2$ is a stable \m\  Jacobi field along $\ga$ with $J(0)=0$. By \cref{lemma:stable-y-is-monotonic} we have $J\equiv 0$, and therefore, $J_1=J_2$.

Existence: There exists a  \m\  Jacobi field $J_n=x_n\dot\ga+y_ni\dot\ga$ along $\ga$ with $J_n(0)=v$ and $y_n(n)=0$ because the boundary value problem $\ddot y_n+\kmu y_n=0$, $y_n(0)=\langle v, \dot\ga(0) \rangle$, $y_n(n)=0$ has a solution when $\kmu\leq0$. Such $J_n$ is unique: if $J_{n,i}$, $i\in\{1,2\}$ are such that $J_{n,i}(0)=v$ and $y_{n,i}(n)=0$, then $(y_{n,1}-y_{n,2})(t)$ is a solution of $w_n''+\kmu w_n=0$ with $w_n(0)=w_n(n)=0$. Multiplying $w'' + \kmu w = 0$ by $w$ and integrating by parts over $[0,n]$ gives
$\int_0^n (w')^2\,dt = \int_0^n \kmu  w^2\,dt \leq 0$ since $\kmu \le 0$,
so $y_{n,1}-y_{n,2} \equiv 0$ on $[0,n]$. Uniqueness for the linear IVP then gives $y_{n,1}-y_{n,2} \equiv 0$ on $\re$, hence $J_{n,1}=J_{n,2}$ by \eqref{eqn:magJac1}.

Fix $m\geq n$ and suppose $y_m,y_n\geq 0$ on $[0,m]$ and $[0,n]$ respectively. We compare two solutions via Sturm's Theorem: $y_m - y_n$, which solves $y'' + \kmu y = 0$ with $y(0) = 0$ and $y_m - y_n\geq0$ on $[0,n]$, and the linear comparison function $z(t) = \big(y_m'(0) - y_n'(0)\big)\cdot t$, which solves $z'' = 0$ with $z(0) = 0$. Since $\kmu \leq 0$, Sturm's Theorem gives
    \begin{equation}
        y'_m(0)-y'_n(0) \leq \frac{1}{t} \big(y_n(t)-y_m(t)\big)
    \end{equation}
    for $t\in [0,n]$.
    Since $y_n$ and $y_m$ are convex, both $ y_n$ and $ y_m$ are nonincreasing on $[0,n]$.
   Therefore, $0\le y_m(t)-y_n(t) \leq \max\{\lvert y_n(0)\rvert, \lvert y_m(0)\rvert \}\leq\lvert v\rvert$ for $t\in[0,n]$, and \eqref{eqn:J'} gives 
    \begin{equation}
        \lvert J'_n(0)-J'_m(0)\rvert=\lvert \mu\big(x_n(0)-x_m(0)\big)+ y_n'(0)-y_m'(0)\rvert\leq  \frac{1}{t}\lvert y_n(t)-y_m(t)\rvert\leq \frac{1}{t}\lvert v \rvert.
    \end{equation}
    Hence $\{J_n'(0)\}$ is a Cauchy sequence. The \m\  Jacobi field $J_v$ along $\ga$ with $J_v(0)=v$ and $J'_v(0)=\lim_{n\to\infty}J_n'(0)$ is the limit of the $J_n$, hence stable.
\end{proof}
\begin{theorem}[Comparison theorem for stable magnetic Jacobi fields]\label{THMStableComp}
If $0\leq a\leq b<\infty$, $-b^2\leq \kmu\leq -a^2$, $\ga$ is a \m\  geodesic, $J(t)=x(t)\dot\ga(t)+y(t)i\dot\ga(t)$ is the unique stable \m\  Jacobi field along $\ga$ with $J(0)=v\in T_{\ga(0)}\Sigma$ and $v\perp \dot\ga(0)$, and $t\geq 0$, then
    \begin{enumerate}
        \item $\lvert v \rvert e^{-bt}\leq  \lvert  y(t)\rvert\leq \lvert v \rvert e^{-at} $;\label{item:comparison_1}
        \item $\mu a \lvert v \rvert \big( 1-e^{-bt} \big)\leq ab \lvert x(t)\rvert \leq \mu b \lvert v \rvert \big( 1-e^{-at} \big)$. \label{item:comparison_2}
    \end{enumerate}
\end{theorem}
\begin{remark}\label{REMNegCurvImpliesAnosov}\cref{THMStableComp}\eqref{item:comparison_1} gives $\ref{H4}\implies\ref{HA}$. 
\end{remark}
\begin{proof}
\eqref{item:comparison_1} follows from \cref{thm:sturm2}.
\eqref{item:comparison_2}: $x(0)=0$ since $v\perp \dot\ga(0)$, so $\vert x(t)\rvert=\lvert \mu\int_0^t y(s)ds\rvert$ from \eqref{eqn:magJac1}. \eqref{item:comparison_1} gives $a\lvert x(t)\rvert\leq \mu\lvert \int_0^t  \lvert v \rvert ae^{-at}   ds\rvert\leq \mu\lvert v \rvert \big( 1-e^{-at} \big)$. Likewise for $b$.
\end{proof}
\begin{definition}[Magnetic conjugate point]\label{def:magnetic-conjugate-point} Two points $p$ and $q$ on $\Sigma$ are $\mu$-magnetically conjugate along a \m\  geodesic $\ga$ if there exists a nonzero \m\  Jacobi field along $\ga$ whose orthogonal component vanishes at both $p$ and $q$. We say $\f$ has no conjugate points if no such pair exists for any \m\ geodesic.   
\end{definition}
\begin{definition}[Magnetic exponential map]\label{def:magnetic-exponential-function}$\emu\colon T\M\to \M$, $v\mapsto\ga_{v/\left\Vert v\right\Vert}(\lVert v\rVert)$ is called the \emph{\m\ exponential map}, where $\ga_{v/\left\Vert v\right\Vert}$ is from \cref{def:mag flow}.
\end{definition}

\begin{theorem}[{\cite[Theorem 3]{Adachi2}}] \label{thm:adachi-connectivity}
If \ref{H1}, then $\emu_p$ is a covering map at any $p\in \M$.
\end{theorem}

This gives the magnetic counterpart of \cref{Lemma:connectivity-by-geodesic}. 
\begin{corollary}\label{lemma:pq-connectivity-and-no-self-intersect}
If \ref{H1}, then for $p,q \in \M$ there is a unique \m\  geodesic $\gamma_{pq}$ with $\gamma_{pq}(0)=p$ and $\gamma_{pq}(\tilde{d})=q$ for some $\tilde{d}\geq0$ (when $\mu=0$, we write $g_{pq}$ for the geodesic $\gamma_{pq}$). In particular, no $\mu$-magnetic geodesic self-intersects on $\M$, and $\f$ has no $\mu$-magnetically conjugate points. 
\end{corollary}

\begin{proof}
Since $\M$ is simply connected, by \cref{thm:adachi-connectivity}, $\emu_p$ is in fact a diffeomorphism. Hence $\gamma_{pq}(t) = \emu_p(t\, (\emu_p)^{-1}(q)/\|(\emu_p)^{-1}(q)\|)$ is the unique \m\ geodesic from $p$ to $q$, and $\ga_{pq}$ does not self-intersect.

Fix $v \in T^1_p\M$ and $t_0>0$, and identify $T_{t_0v}(T_p\M) \cong T_p\M$. For $w = cv + w_\perp \in T_p\M$ with $w_\perp \perp v$, let $J(t) = x(t)\dot\gamma_v(t) + y(t)\,i\dot\gamma_v(t)$ be the \m\ Jacobi field along $\gamma_v$ with $J(0)=0$ and $J'(0)=w_\perp$. Then $d(\emu_p)_{t_0v}(w_\perp) = J(t_0)$ and $$d(\emu_p)_{t_0v}(w) = \bigl(c+x(t_0)\bigr)\dot\gamma_v(t_0) + y(t_0)\,i\dot\gamma_v(t_0),$$ 
which vanishes if and only if $c+x(t_0)=0$ and $y(t_0)=0$. Since $c$ is free, $\ker d(\emu_p)_{t_0v}$ is nontrivial if and only if some nonzero \m\ Jacobi field $J$ along $\gamma_v$ satisfies $y(0)=y(t_0)=0$, i.e., $p$ and $\gamma_v(t_0)$ are $\mu$-magnetically conjugate (\cref{def:magnetic-conjugate-point}). Since $\emu_p$ is a diffeomorphism, $d(\emu_p)_{t_0v}$ is injective for all $v$ and $t_0>0$, hence 
$\f$ has no $\mu$-magnetically conjugate points.
\end{proof}

\begin{definition}[Magnetic distance]\label{def:mag-distance}
$\tilde{d}_\mu(p,q)\dfn\tilde{d}$ from \cref{lemma:pq-connectivity-and-no-self-intersect} is called the \emph{\m\  distance} from $p$ to $q$.
\end{definition}
\begin{remark}
This is not a metric because $\ga_{pq}$ is different from $\ga_{qp}$, so $\tilde{d}_\mu(p,q)\neq\tilde{d}_\mu(q,p)$ (\cref{cor:td-is-not-metric}). Moreover, the triangle inequality fails (\cref{lemma:magnetic-triangle-inequality-fails}).
\end{remark}


\subsection{Magnetic boundary}\label{ssection:connectivityI}

Mirroring the classical construction of the ideal boundary $\Minfty$ of $\M$ via asymptotic geodesic rays, we
define the \m\ boundary $\M_\mu(\infty)$ via
asymptotic \m\ geodesics (\cref{def:asymptotic-mag-geodesics,def:asymptotic-maggeo-mag-boundary}) and establish a natural homeomorphism between $\Minfty$ and $\M_\mu(\infty)$ (\cref{prop:boundary-homeomorphic}).
The key ingredient is the angle between a \m\ geodesic and its geodesic chords, whose control (\cref{Lemma:angle_<_pi/2}) allows endpoints of a magnetic geodesic to be read off from those of its boundary chord  (\cref{prop:stable-mag-jf-same-endpoint,prop:boundary-bijection}), and the homeomorphism (\cref{lemma: well-define-boundary-homeo}) then yields a characterization of the magnetically flat case: it is precisely the one in which some magnetic geodesic has the same forward and backward endpoints (\cref{prop:h2-has-no-closed-orbit}).


\begin{proposition}[{\cite[p.~229, Theorem 1]{Adachi1}}]\label{prop:mag-intersect-geocircle-atmost-once}\label{cor:magnetic-geodesic-is-unbounded}
    Let $S_{r}(p)\dfn\{x\in \M\mid d(x, p)=r\}$ denote the geodesic circle  of radius $r$ centered at $p$. If $(\Sigma,\mu)$ satisfies \ref{H1}, then all \m\  rays $\gamma|_{[0,\infty)}$ and $\gamma|_{(-\infty,0]}$ cross the geodesic circle $S_{r}(\gamma(0))$ exactly once for any $r\geq 0$, and (by taking $r= d(\ga(0),\ga(s))$) any \m\  geodesic is unbounded in both directions.
\end{proposition}
This allows us to introduce asymptotic magnetic geodesics and the magnetic boundary.

\begin{definition}[Asymptotic magnetic geodesics]\label{def:asymptotic-mag-geodesics}
    Two \m\  geodesics $\ga_1$ and $\ga_2$ on $\M$ are \emph{forward asymptotic} if there exists a nondecreasing continuous function $s\colon \re\to \re$ such that $d(\ga_1(t),\ga_2(s(t)))<C$ for some $C>0$ and for all $t\in[0,\infty)$, or equivalently, if $d_H(\ga_1([0,\infty)),\ga_2([0,\infty)))<\infty$ where $d$ and $d_H$ are the distance and Hausdorff distance functions, respectively  induced by the Riemannian metric of $\Sigma$. 
    Two \m\  geodesics $\ga_1$ and $\ga_2$ are \emph{backward asymptotic} if there exists a nonincreasing continuous function $s\colon \re\to \re$ such that $d(\ga_1(t),\ga_2(s(t)))<C$ for some $C>0$ and for all $t\in(-\infty,0]$.
\end{definition}

\begin{proposition}\label{prop:asymptoticity-equivalence-relation}
    Being asymptotic is an equivalence relation.
\end{proposition}
\begin{proof}
    We show this for forward asymptoticity; the proof of backward asymptoticity is similar. Reflexivity and symmetry are obvious, so it remains to show transitivity.
    
    Let $\ga_i$, $i\in\{1,2,3\}$ be \m\  geodesics for which there are $C_1, C_2>0$ and continuous nondecreasing  $s_1,  s_2\colon \re\to \re$ such that $d(\ga_1(t),\ga_2(s_1(t)))<C_1$ and $d(\ga_2(t),\ga_3( s_2(t)))<C_2$ for all $t\in[0,\infty)$. With $s_3\dfn s_2\circ s_1$, the triangle inequality gives
    \[d(\ga_1(t),\ga_3(s_3(t)))\leq d(\ga_1(t),\ga_2(s_1(t)))+ d(\ga_2(s_1(t)),\ga_3(s_3(t)))<C_1+C_2.\qedhere\]
\end{proof}

\begin{definition}[Magnetic boundary]\label{def:asymptotic-maggeo-mag-boundary}
If \ref{H1} define the \emph{forward} and \emph{backward \m\  boundary} $\M^+_\mu(\infty)$ and $\M^-_\mu(\infty)$ of $\M$ as the equivalence classes of forward and backward asymptotic \m\  geodesics, respectively. We also write $\M_\mu(\infty)$ for either $\M^+_\mu(\infty)$ or $\M^-_\mu(\infty)$ when the choice is clear from context.

For any \m\  geodesic $\gamma$, we write $\gamma(\infty)$ for its forward equivalence class and $\gamma(-\infty)$ for its backward equivalence class.
\end{definition}
We could next imitate the analogous development for geodesic flows to attach these boundaries to the universal cover, but instead we will canonically identify the boundaries with the Riemannian boundary at infinity. The approach is to find an asymptotic direction of a magnetic orbit as ``seen'' by looking along geodesics.
\begin{definition}[Chord]\label{def:chord}\label{def:mag_angles}
A (geodesic) chord of a curve $C$ (at $C(0)$) is a unit-speed geodesic
\[s\mapsto g_{C,t}(s)\dfn\begin{cases} g_{C(0)C(t)}(s)&\text{if }t>0,\\g_{C'0)}(s)&\text{if }t=0,\\g_{C(t)C(0)}(s+l(t))&\text{if }t<0,\end{cases}\]
where $g_{C'(0)}$ is as in \cref{remark:notation_mag_geo}, $g_{pq}$ as in \cref{lemma:pq-connectivity-and-no-self-intersect}, $l(t)=d(C(0),C(t))$, and $d=\td_0$ (\cref{def:mag-distance}), so $g_{C,t}(0)=C(0)$ in either case; see \cref{fig:def_chord}.
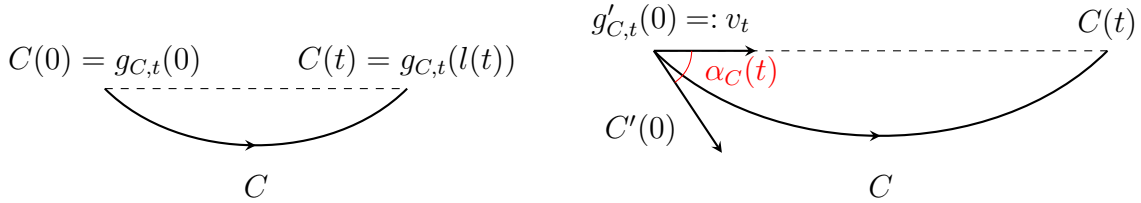
\begin{figure}[hbt]
  \centering
  \begin{tikzpicture}[>=stealth]
  \draw[thick, postaction={decorate},
      decoration={markings, mark=at position 0.5 with {\arrow{stealth}}}]
      (-2,0) .. controls (-1,-1) and (1,-1) .. (2,0);
    \draw[dashed] (-2,0) -- (2,0);

\node[above] at (-2,0) {$C(0)=g_{C,t}(0)$};
\node[above] at (2,0) {$C(t)=g_{C,t}(l(t))$};
\node[below] at (0,-1) {$C$};
\end{tikzpicture}\hfil\begin{tikzpicture}[>=stealth,scale=1.5]

\draw[dashed] (-2,0) -- (2,0);

\draw[thick, postaction={decorate},
      decoration={markings, mark=at position 0.5 with {\arrow{stealth}}}]
      (-2,0) .. controls (-1,-1) and (1,-1) .. (2,0);

\node[below] at (0,-1) {$C$};
\node[above] at (2,0) {$C(t)$};

\draw[->, thick] (-2,0) -- (-1.1,0);
\node[above] at (-1.85,0) {$g'_{C,t}(0)\nfd v_t$};

\draw[->, thick] (-2,0) -- ++(0.6,-0.9);
\node[left] at (-1.7,-0.7) {$C'(0)$};

\coordinate (O) at (-2,0);        
\coordinate (A) at (-1.7,0);      
\coordinate (B) at (-1.7,-0.455); 

\pic [draw=red,
      text=red,
      "$\alpha_C(t)$",
      angle radius=5mm,          
      angle eccentricity=1.2,    
      label distance=2mm,        
      pic text options={right}]  
      {angle = B--O--A};

\end{tikzpicture}
  \caption{Chord and angle in \cref{def:chord}}
  \label{fig:def_chord}\label{fig:chord_angle}
\end{figure}
Denote by $\alpha_C(t)$ the angle subtended by $C'(0)$ and $v_t\dfn g'_{C,t}(0)\in T^1_{C(0)}\M$ for any $t$ (\cref{fig:chord_angle}).
\end{definition}
Since $\f$ is continuous, so is $t\mapsto\alpha_\ga(t)$. We show that it is increasing and bounded, hence has a limit as $t\to\infty$; this provides the asymptotic direction.

\begin{lemma}\label{lemma:angle-alpha-is-increasing}
    If \ref{H1} and $\mu\neq0$, then $t\mapsto\alpha_\ga(t)$ is strictly increasing on $[0,\infty)$ for any \m\ geodesic $\ga$.   
\end{lemma}
\begin{proof}
Otherwise, there are a \m\ geodesic $\ga$ and $0<t_1 < t_2$ with $\alpha_\ga(t_1)=\alpha_\ga(t_2)$, so $p\dfn\ga(0)$ and $q\dfn\ga(t_2)$ satisfy $\ga_{pq}(\re)\cap g_{pq}(\re) \supset \{p,q, \ga(t_1)\}$, contrary to \cref{rmk:general-behaviors-easy-observation}\eqref{LeftRightTangency}.
\end{proof}

Note that if $K_\mu\equiv0$ (horocycle flow), then $\alpha_\ga\to\pi/2$ as $t\to\infty$. By contrast:
\begin{lemma}[{\cite{Grognet}}]\label{lemma:Anosovmag-angle<pi/2}
If \ref{H4}, then 
there exists $\alpha_0<\pi/2$ such that $\alpha_\ga(t)<\alpha_0$ for any $\mu$-magnetic geodesic $\ga$ and $t>0$.
\end{lemma}
This is surprisingly useful. Bounding $\alpha_\ga$ shows that it converges, but the uniform bound $\alpha_0<\pi/2$ in \cref{lemma:Anosovmag-angle<pi/2} yields more: the angle between any magnetic geodesic segment and its chord is uniformly acute, so every magnetic geodesic segment makes definite forward progress in every chord direction. This yields a continuous bijection between orbits of the magnetic flow and orbits of the geodesic flow, which allows Grognet to prove that such magnetic flows are quasigeodesic \cite[Proposition 3.1]{Grognet} and orbit-equivalent to the geodesic flow of $\Sigma$ (\cref{prop:grognet_oe}). 

We now extend \cref{lemma:Anosovmag-angle<pi/2} to \m\ flows satisfying \ref{H1}. Even without the uniform bound away from $\pi/2$, this will play an important role in extending \cref{lemma:pq-connectivity-and-no-self-intersect} to the case where one of the points $p$ and $q$ is on $\Minfty$ (\cref{def:angle-for-pqtilde-end-points,thm:pqtilde}), as well as in the proof of the characterization of magnetic flatness (\cref{prop:h2-has-no-closed-orbit}).

\begin{lemma}\label{Lemma:angle_<_pi/2}
If \ref{H1}, then $\alpha_\ga(t)<\pi/2$ for any $\mu$-magnetic geodesic $\ga$ and $t>0$.
\end{lemma}

\begin{proof}
If $\alpha_{\ga}(t_0)=\pi/2$ for some \m\  geodesic $\ga$ and some $t_0>0$, then $\alpha_{\ga}(t_1)>\pi/2$ for $t_1>t_0$ by \cref{lemma:angle-alpha-is-increasing}. By \cref{lemma:pq-connectivity-and-no-self-intersect}, for any $\mu$ that satisfies \ref{H1}, there is a unique $\mu$-magnetic geodesic $\ga_{\mu}$ such that $\ga_{\mu}(0)=\ga(0)$ and $\ga_{\mu}(t_{\mu})=\ga(t_1)$ for some $t_{\mu}>0$.  
   Since $\mu\mapsto \alpha_{\ga_\mu}(t_\mu)$ is continuous for all $\mu$ that satisfy \ref{H1}, there exists $\mu^*$ such that $K_{\mu^*}<0$ and $\alpha_{\ga_{\mu^*}}(t_{\mu^*})>\pi/2$, contrary to \cref{lemma:Anosovmag-angle<pi/2}.
\end{proof}

We next use boundary chords to establish a correspondence between the magnetic boundary and the Riemannian boundary at infinity (\cref{prop:boundary-bijection}).

\begin{definition}[Boundary chord; endpoint {\cite[p.~229]{Adachi1}}]\label{def:angle-for-pqtilde-end-points}
    Consider $(\Sigma,\mu)$ satisfying \ref{H1} and a \m\  geodesic $\ga$.  Let $g_{v^\pm}$ be the geodesic with initial vector $v^\pm\coloneqq \lim_{t\to\pm\infty}g'_{\ga,t}(0)\in T^1_{\ga(0)}\M$, where, $g_{\ga,t}$ is from \cref{def:chord} and $v^\pm$ is well defined by \cref{lemma:angle-alpha-is-increasing}, \cref{Lemma:angle_<_pi/2} and the monotone convergence theorem; see \cref{fig:extended_chord_vector}.  
       
    As in \cref{ssection:geometry-geoflow}, let $\om=\M\cup \Minfty$, equipped with the cone topology. We call $g_{v^+}\colon [0,\infty]\to \om$ (resp. $g_{v^-}\colon [-\infty,0]\to \om$) the \emph{forward} (resp. \emph{backward}) {boundary chord} of $\ga$ at $\gamma(0)$ (or of $\dot\gamma(0)$) and $\ga[\pm\infty]\dfn g_{v^\pm}(\pm\infty)\in\Minfty$ the \emph{forward}, and respectively, \emph{backward} endpoint of $\ga$. A geodesic $g$ is said to be the \emph{boundary chord} of $\ga$ if $g(\pm\infty)=\ga[\pm\infty]$.
\end{definition}
\begin{figure}[hbt]
  \centering
  \begin{tikzpicture}[>=stealth, scale=2.5]

\draw[dashed] (-2,0) -- (2,0);

\draw[thick,
      postaction={decorate},
     decoration={markings, mark=at position 0.5 with {\arrow{stealth}}}]
      (-2,0) .. controls (-1,-1) and (1,-1) .. (2,0);

\node[below] at (0,-0.8) {$\gamma$};
\node[above] at (-2,0) {$\gamma(0)$};
\node[above] at (2,0) {$\gamma[\infty]$};

\draw[->, thick] (-2,0) -- (-1,0);
\node[above] at (-1.5,0) {$v$};

\draw[dashed] (-2,0) -- (1.2,-0.5);
\node[below] at (1.4,-0.5) {$\gamma(t)$};

\draw[->, thick]
  (-2,0) -- ++(0.977,-0.154)
  node[below] {$g'_{\gamma,t}(0)$};
\end{tikzpicture}
  \caption{boundary chord and endpoint}
  \label{fig:extended_chord_vector}
\end{figure}
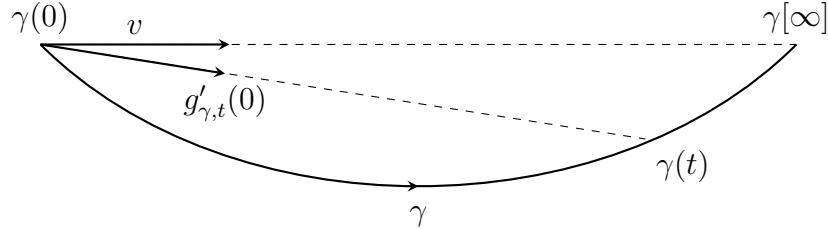

\begin{remark}\label{Halfplanechords}
Here, $\gamma([0,\infty))$ need not have finite Hausdorff distance from $g_v([0,\infty))$. Horizontal lines in the upper half-plane model of $\mathbb H^2$ as 1-magnetic geodesics provide an example: their endpoint corresponds to vertical geodesics, whose neighborhoods of size $r$ do not contain any 1-magnetic geodesic for any $r>0$. In contrast, the \m\ flow of $\Sigma$ with negative \m\ magnetic curvature is Anosov, and in fact, a quasigeodesic flow (\cite[Proposition 3.1]{Grognet}). Therefore, by the Morse lemma, the Hausdorff distance between a \m\ geodesic $\gamma$ and its boundary chord is finite;  see \cite[Chapter~6, Theorem~2.3]{Handbook-Hasselblatt-Katok}. This difference makes \cref{def:asymptotic-maggeo-mag-boundary} important.
\end{remark}

\begin{lemma}
    If $(\Sigma,\mu)$ satisfies \ref{H1}, then $\gamma(t)\to\gamma[\infty]$ in the cone topology for any \m\  geodesic $\ga$.
\end{lemma}
\begin{proof}
By \cref{cor:magnetic-geodesic-is-unbounded}, $d(\ga(0),\ga(t))\xrightarrow{t\to\infty}\infty$. From \cref{def:angle-for-pqtilde-end-points}, $\measuredangle_{\ga(0)}(\ga(t),g_v(\infty))=
    \measuredangle_{\ga(0)}(\dot g_v(0),\dot g_{\ga,t}(0))=\measuredangle_{\ga(0)} (v,v_t)\xrightarrow{t\to \infty}0$. This proves convergence in the cone topology.
\end{proof}

\begin{theorem}[{\cite[Corollary, p.~305]{Adachi2}}]\label{thm:pqtilde}
If \ref{H1}, $p\in \M$ and $x\in \Minfty$, then there are a unique \m\  geodesic ray $\ga_{px}$ such that $\ga_{px}(0)=p$ and $\ga_{px}[\infty]=x$, and a unique \m\  geodesic ray $\ga_{xp}$ such that $\ga_{xp}(0)=p$ and $\ga_{xp}[-\infty]=x$; likewise define geodesic rays $g_{px}$ and $g_{xp}$. 
Furthermore, $H^\pm_p\colon T^1_p\M\to\Minfty$, $v\mapsto\gav[\pm\infty]$, are homeomorphisms with respect to the cone topology on $\Minfty$, where $\ga_v$ is from \cref{def:mag flow}.
\end{theorem}
In particular, magnetic geodesics with a common endpoint are disjoint.

We now study the relation between $\gamma(\pm\infty)$ and $\ga[\pm\infty]$ (\cref{def:asymptotic-maggeo-mag-boundary,def:angle-for-pqtilde-end-points}) for a given magnetic geodesic $\ga$. That is, the following two ways of classifying \m\  geodesics are the same (\cref{prop:boundary-bijection}):
\begin{itemize}[label={}]
    \item\cref{def:asymptotic-maggeo-mag-boundary}; asymptoticity of \m\  geodesics,
    \item\cref{def:angle-for-pqtilde-end-points}; asymptoticity of their boundary chords.
\end{itemize} 

We begin with an oft-used fact.
\begin{proposition}\label{prop:stable-mag-jf-same-endpoint}
    If $(\Sigma,\mu)$ satisfies \ref{H1}, then a \m\  variation $\Gamma\colon (a,b)\times \re\to \M$ is stable (resp. unstable; see \cref{def:stable-unstable-magJF}) if and only if \m\  geodesics $\ga_s(\cdot)\dfn \Gamma(s,\cdot)$ for all $s\in (a,b)$ are forward (resp. backward) asymptotic (\cref{def:asymptotic-maggeo-mag-boundary}).
\end{proposition}

\begin{proof}
Without loss of generality, let $(a,b)=(-1,1)$. 
For any $s\in (-1,1)$, denote by $y_s$ the orthogonal component of the \m\  Jacobi field $J_s$ corresponding to $\Gamma$ along $\ga_s(\cdot)$.

We first show that if a \m\ variation is stable, then all \m\ geodesics in this variation are forward asymptotic. Since $\Gamma$ is stable and nontangential  for any $s\in (-1,1)$, $y_s$ is convex and bounded on $[0,\infty)$, and by \cref{lemma:stable-y-is-monotonic}, $t\mapsto\lvert y_s(t)\rvert$ is nonincreasing and nonvanishing. We assume $y_s(t)\geq0$ on $[0,\infty)$ for all $s\in (-1,1)$ without loss of generality. 

The smoothness of $\Gamma$ implies that $s\mapsto y_s(0)$ is continuous on $(-1,1)$. Hence, for $\de\in(0,1)$, there exists a constant $C=C(\de)$ such that $y_s(t)<C$ for all $s\in[0,\de]$ and $t>0$.
Let $\{r_s\colon \re\to\re\}_{s\in(-1,1)}$ be a family of nondecreasing and continuous functions  with $r_0=\mathrm{id}_\re$ such that for any $t\in \re$, $s\mapsto \ga_s(r_s(t))$ is a curve orthogonal to $\ga_s$ for all $s\in(-1,1)$. Then there exists $T>0$ such that for all $t>T$, we have 
\[d(\ga_{\de}(r_{\de}(t)), \ga_{0}(t)\leq\int _0^{\de} y_s(r_s(t)) ds\leq \de\cdot C,\] 
where the second term represents the length of the curve from $\ga_{\de}(r_{\de}(t))$ to $\ga_{0}(t)$ orthogonal to $\ga_s$ for all $s\in[0,\de]$. Therefore, $\ga_{\de}(\infty)=\ga_0(\infty)$. 
Since $\de$ can be arbitrarily chosen in $(0,1)$, following a similar proof for $\de\in(-1,0)$ completes the proof of the forward direction.


Next we show the contrapositive of the other direction, that is, if a \m\ variation $\Gamma\colon (a,b)\times \re\to \M$ is not stable, then there exists $s\in (-1,1)$ such that the \m\  geodesic $\ga_s(\cdot)\dfn \Gamma(s,\cdot)$ is not forward asymptotic to some other \m\ geodesic in $\Gamma$. 
    
    Suppose that $\Gamma$ is not stable. Then there exists $s_0\in(-1,1)$ such that $|y_{s_0}(t)|$ of the \m\  Jacobi field $J_{s_0}$ to $\Gamma$ along $\ga_{s_0}$ is unbounded on $[0,\infty)$. Moreover, since $y_{s_0}(t)$ is convex, $|y_{s_0}(t)|$ is eventually increasing with $|y_{s_0}(t)|\xrightarrow{t\to\infty}\infty$.     Note that there exists sufficiently large $t_0>0$ such that
    \begin{itemize}
    \item  $y_{s_0}(t_0)y_{s_0}'(t_0)> 0$ (otherwise $y_{s_0}'(t)$ would be eventually constant 0, and $y_{s_0}(t)$ would be bounded on $t\in[0,\infty)$, a contradiction), and
    \item $y_{s_0}(t)$ does not change sign on $t\in[t_0,\infty)$.
\end{itemize} 
Since $\Gamma$ is smooth ($C^2$-differentiable),
there exists $\de>0$ such that for any $s\in[s_0,s_0+\de]\subset(-1,1)$, we have $y_{s}(t_0)y_{s}'(t_0)> 0$, $y_s(t_0)y_{s_0}(t_0)>0$ and $y_s'(t_0)y_{s_0}'(t_0)>0$. 
Fix such an $s$ for the remainder of this proof.
These initial conditions for $y_s$ at $t_0$ together with \eqref{MagneticJacobiEquation} imply that 
$|y_s'(t)|$ is increasing on $t\in[t_0,\infty)$, hence (as before) $|y_s(t)|\xrightarrow{t\to\infty}\infty$ and $y_s(t)$ does not change sign on $t\in[t_0,\infty)$. Moreover, for any $C>0$, there exists $T=T(C,s_0,\de)>0$, independent of $s$, such that $|y_s(t)|>C$ for all $t>T$.
Let $\{r_s\colon \re\to\re\}_{s\in [s_0,s_0+\de]}$ be such that $s\mapsto \ga_s(r_s(t))$ is a curve orthogonal to $\Gamma(s,\re)$ and $|y_s(r_s(t))|\geq C$ for any $t>T$ and $s\in [s_0,s_0+\de]$. 
Then for any continuous and nondecreasing function $r\colon \re\to\re$:

Case I: if $r(t)\xrightarrow{t\to\infty}\infty$, then for any $t$ sufficiently large such that $r_s(t)>T$ for all $s\in [s_0,s_0+\de]$,
 \[d(\ga_{s_0}(t),\ga_{s_0+\de}(r(t)))\geq\int _{s_0}^{s_0+\de} |y_s(r_s(t)| ds \geq C\cdot \de.\] 
 
Case II: if $\lim_{t\to\infty}r(t)<\infty$, then 
 $d(\ga_{s_0}(t),\ga_{s_0+\de}(r(t)))\to \infty$ by \cref{cor:magnetic-geodesic-is-unbounded}. 
 
 Hence, $\ga_{s_0}(\infty)\neq \ga_{s_0+\de}(\infty)$, and the proof is complete. The unstable case is similar.
\end{proof}

\begin{proposition}\label{prop:boundary-bijection}
    If $(\Sigma,\mu)$ satisfies \ref{H1}, then two \m\  geodesics $\ga_1$ and $\ga_2$ are asymptotic if and only if they have asymptotic boundary chords. That is, $\ga_1(\pm\infty)=\ga_2(\pm\infty)\in \M^\pm_{\mu}(\infty)$ if and only if $\ga_1[\pm\infty]=\ga_2[\pm\infty]\in\Minfty$.
\end{proposition}

\begin{proof}   
    We prove the ``$+$''-case: $\ga_1(\infty)=\ga_2(\infty)\in \M^+_{\mu}(\infty)$ if and only if $\ga_1[\infty]=\ga_2[\infty]\in\Minfty$.  The other case is similar. Let $v\dfn \dot\ga_1(0)$ and $w\dfn \dot \ga_2(0)$.
    
     If $\ga_1(\re)\cap\ga_2(\re)\neq \emptyset$, then by choosing a suitable time shift by $T_0$, the \m\  geodesic rays induced by $\f_{T_0}v$ and by $w$ will not intersect since $\f$ does not have conjugate points by \cref{lemma:pq-connectivity-and-no-self-intersect}. Without loss of generality, we assume $\ga_1([0,\infty))$ and $\ga_2([0,\infty))$ are as shown in \cref{fig:magnetic_boundary}.
\begin{figure}[hbt]
  \centering
  \begin{tikzpicture}[scale=1.4, >=stealth]
\usetikzlibrary{decorations.markings}

\draw[thin] (0,0) circle (2);

\path[name path=gamma_v]
  (-2,0)
    .. controls (-1,-0.3) and (0.3,-0.3) ..
  (1.7,1.06)
  coordinate[pos=0.3] (Pone)   
  coordinate[pos=1]   (Qone)  
    coordinate[pos=0.85] (I);
\path[name path=gamma_w]
  (-1.9,-0.624)
    .. controls (-0.5,-1.2) and (0.5,-1.2) ..
  (1.9,-0.624)
  coordinate[pos=0.3] (Ptwo)   
  coordinate[pos=1]   (Qtwo)  
  coordinate[pos=0.56] (bifur)
coordinate[pos=0.8] (II);
\draw[thick,
  postaction={decorate},
  decoration={markings,
    mark=at position 0.3 with {
      \draw[->, thick] (0,0) -- (0.8,0) node[below] {$v$};
    }
  }]
  (-2,0)
    .. controls (-1,-0.3) and (0.3,-0.3) ..
  (1.7,1.06);

\draw[thick,
  postaction={decorate},
  decoration={markings,
    mark=at position 0.3 with {
      \draw[->, thick] (0,0) -- (0.8,0) node[below, xshift=-1mm] {$w$};
    }
  }]
  (-1.9,-0.624)
    .. controls (-0.5,-1.2) and (0.5,-1.2) ..
  (1.9,-0.624);

\draw[dashed] (Pone) -- (Qone) node[midway, above] {$g_1$};
\draw[dashed, name path=g_2] (Ptwo) -- (Qtwo) node[midway, above] {$g_2$};

\node[right] at (Qone) {$\ga_1[\infty]$};
\node[right] at (Qtwo) {$\gamma_2[\infty]$};

\fill (Pone) circle (1pt);
\fill (Ptwo) circle (1pt);

\coordinate (Vtip) at (1.97,0.345);  

\draw[dashed, name path=g] (Pone) -- (Vtip) node[midway, above] {$g$};

\coordinate (GammaV) at (0.222, 0.065);
\fill[red] (GammaV) circle (1.2 pt);
\node[below,red] at (GammaV) {$\gamma_1(T)$};

\draw[blue,thick, name path=key_seg] (I) -- (II);

\path[name intersections={of=g and key_seg, by=p1}];
\path[name intersections={of=g_2 and key_seg, by=p2}];
\path[name intersections={of=gamma_w and key_seg, by=p3}];

\fill[blue] (I) circle (1.2pt);
\node[right,blue] at (I) {$\ga_1(t)$};

\fill[blue] (p1) circle (1.2pt);
\node[below right,blue] at (p1) {$g(t_1(t))$};

\fill[blue] (p2) circle (1.2pt);
\node[above right, blue] at (p2) {$g_2(t_2(t))$};

\fill[blue] (p3) circle (1.2pt);
\node[below, blue] at (p3) {$\ga_2(r(t))$};

\fill[red] (bifur) circle (1.2pt);
\node[below,red] at (bifur) {$\ga_2(T)$};

\end{tikzpicture}
\caption{$\ga_1(\infty)=\ga_2(\infty)\implies\ga_1[\infty]=\ga_2[\infty]$}
  \label{fig:magnetic_boundary}
\end{figure}

We first show $\ga_1(\infty)=\ga_2(\infty)\implies\ga_1[\infty]=\ga_2[\infty]$ by contraposition.
Let $g_1, g_2$ be forward boundary chords (\cref{def:angle-for-pqtilde-end-points}) of $\ga_1$ and $\ga_2$ respectively, and let $g\dfn g_{\ga_1,T}$ (\cref{def:chord}) for some $T>0$ be a chord such that $g((0,\infty])\cap g_i((0,\infty])=\emptyset$, $i\in\{1,2\}$. 
Hence for any $t,t'>T$, after choosing a proper order of labeling $\ga_1$ and $\ga_2$, $g$ intersects the geodesic segment connecting $\ga_1(t)$ and $\ga_2(t')$.
    
For a nondecreasing continuous function $r\colon \re\to\re$, if $r(t)\xrightarrow{t\to\infty}T_0<\infty$ for some $T_0\in\re$, then $d(\ga_1(t),\ga_2(r(t)))\xrightarrow{t\to \infty}\infty$ by \cref{cor:magnetic-geodesic-is-unbounded}; if $r(t)\xrightarrow{t\to\infty}\infty$, then there exists $T_1=T_1(T,r)>T$ such that $r(t)\geq T$ for any $t>T_1$. Therefore, we have functions $t_1, t_2$ of $t$ such that for all $t>T_1$,
\begin{equation}\label{eqn:distance_comparison}
        d(\ga_1(t),\ga_2(r(t)))\geq d(g_2(t_2(t)),g(t_1(t)))\geq
\min\{d(g(t_1(t)),g_2([0,\infty))),d(g_2(t_2(t)),g([0,\infty)))\},
\end{equation}
and $\ga_1(t)$, $\ga_2(r(t))$, $g_2(t_2(t))$ and $g(t_1(t))$ lie on a common geodesic; see \cref{fig:magnetic_boundary}.
    Since $g(\infty)\neq g_2(\infty)$, both $d(g(t_1(t)),g_2([0,\infty)))$ and $d(g_2(t_2(t)),g([0,\infty)))$ in \eqref{eqn:distance_comparison} go to infinity. Therefore, $d(\ga_1(t),\ga_2(r(t)))\xrightarrow{t\to \infty}\infty$, which proves $\ga_1(\infty)=\ga_2(\infty)\implies\ga_1[\infty]=\ga_2[\infty]$ 
   
   To show that $\ga_1[\infty]=\ga_2[\infty]\implies \ga_1(\infty)=\ga_2(\infty)$, let $s_0\coloneqq d(\ga_1(0),\ga_2(0))$ and $L\colon (-\eps_0,s_0+\eps_0)\to \M$ be the geodesic segment passing through points $L(0)=\ga_1(0)$ and $L(s_0)=\ga_2(0)$ for some $\eps_0>0$. Define a \m\ variation $\Gamma\colon (-\eps_0,s_0+\eps_0)\to \M$ such that $\Gamma$ is stable, $\Gamma(s,0)=L(s)$, and $\Gamma(0,\cdot)=\ga_1(\cdot)$.
    Define $w_0\coloneqq 
    \frac{\partial \Gamma}{\partial t}(s_0,0)\in T^1_{\ga_2(0)}\M$. We show that $\ga_1(\infty)=\ga_{w_0}(\infty)$ and $w_0=\dot\ga_2(0)$. 

    By \cref{prop:stable-mag-jf-same-endpoint}, all \m\  geodesics  $\ga_s(\cdot)\dfn \Gamma(s,\cdot)$ for $s\in (-\eps_0,s_0+\eps_0)$ are forward asymptotic since $\Gamma$ is stable. Hence $\ga_1(\infty)=\ga_{w_0}(\infty)$. Then by $\ga_1(\infty)=\ga_2(\infty)\implies\ga_1[\infty]=\ga_2[\infty]$ from above, 
    we have $\ga_1[\infty]=\ga_2[\infty]=\ga_{w_0}[\infty]\in \Minfty$. On the other hand, there exists a unique \m\  geodesic from $\pi {w_0}=\ga_2(0)$ to $\ga_{w_0}[\infty]=\ga_2[\infty]$ by \cref{thm:pqtilde}. Therefore, $\ga_1(\infty)=\ga_{w_0}(\infty)=\ga_2(\infty)$ and $w_0=\dot\ga_2(0)$. 
\end{proof}

\begin{lemma}\label{lemma: well-define-boundary-homeo}
    The following is well defined, i.e., independent of $p$:
\begin{equation}
    \tilde h\colon \Minfty \to \M^+_{\mu}(\infty), \quad
    q = \gamma_{pq}[\infty] \mapsto \gamma_{pq}(\infty).
  \label{eq:homeo}
\end{equation}
\end{lemma}
\begin{proof}
    For $q\in \Minfty$, $p_1,p_2\in \M$ we have $\tilde h (q)=\ga_{p_1q}(\infty)=\ga_{p_2q}(\infty)$ by \cref{prop:boundary-bijection}.
\end{proof}

\begin{definition}\label{def:magentic-cone-topology}
The \emph{magnetic cone topology} on $\overline \Sigma^+_{\mu}\dfn \M \cup \M_\mu^+(\infty)$ is  the one that makes \begin{equation}\label{eqn:homeo-extended}
h\colon \om \to \overline \Sigma^+_{\mu},\qquad q\mapsto\begin{cases} q &\text{if }q\in \M,\\ \tilde h(q) &\text{if }q\in \Minfty\end{cases}
\end{equation}
a homeomorphism.
\end{definition}

\begin{remark}\label{rmk:regularity_of_boundary_map}
If $g_p\colon T^1_p\Sigma\to\M^+_\mu(\infty)$,  $v\mapsto\gamma_v(\infty)$, and $p,q\in\M$, then \ref{HA} implies that $g_p^{-1}\circ g_q$ is nearly $C^2$ \cite{HurderKatok}. It is not clear what regularity of $g_p^{-1}\circ g_q$ is implied by \ref{H1}. 
\end{remark}
\cref{prop:boundary-bijection} and \cref{lemma: well-define-boundary-homeo} give the following.
\begin{proposition}\label{prop:boundary-homeomorphic}
If \ref{H1}, then $\overline \Sigma^+_{\mu}$ with the magnetic cone topology in \cref{def:magentic-cone-topology} is homeomorphic to $\om$ with the cone topology in \cref{ConeTopology}.
    In particular, the magnetic boundary $\M^+_\mu(\infty)$ (\cref{def:asymptotic-maggeo-mag-boundary}) is homeomorphic to $\Minfty$ in a natural way.
\end{proposition}

From now on, when $(\Sigma,\mu)$ satisfies \ref{H1}, we will no longer distinguish  $\overline \Sigma^{+,-}_{\mu}$ and $\om$, $\M_\mu(\infty)$ ($\M_\mu^{\pm}(\infty)$) and $\Minfty$, or $\ga(\infty)$ and $\ga[\infty]$ for any \m\  geodesic $\ga$ on $\M$.

\subsection{Characterization of magnetic flatness}\label{sec:mag-flat}
We characterize $\mu$-magnetically flat surfaces as those where the forward and backward endpoints of a \m\ geodesic coincide. As a consequence, if \ref{H2}, then any \m\ geodesic on $\Sigma$ has distinct forward and backward endpoints (\cref{cor:h2-diff-endpoints}).


\begin{definition}\label{def:gamma-jacobi-field}
Let $\ga$ be a \m\  geodesic. Denote by $\Delta(t,s)\colon\re\times\re_{\geq0}\to\M$ a variation of geodesics orthogonal to $\gamma$, where $\Delta(t,\cdot)$ is a geodesic ray with $\Delta(t,0)=\ga(t)$ and $\frac{\partial \Delta(t,0)}{\partial s} =i \dot \ga (t)$ for all $t$, as suggested in \cref{fig:orth_variation}.

$V_t(s)\coloneqq \frac{\partial \Delta(t,s)}{\partial t}$ denotes the orthogonal (Riemannian) Jacobi field along $\Delta(t,\cdot)$.

\begin{figure}[hbt]
  \centering
  \begin{tikzpicture}[>=stealth,scale=1.5]
\usetikzlibrary{decorations.markings}

\draw[thick,
  postaction={decorate},
  decoration={markings,
    mark=at position 0.6 with {\arrow{stealth}}
  }]
  (-2,0)
    .. controls (-1.2,-1.1) and (1.2,-1.1) ..
  (2,0);

\coordinate (q) at (-1.025,-0.619);
\coordinate (m) at (0,-0.825);
\coordinate (p) at (1.025,-0.619);

\node[below right] at (p) {$\gamma(t)$};


\draw[blue,dashed] (m) -- ++(90:1.5);

\draw[blue,dashed] (p) -- ++(120:1.5);

\draw[blue,dashed] (q) -- ++(60:1.5);


\draw[blue,thick]
  (m) ++(0:0.1) -- ++(90:0.1) -- ++(180:0.1);

\draw[blue,thick]
  (p) ++(2:0.1) -- ++(120:0.1) -- ++(210:0.1);

\draw[blue,thick]
  (q) ++(170:0.1) -- ++(60:0.1) -- ++(-30:0.1);

\node[blue] at (1.3,0.5) {$\Delta(t,s)$};

\draw[blue,->,thick] (-0.25,-1.2) -- (0.25,-1.2);
\node[blue,below] at (0,-1.2) {$t$};

\draw[blue,->,thick] (-0.2,-0.6) -- (-0.2,0.1);
\node[blue,left] at (-0.2,-0.5) {$s$};

\end{tikzpicture}
  \caption{Variation $\Delta(t,s)$ in \cref{def:gamma-jacobi-field}}
  \label{fig:orth_variation}
\end{figure}
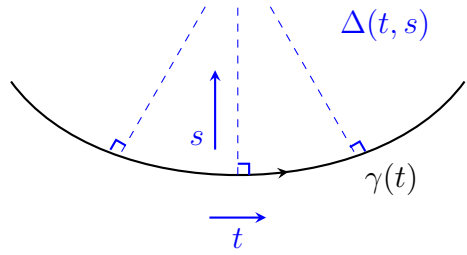
\end{definition}

\begin{lemma}[{\cite[p.~229]{Adachi1}}]\label{lemma:gamma-jf-disjoint}
With \ref{H1} and notations from \cref{def:gamma-jacobi-field}, we have
    \begin{enumerate}
        \item $V_t(s)\neq0$ for any $(t,s)\in \re\times \re_{\geq0}$;
        \item$\Delta$ is injective:  for $t_1\neq t_2$, geodesics $\Delta(t_1,\cdot)$ and $\Delta(t_2,\cdot)$ do not intersect on $\M$.
    \end{enumerate}

\end{lemma}

\begin{corollary} [{\cite[Remark(1), p.~230]{Adachi1}}]\label{lemma:single-infinity-point-gamma-jf}
If \ref{H1}, then for any \m\  geodesic $\ga$, $\gamma(\infty)=\gamma(-\infty)$ if and only if the orthogonal geodesics $\Delta(t,s)$ from  \cref{def:gamma-jacobi-field} go to the same endpoint on $\Minfty $ for all $t$ as $s\to\infty$.
\end{corollary}

\begin{remark}\label{remark:horoball}
This means that $\ga$ is a curve orthogonal to a family of geodesics with a common boundary point and hence a horocycle of $\Sigma$.
\end{remark}


Suppose \ref{H1}. Given a \m\ geodesic $\ga$ and $t \in \re$, let $g_t$ be the forward boundary chord of the \m\ geodesic $\ga_t(s)\dfn \ga(t+s)$; i.e., $g_t$ is the geodesic with $g_t(0)=\ga(t)$ and $g_t(\infty)=\ga(\infty)\in\Minfty$. 
Define \begin{equation}\label{eqn:angle_eta}
    \eta(t)\coloneqq \measuredangle_{\ga(t)}(\dot\ga(t),\dot g_t(0)).    
\end{equation} By \cref{Lemma:angle_<_pi/2}, $\eta(t)\leq \pi/2$, and we conjecture that equality only happens in the \m ally flat case.

\begin{conjecture}\label{conjecture}
If \ref{H1} and $\eta(T_0)=\pi/2$ for some $T_0$, then $\eta(t)=\pi/2$ for all $t\in\re$, so $\Sigma$ is $\mu$-magnetically flat by \cref{prop:h2-has-no-closed-orbit}.
\end{conjecture}


We can prove a partial (one-sided) result:
\begin{lemma}\label{lemma:infinite-chord-angle<pi/2}
If $(\Sigma,\mu)$ satisfies \ref{H1} 
and $\eta(T_0)=\pi/2$ for some $T_0\in\re$, then $\eta(T)=\pi/2$ for all $T\geq T_0$.
\end{lemma}

\begin{proof}
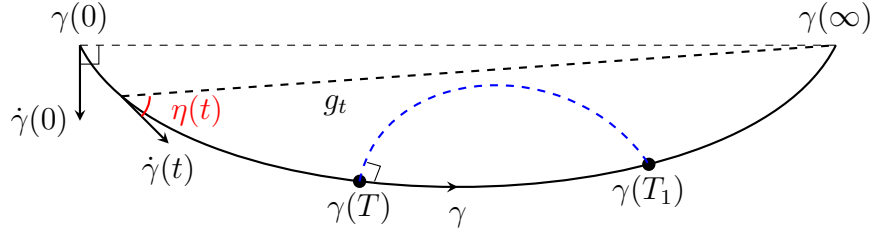
\begin{figure}[hbt]
  \centering
  \begin{tikzpicture}[>=stealth, scale=2.5]

\draw[dashed] (-2,0) -- (2,0);

\draw[thick,
      postaction={decorate},
     decoration={markings, mark=at position 0.5 with {\arrow{stealth}}}]
      (-2,0) .. controls (-1.5,-1) and (1.5,-1) .. (2,0)

coordinate[pos=0.4] (point_s)
coordinate[pos=0.7] (point_t)
coordinate[pos=0.1] (point_u);
\node[below] at (0,-0.8) {$\gamma$};
\node[above] at (-2,0) {$\gamma(0)$};
\node[above] at (2,0) {$\gamma(\infty)$};
\fill (point_s) circle (1pt);
\fill (point_t) circle (1pt);

\node[below]  at (point_s) {$\gamma(T)$};
\node[below] at (point_t) {$\gamma(T_1)$};

\draw[dashed, thick] (point_u) -- (2,0) node[pos=0.3, yshift=-6mm, above] {$g_t$};
\draw[->, thick] (point_u) -- ++(0.25,-0.25) coordinate (tan_u);
\node[below] at (tan_u) {$\dot\gamma(t)$};
\draw[dashed,blue, thick]
  (point_s)
    .. controls (-0.4,-0.1) and (0.6,-0.02) ..
  (point_t);

\draw[->,thick] (-2,0) -- (-2,-0.4) node[left]{$\dot\ga(0)$};
\draw (-2,0) ++(0.1,0) -- ++(0,-0.1) -- ++(-0.1,0);
\draw (point_s) ++(0.04,0.1) -- ++(0.07,-0.02) -- ++(-0.031,-0.08);

\draw[thick,red]
  (point_u)
    ++(0.15,0)
      arc[start angle=0, end angle=-45, radius=0.15];

\node[red] at (point_u) [xshift=10mm, yshift=-2mm] {$\eta(t)$};

\end{tikzpicture}
  \caption{Geodesic connecting $\gamma(T)$ and $\gamma(T_1)$ in  \cref{lemma:infinite-chord-angle<pi/2}}
  \label{fig:connect_geo}
\end{figure}

After reparametrization, assume $T_0=0$. Consider the variation $\Delta$ of geodesics orthogonal to $\ga$ from \cref{def:gamma-jacobi-field}. We claim that $\lim_{s\to\infty}\Delta(T,s)=\ga(\infty)$ for all $T\geq 0$, and this implies $\eta(T)=\pi/2$ for all $T\geq 0$ by uniqueness of the geodesic from $\ga(T)$ to $\ga(\infty)$.

Fix $T>0$. By \cref{lemma:gamma-jf-disjoint}, we have either $\lim_{s\to\infty}\Delta(T,s)=\ga(\infty)$ or $\Delta(T,(0,\infty))\cap\ga([T,\infty))= \ga(T_1)$ for some $T_1=T_1(T)>T$. In the latter case, $\ga([T,T_1])$ subtends an angle of $\pi/2$ with its chord at the point $\ga(T)$ as suggested in \cref{fig:connect_geo}, contrary to \cref{Lemma:angle_<_pi/2}. Therefore, $\lim_{s\to\infty}\Delta(T,s)=\ga(\infty)$ for all $T\geq0$.
\end{proof}



\begin{theorem}\label{prop:h2-has-no-closed-orbit}
If \ref{H1}, then the following are equivalent.
\begin{enumerate}
    \item $\Sigma$ is $\mu$-magnetically flat, i.e., not \ref{H2}; \label{l:1}
    \item $\ga(\infty)=\ga(-\infty)$ for any \m\  geodesic $\ga$ on $\M$;\label{l:2}
    \item $\ga(\infty)=\ga(-\infty)$ for some \m\  geodesic $\ga$ on $\M$;\label{l:3}
    \item there exists a \m\ geodesic $\ga$ such that $\eta(t)=\pi/2$ for all $t\in\re$, where $\eta(t)$ is from \eqref{eqn:angle_eta}.\label{l:4}
\end{enumerate}
\end{theorem}

\begin{proof}
$\eqref{l:1}\implies\eqref{l:2}\implies\eqref{l:3}\Leftrightarrow\eqref{l:4}$ is clear (the latter by \cref{lemma:single-infinity-point-gamma-jf}). We now show $\eqref{l:3}\implies\eqref{l:1}$. 

If (\ref{l:3}) holds, let $x\dfn\ga(\infty)=\ga(-\infty)$ and $\Gamma\colon\mathbb R\times(-\infty,0]\to\tilde\Sigma$ the \m\ variation $(s,t)\mapsto\gamma_{xg_{x\gamma(0)}(t)}(s)$. Its range $R$ contains a fundamental domain since it is a horoball by \eqref{l:4}. Thus, \eqref{l:1} follows if $R$ is magnetically flat, and by \cref{lemma:sing-vector-has-constant-y} and \eqref{MagneticJacobiEquation}, this follows once we show that $\Gamma$ is parallel---for which it is in turn sufficient   that $\gamma_{xg_{x\gamma(0)}(t)}(+\infty)=x$ for all $t<0$ (\cref{prop:stable-mag-jf-same-endpoint,lem:parallel-iff-stable-and-unstable}). 

If $\gamma_{xg_{x\gamma(0)}(t)}(+\infty)\neq x$ for any $t<0$, then $\gamma_{xg_{x\gamma(0)}(t)}$ and $\gamma$ have a common point $q$ (see \cref{fig:closed-mag-geo}) and hence are 
\m\  geodesics from $x\in\Minfty$ to $q\in \M$. Since $\gamma$ does not go through $g_{x\gamma(0)}(t)$, they are distinct, 
contrary to \cref{thm:pqtilde}.
\begin{figure}[hbt]
  \centering
\begin{tikzpicture}[scale=3,>=stealth]

\draw[thick,postaction={decorate},
      decoration={markings}] (0,0) circle (1);

\fill (1,0) circle (0.8pt);
\node[below right] at (1,0) {$x=\ga(\pm\infty)$};

\fill (-1,0) circle (0.8pt);
\node[left] at (-1,0) {$g_{x\ga(0)}(\infty)$};

\draw[dashed,red,thick] (-1,0) -- (1,0);

\fill[blue] (-0.4,0) circle (0.8pt);
\node[below,blue] at (-0.4,0) {$g_{x\ga(0)}(t)$};
\node[above left] at (-0.8,0.7) {$\M$};
\draw (-0.4,0) ++(0.05,0) -- ++(0,0.05) -- ++(-0.04,0);
\node[red] at (0.3,-0.1) {$g_{x\ga(0)}$};

\draw[thick,
      postaction={decorate},
      decoration={markings, mark=at position 0.5 with {\arrow{stealth}}}]
  (1,0)
    .. controls (1,0.9) and (-0.8,0.9)
    .. (-0.85,0);

\draw (-0.85,0) ++(0.05,0) -- ++(0,0.05) -- ++(-0.05,0);

\draw[thick]
  (1,0)
    .. controls (1,-0.9) and (-0.8,-0.9)
    .. (-0.85,0);




\draw[blue,thick,postaction={decorate},
      decoration={markings, mark=at position 0.5 with {\arrow{stealth}}}]
      (1,0) arc (0:180:0.7 and 0.3);
\draw[blue,thick]
 (-0.4,0)
    .. controls (-0.35,-0.4) and (0.2,-0.6)
    .. (0.89,-0.456);
\fill[blue] (0.75,-0.485) circle (0.8pt);
\node[above,blue] at (0.75,-0.485) {$q$};
\node[below,blue] at (0.3,0.273) {$\ga_{xg_{x\ga(0)}(t)}$};
\node[above right] at (-0.85,0.02) {$\ga(0)$};
\fill (-0.85,0) circle (0.8pt);
\node[above right] at (0.1,0.65) {$\ga$};
\fill[blue] (0.89,-0.456) circle (0.8pt);
\node[right,blue] at (0.89,-0.456) {$\ga_{xg_{x\ga(0)}(t)}(+\infty)$};

\end{tikzpicture}
  \caption{\cref{prop:h2-has-no-closed-orbit} $\eqref{l:3} \implies \eqref{l:1}$}
  \label{fig:closed-mag-geo}
\end{figure}
\end{proof}

\begin{corollary}\label{cor:h2-diff-endpoints}
   If \ref{H2}, then $\ga(\infty)\neq \ga(-\infty)$ for any \m\  geodesic $\ga$.
\end{corollary}

Assuming \cref{conjecture}, the following is equivalent to the statements in \cref{prop:h2-has-no-closed-orbit}:
\begin{enumerate}\label{conj}
    \item[(5)]\label{l:5} There exists a \m\  geodesic $\ga$ such that $\eta(t)=\pi/2$ for some $t\in\re$.
\end{enumerate}


\subsection{Magnetic visibility and orbit-equivalence}\label{ssection:magnetic-visibility} 
By \cref{cor:h2-diff-endpoints}, the endpoints on $\Minfty$ of a hyperbolic \m\ geodesic are distinct. We now ask whether every pair of distinct points on $\Minfty$ arises as the endpoints of some \m\ geodesic:
\begin{definition}[Magnetic visibility]\label{def:magnetic-visibility-property}
    A surface $\Sigma$ is said to have the \emph{\m\  visibility property} if for any two distinct points $x,y\in\Minfty$, there exists a \m\  geodesic $\ga$ such that $\ga(-\infty)=x$ and $\ga(\infty)=y$, where $\ga(\pm\infty)$ are from \cref{def:asymptotic-maggeo-mag-boundary} (or by \cref{prop:boundary-homeomorphic} equivalently, \cref{def:angle-for-pqtilde-end-points}).

\end{definition}
If $\kmu<0$, then \m\ geodesics are quasigeodesics, so the Hausdorff distance between a \m\ geodesic and its boundary chord is finite; this yields visibility:
\begin{proposition}[{\cite{Grognet}}]
    If \ref{H4}, then for any two distinct points $x, y\in \om$, there exists a unique \m\  geodesic $\ga$ such that $x=\ga(t_0)$ and $y=\ga(t_1)$ for some $t_0,t_1\in[-\infty,\infty]$.
\end{proposition}
With just \ref{H1}, magnetic visibility can fail---but only in the magnetically flat case:

\begin{theorem}[Magnetic visibility]\label{thm:ptildeqtilde-magnetic-visibility-property}
    If \ref{H2}, then $\Sigma$ has the \m\  visibility property.
\end{theorem}
\begin{proof}
When $\mu=0$, this is \cref{prop:connectivity-geoflow-MtoMbdry}\eqref{lemma:connectivity_geoflow_nega_curv}. We assume $\mu>0$, and the proof for $\mu<0$ is similar. For two distinct points $x,y$ on $\M(\infty)$,
denote by $\widehat{xy}$ the open \emph{clockwise} arc in $\M(\infty)$ from $x$ to $y$. Then $\big\{\{x\}, \widehat{xy}, \{y\},\widehat{yx}\big\}$ is a partition of $\Minfty$. 

For a geodesic $g$ with $g(-\infty)=x\neq y=g(\infty)$ and  $n\in\mathbb{N}_0$, set $p_n\coloneqq g(n)$ and denote by 
\begin{enumerate}
        \item $l$ a geodesic orthogonal to $g$ at $p_0$ with $l(0)=p_0$ and $\xi\coloneqq l(\infty)\in\widehat{yx}$;
        \item $\ga_n$ the reparametrization of the \m\  geodesic $\ga_{xp_n}$ from $x$ to $p_n$ such that $\ga_n(-\infty)=x$ and $\ga_n(0)=\ga_n(\re)\cap l([0,\infty))= l(t_n)$ for some $t_n\geq0$ (by \cref{thm:pqtilde}, $\{t_n\}$ is increasing);
        \item $y_n\coloneqq \ga_n(\infty)\in \widehat{xy}$ (then $y_n\neq x$ by \cref{cor:h2-diff-endpoints}, $y_n\notin\widehat{yx}\cup \{y\}$ by \cref{thm:pqtilde} and \cref{rmk:general-behaviors-easy-observation}(\ref{LeftRightTangency}));
        \item $g_n$ the boundary chord of $\ga_n$ (see \cref{def:angle-for-pqtilde-end-points});
        \item $\tilde g$ an axis of some axial isometry $\psi$ such that $\tilde g(-\infty)\in \widehat{\xi x}$ and $\tilde g(\infty)\in \widehat{y\xi}$.
    \end{enumerate}

    \begin{figure}[bt]
  \centering
  \begin{tikzpicture}[scale=4]
\draw[thick] (0,0) circle (1);
\fill (-1,0) circle (0pt);
\node[above left] at (-1,0) {$g(-\infty)=x$};
\node[below] at (-0.5,0) {$g$};

\fill (1,0) circle (0pt);
\node[above right] at (1,0) {$y=g(\infty)$};

\draw[dashed] (-1,0) -- (1,0);
\fill (0,0) circle (0.5pt);
\node[above] at (-0.1,0) {$l(0)=p_0$};
\fill (0,-1) circle (0pt);
\node[below] at (0,-1) {$\xi \dfn l(\infty)$};

\draw[dashed,blue,thick] (0,0) -- (0,-1);
\draw (0,0) ++(0.05,0) -- ++(0,-0.05) -- ++(-0.05,0);

\node[right,blue] at (0,-0.7) {$l$};


\draw[thick]
  (-1,0)
    .. controls (-0.6,-0.4) and (-0.2,-0.4)
    .. (0.2,0)
    .. controls (0.45,0.35) and (0.55,0.6)
    .. (0.6,0.8);
\fill (0, -0.163) circle (0.5 pt);
\node[right] at (0, -0.18) {$\ga_1(0)=l(t_1)$};
\fill (0.2, 0) circle (0.5 pt);
\node[above right] at (0.25, 0) {$p_1$};
\fill (0, -0.163) circle (0.5 pt);
\node[right] at (0, -0.18) {$\ga_1(0)=l(t_1)$};
\node[above right] at (0.6,0.8) {$y_1$};

\draw[thick]
  (-1,0)
    .. controls (-0.6,-0.55) and (-0.1,-0.55)
    .. (0.45,0)
    .. controls (0.70,0.35) and (0.7,0.45)
    .. (0.75,0.661);
\fill (0, -0.33) circle (0.5 pt);
\node[right] at (0, -0.33) {$\ga_2(0)=l(t_2)$};
\fill (0.45, 0) circle (0.5 pt);
\node[above right] at (0.49, 0) {$p_2$};
\node[above right] at (0.75,0.661) {$y_2$};

\draw[thick]
  (-1,0)
    .. controls (-0.63,-0.75) and (0.15,-0.75)
    .. (0.7,0)
    .. controls (0.85,0.23) and (0.88,0.42)
    .. (0.89,0.456);
\fill (0, -0.53) circle (0.5 pt);
\node[right] at (0, -0.55) {$\ga_n(0)=l(t_n)$};
\fill (0.7, 0) circle (0.5 pt);
\node[above right] at (0.775, 0) {$p_n$};
\node[above right] at (0.89,0.456) {$y_n$};

\draw[dashed](-1,0) -- (0.89,0.456);
\node[above] at (0,0.256) {$g_n$};

\draw[thick]
  (-1,0)
    .. controls (-0.8,-0.6) and (-0.2,-0.9)
    .. (0,-0.9)
    .. controls (0.2,-0.9) and (0.6,-0.6)
    .. (0.9,0)
    .. controls (0.95,0.1) and (0.96,0.2)
    .. (0.97,0.243);
\draw[red,thick]
  (-0.866,-0.5)
    .. controls (-0.3,-0.9) and (0.3,-0.9)
    .. (0.866,-0.5);
\node[left] at (-0.866,-0.5) {$\eta(-\infty)$};
\node[right] at (0.866,-0.5) {$\eta(\infty)$};
\node[left,red,thick] at (-0.566,-0.7) {$\eta$};

\fill[red] (0.31, -0.76) circle (0.5 pt);
\node[below,red] at (0.31, -0.76) {$q_2$};

\fill[red] (-0.40, -0.74) circle (0.5 pt);
\node[below,red] at (-0.40, -0.74) {$q_1$};

\end{tikzpicture}
  \caption{Convergence of $t_n$}
  \label{fig:visibility-I}
\end{figure}
    \cref{thm:pqtilde} gives a \m\  geodesic $\eta_0$ such that $\eta_0(-\infty)=\tilde g(-\infty)$, and $\eta_0(\infty)\neq\tilde g(-\infty)$ by \cref{cor:h2-diff-endpoints}. By 
    \cref{Lemma:connectivity-by-geodesic}(\ref{lemma:connectivity_geoflow_nega_curv}) and \cref{prop:endpoints-converge-to-axis-endpoints}, for some sufficiently large $n$, the \m\ geodesic $\eta\dfn \psi^n(\eta_0)$ is such that $\eta(-\infty)\in \widehat{\xi x}$ and $\eta(\infty)\in \widehat{y \xi}$.
The intersection of $\eta$ and $l$ is a point $l(t_*)$ with $t_*\ge t_n$ for all $n$ because otherwise, $\eta(\re)\cap \ga_n(\re)\supset\{q_1,q_2\}$ for some sufficiently large $n$ as indicated in \cref{fig:visibility-I}, contrary to \cref{lemma:pq-connectivity-and-no-self-intersect}. 

The increasing sequence $t_n\le t_*$ converges (monotone convergence theorem), so there is a \m\  geodesic $\ga\coloneqq\ga_{xp}$ with $\ga(0)=p\dfn\lim _{n\to\infty}l(t_n)\in \M$ and $\ga(-\infty)=x$. 

It remains to show that $\ga(\infty)=y$, as desired, and we prove this by ruling out all other possibilities. First, $\ga(\infty)\neq x$ by \cref{cor:h2-diff-endpoints}. Next, if $\ga(\infty)\in \widehat{xy}$ (\cref{fig:visibility-II}), then $\ga$ intersects $g(\re)$ at $g(T)$ for some $T\in\re$, so $\ga$ intersects $\ga_n$ at some $q\in\ga_n(\re)$ for any $n>T$, contrary to \cref{thm:pqtilde}. Finally, we prove that $y^*\dfn\ga(\infty)\notin \widehat{yx}$, and therefore, $\ga(\infty)=y$. 
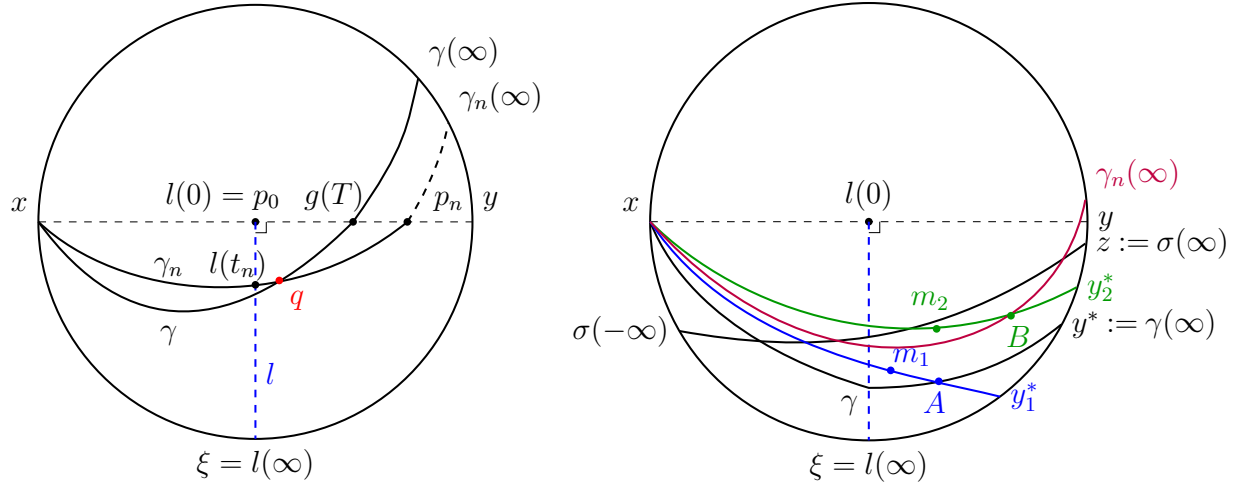
\begin{figure}[hbt]
  \centering
  \begin{minipage}{0.45\linewidth}
    \centering
    \resizebox{\linewidth}{!}{\begin{tikzpicture}[scale=3]
\draw[thick] (0,0) circle (1);
\fill (-1,0) circle (0pt);
\node[above left] at (-1,0) {$x$};

\fill (1,0) circle (0pt);
\node[above right] at (1,0) {$y$};

\draw[dashed] (-1,0) -- (1,0);
\fill (0,0) circle (0.5pt);
\node[above] at (-0.15,0) {$l(0)=p_0$};
\fill (0,-1) circle (0pt);
\node[below] at (0,-1) {$\xi=l(\infty)$};

\draw[dashed,blue,thick] (0,0) -- (0,-1);
\draw (0,0) ++(0.05,0) -- ++(0,-0.05) -- ++(-0.05,0);

\node[right,blue] at (0,-0.7) {$l$};


\draw[thick]
  (-1,0)
    .. controls (-0.6,-0.55) and (-0.1,-0.55)
    .. (0.45,0)
    .. controls (0.70,0.35) and (0.7,0.45)
    .. (0.75,0.661);

\node[below] at (-0.4, -0.43) {$\ga$};
\fill (0.45, 0) circle (0.5 pt);
\node[above left] at (0.55, 0) {$g(T)$};
\node[above right] at (0.75,0.661) {$\ga(\infty)$};

\draw[thick]
  (-1,0)
    .. controls (-0.63,-0.35) and (0.15,-0.45)
    .. (0.7,0);
\draw[thick,dashed]
 (0.7,0)
    .. controls (0.85,0.23) and (0.88,0.42)
    .. (0.89,0.456);
\node at (-0.4,-0.21) {$\ga_n$};
\fill (0, -0.29) circle (0.5 pt);
\node at (-0.08, -0.2) {$l(t_n)$};
    
\fill[red] (0.11, -0.27) circle (0.5 pt);
\node[red,below right] at (0.11, -0.27) {$q$};
\fill (0.7, 0) circle (0.5 pt);
\node[above right] at (0.775, 0) {$p_n$};
\node[above right] at (0.89,0.456) {$\ga_n(\infty)$};

\end{tikzpicture}}
  \end{minipage}\hfill
  \begin{minipage}{0.55\linewidth}
    \centering
    \resizebox{\linewidth}{!}{    \begin{tikzpicture}[scale=3]
\draw[thick] (0,0) circle (1);
\fill (-1,0) circle (0pt);
\node[above left] at (-1,0) {$x$};

\fill (1,0) circle (0pt);
\node[right] at (1,0) {$y$};

\draw[dashed] (-1,0) -- (1,0);
\fill (0,0) circle (0.5pt);
\node[above] at (0,0) {$l(0)$};
\fill (0,-1) circle (0pt);
\node[below] at (0,-1) {$\xi=l(\infty)$};

\draw[dashed,blue,thick] (0,0) -- (0,-1);
\draw (0,0) ++(0.05,0) -- ++(0,-0.05) -- ++(-0.05,0);

\node[below left] at (0,-0.75) {$\ga$};


\draw[thick]
  (-0.866,-0.5)
    .. controls (-0.3,-0.6) and (0.3,-0.6)
    .. (0.99,-0.1);
\node[left] at (-0.866,-0.5) {$\sigma(-\infty)$};
\node[right] at (0.99,-0.1) {$z\dfn \sigma(\infty)$};

\draw[thick]
  (-1,0)
    .. controls (-0.8,-0.5) and (-0.2,-0.7)
    .. (0,-0.76)
    .. controls (0.3,-0.76) and (0.6,-0.7)
    .. (0.88,-0.47);
\node[right] at (0.88,-0.47) {$y^*\dfn \ga(\infty)$};

\fill[green!60!black] (0.31, -0.49) circle (0.5 pt);
\node[green!60!black,above] at (0.28, -0.45) {$m_2$};

\draw[green!60!black,thick]
  (-1,0)
    .. controls (-0.6,-0.4) and (0.2,-0.7)
    .. (0.95,-0.3);
\node[green!60!black,right] at (0.95,-0.3) {$y^*_2$};

\draw[blue,thick]
  (-1,0)
    .. controls (-0.6,-0.6) and (0.2,-0.7)
    .. (0.6, -0.8);

\node[blue, right] at (0.6,-0.8) {$y^*_1$};
    
\fill[blue] (0.32, -0.73) circle (0.5 pt);
\node[blue,below] at (0.3, -0.73) {$A$};

\fill[blue] (0.1, -0.68) circle (0.5 pt);
\node[blue,right] at (0.06, -0.62) {$m_1$};

\node[purple,above right] at (0.99,0.1) {$\ga_n(\infty)$};



\draw[purple,thick]
  (-1,0)
    .. controls (-0.2,-0.9) and (0.85,-0.65)
    .. (0.99,0.1);

\fill[green!60!black] (0.65, -0.43) circle (0.5 pt);
\node[green!60!black,below right] at (0.58, -0.43) {$B$};

\end{tikzpicture}}
  \end{minipage}
  \caption{Endpoints of $\ga$ and possibilities for $\ga_{xm}(\infty)$}
  \label{fig:visibility-II}\label{fig:visibility-III}
\end{figure}

Suppose $y^*\in \widehat{yx}$; we will show that this contradicts \cref{thm:pqtilde}.
By \cref{prop:geo-periodic}, there exists an axis $g_1$ of some axial isometry $\psi$ such that  $g_1(-\infty)\in \widehat{y^*x}$ and $g_1(\infty)\in \widehat{yy^*}$. Given a \m\ geodesic $\sigma_0$ such that $\sigma_0(-\infty)=g_1(-\infty)$, by \cref{thm:pqtilde} and \cref{prop:endpoints-converge-to-axis-endpoints}, there exists a \m\ geodesic $\sigma\dfn \psi^n(\sigma_0)$ for some sufficiently large $n$ such that $\sigma(-\infty)\in \widehat{y^*x}$ and $z\dfn \sigma(\infty)\in \widehat{yy^*}$ as suggested in \cref{fig:visibility-III}.  
    
    Pick a point $m$ in the open subset of $\M$ enclosed by $\sigma$, $\ga$ and $\widehat{zy^*}$ and $t'>0$ such that $l(t')= \ga_{xm}(\re)\cap l([0,\infty))$. Then $t'<\lim_{n\to\infty}t_n$ since otherwise, $\ga_{xm}$ and $\ga$ intersect on both $\M$ and $\Minfty$,  contrary to \cref{thm:pqtilde}. We now consider $\ga_{xm}(\infty)$:
    \begin{enumerate}
        \item if $m=m_1$ is such that $\ga_{xm_1}(\infty)\eqqcolon y^*_1\in \widehat{y^*x}$, then two \m\ geodesics $\ga_{xm_1}\neq \ga$ go from $x$ to some point $A\in\M$, which contradicts \cref{thm:pqtilde};
        \item if $m=m_2$ is such that $\ga_{xm_2}(\infty)\eqqcolon y^*_2\in \{y^*\}\cup\widehat{yy^*}$, then for $n$ such that $t_n>t'$, two \m\ geodesics $\ga_{xm_2}\neq \ga_n$ go from $x$ to some point $B\in\M$, again contrary to \cref{thm:pqtilde}.
    \end{enumerate}
    Hence $y^*\notin \widehat{yx}$, and as noted above, this implies $\ga(\infty)=y$, so the proof is now complete.  
\end{proof}


Magnetic visibility (\cref{thm:ptildeqtilde-magnetic-visibility-property}) guarantees that every geodesic is the boundary chord of some hyperbolic \m\ geodesic. The following observation will be used in the proof of orbit-equivalence (\cref{prop:orbit-equivlence-partial}):

\begin{lemma}\label{lemma:ortho-proj-any-point}
If \ref{H2} and $\ga$
    is a \m\ geodesic with boundary chord $g$, then for any $p\in g(\re)$, the geodesic $g_\perp$ with  $g_\perp\perp g$ at $p$ intersects $\ga$ on $\M$. 
\end{lemma}
\begin{proof}
    By the uniqueness of the geodesic with endpoints $x\dfn g(\infty)\neq g(-\infty)\nfd y$ and since $g(\re)\cap g_\perp(\re)=\{p\}$, we have $g_\perp(\infty)\in\widehat{xy}$ and $g_\perp(-\infty)\in\widehat{yx}$ without loss of generality. Therefore, $g_\perp$ intersects $\ga$ on $\M$. 
\end{proof}

\cref{lemma:pq-connectivity-and-no-self-intersect,thm:pqtilde,prop:h2-has-no-closed-orbit,thm:ptildeqtilde-magnetic-visibility-property} plus \cref{cor:magnetic-flat-strip-has-0-magnetic-curvature-and-magnetic-flat-strip-only-singular}\eqref{Cor:magnetic-flat-strip-only-singular} below 
give
\begin{theorem}[Visibility]\label{thm:connectivity-combo}
If \ref{H1}, then 
    \begin{enumerate}
        \item for any $p\in \om$ and $q\in \M$,  up to reparametrization, there exist a unique \m\  geodesic from $p$ to $q$ and a unique \m\  geodesic from $q$ to $p$;\label{l:nonflatstrip}
        \item any two \m\  geodesics intersect at most once on $\M$.
    \end{enumerate}
    Furthermore, \ref{H2} if and only if
    \begin{enumerate}[resume]
        \item $\ga(\infty)\neq\ga(-\infty)$ for any \m\  geodesic $\ga$;
        \item for any two distinct points $p,q$ on $\Minfty$, there exists a \m\  geodesic $\ga_{pq}$ from $p$ to $q$. Moreover, if $\dot\ga_{pq}(t)\in  \Reg$ for some $t\in\re$, then $\ga_{pq}$ is unique up to reparametrization.\label{l:flatstrip}\label{item4:connectivity}
    \end{enumerate}      
\end{theorem}

Note that unlike in (\ref{l:nonflatstrip}), uniqueness fails in (\ref{l:flatstrip})---whenever there are magnetic flat strips---and requires \ref{H3}; see \cref{section:mag-axes-and-flat-strip}. 
On the other hand, \ref{H3} implies that a magnetic flow is orbit-equivalent to the geodesic flow of $\Sigma$ (\cref{prop:orbit-equivlence-partial,thm:orbit-equivalence}).

\begin{definition}[Orbit-equivalence, {\cite[Definition 1.3.22]{Fisher-Hasselblatt}}]\label{def:orbit-equiv}
  A flow $\Psi$ on $X$ is said to be an \emph{orbit factor} of a flow $\Phi$ on $Y$ if there exists a continuous surjection $h\colon Y \to X$ that sends orbits of $\Phi$ to orbits of $\Psi$. 
  If $h$ is a homeomorphism, then the flows are said to be \emph{orbit-equivalent}.
\end{definition}

Grognet establishes orbit-equivalence of uniformly hyperbolic magnetic flows (for a nonconstant magnetic field on a manifold) to the geodesic flow of a manifold with negative curvature; specialized to surfaces with constant magnetic field, this reads as follows.
\begin{proposition}[{\cite[Theorem 3.1]{Grognet}}]\label{prop:grognet_oe}
If \ref{H4}, then  the \m\  flow of $\Sigma$ is orbit-equivalent to the geodesic flow of $\Sigma$. 
\end{proposition}

{\color{black}
For surfaces, we extend Grognet's argument to \ref{H3}.
\begin{theorem}\label{prop:orbit-equivlence-partial}
If \ref{H3}, then $\f$ is orbit-equivalent to the geodesic flow of $\Sigma$.
\end{theorem}
The proof 
needs some preparation and appears after that of \cref{lemma:uni-convergence-of-geo} below. 
\cref{prop:orbit-equivlence-partial,prop:periodic-orbits-vectors} both use the following.
\begin{lemma}\label{lemma:open-set-in-sm-has-boundary-with-intervals}
If \ref{H1} and $V\subset\smm$ is open, then so are $\{\gav(\re)\mid v\in V\}\subset\M$ and $V_{\pm\infty}\dfn\{\gav(\pm\infty)\mid v\in V\}\subset\Minfty$. In particular, unless $V=\emptyset$, $V_{\pm\infty}$ contain open arcs on $\Minfty$.
    If $V\subset\smm$ is closed, then so is $\{\gav(\re)\mid v\in V\}\subset\M$.
\end{lemma}
\begin{proof}
Suppose $V$ is open (resp.\ closed) and $q\in \{\ga_v(\re)\mid v\in V\}$.
If $q=\gav(0)$ for some $v\in V$, then $q$ is an interior (resp. accumulation) point of $\pi V\subset \{\ga_v(\re)\mid v\in V\}$.
If $q=\ga_{v_0}(t_0)$ for some $v_0\in V$ and $t_0\neq 0$, then $ V\cap T^1_{\pi v_0}\M$ is open (resp.\ closed) and $q$ is an interior (resp.\ accumulation) point of $\{\ga_v(\re)\mid v\in V\cap T^1_{\pi v_0}\M\}\subset \{\ga_v(\re)\mid v\in V\}$.
Therefore, $\{\ga_v(\re)\mid v\in V\}$ is open (resp.\ closed) if $V$ is open (resp.\ closed). 

Next, $V_{\pm\infty}\dfn\{\gav(\pm\infty)\mid v\in V\}=\bigcup_{p\in \pi V}\{\gav(\pm\infty)\mid v\in V\cap T^1_p\M\}=\bigcup_{p\in \pi V} H^{\pm}_p( V\cap T^1\M)$, where $H^\pm_p$ are the homeomorphisms appearing in \cref{thm:pqtilde}. If $V$ is open, then so are $V\cap T^1_p\M$ in $T^1_p\M$, and both $H^{\pm}(V\cap T^1_p\M)$ and $V_{\pm\infty}$ in $\Minfty$. 
Moreover, $V_{\pm\infty}$ contain open arcs on $\Minfty$ since $\Minfty$ is homeomorphic to $S^{n-1}$, where open sets are unions of  arcs.
\end{proof}
We now insert a few purely Riemannian auxiliary results. \cref{prop:orbit-equivlence-partial,prop:joint-continuous} use ``orthogonal'' projection to a geodesic and its continuity (this construction more generally provides a projection to a convex set).
\begin{definition}[Orthogonal projection]\label{def:ortho-projection}
If $\Sigma$ is a closed surface with negative curvature and $g \colon \re \to \M$ is a geodesic, then the \emph{orthogonal projection $\Pi_g\colon\M\to g(\re)$ to $g$} is defined by
\[
d(p, \Pi_g(p)) = \min_{t \in \re} d(p, g(t)).
\]
\end{definition}

\begin{remark}\label{remark:orthogonal-projection}
\begin{enumerate}
    \item  This is well defined because $t \mapsto d(p, g(t))$ is strictly convex (since $\Sigma$ has negative curvature) and hence attains its minimum at exactly one $t_0 \in \mathbb{R}$.
    \item The point $q = g(t_0)$ is characterized by $\dot g_{pq}(0)\perp\dot{g}(t_0)$. This follows from the first variation formula for arc length (\cite[Chapter~9, Proposition~9.2]{DoCarmo-Riem-Geo}). Hence ``orthogonal projection.''\label{orth-proj-2}
\item$\Pi_g$ is continuous, indeed 1-Lipschitz \cite[Chapter~II.2, Proposition~2.4]{Bridson-Haefliger}, and in fact, smooth. To see this, note that $\Pi_g(p)$ is characterized by
\[
\langle \exp_{\Pi_g(p)}^{-1}(p),\dot{g} \rangle = 0,
\] a smooth equation in $p$. By the Implicit Function Theorem, $\Pi_g$ is smooth since the derivative with respect to the foot-point is nonsingular.\label{orth-proj-3}
\end{enumerate}
\end{remark}

\begin{proposition} \label{prop:joint-continuous}
Suppose $\Sigma$ is a closed surface with negative  curvature. 
Then
\[
\mathcal{P} \colon \M \times (\Minfty \times \Minfty \smallsetminus \Delta) \to \M,
\qquad
(p,x,y)\mapsto\Pi_{g_{xy}}(p)
\]
is continuous, where
\[
\Delta\dfn\big\{(x,x) \mid x \in \Minfty\big\} \subset \Minfty \times \Minfty\]
is the diagonal and $\Pi_{g_{xy}}$ is the orthogonal projection onto the geodesic $g_{xy}$.
\end{proposition}

\begin{proof}
To show that $\mathcal{P}$ is continuous, that is, if $(p_n, x_n, y_n) \to (p, x, y)$ in $\M \times (\Minfty \times \Minfty \smallsetminus \Delta)$, then $\Pi_{g_{x_n y_n}}(p_n) \;\longrightarrow\; \Pi_{g_{xy}}(p)$, we first prove that the geodesic $g_{xy}$ depends continuously on $(x,y) \in \Minfty \times \Minfty \smallsetminus \Delta$.

Fix a base-point $o \in \M$. For each $(x,y) \in \Minfty \times \Minfty \smallsetminus \Delta$, parametrize $g_{xy} \colon \mathbb{R} \to \M$ as the unique unit-speed geodesic with $g_{xy}(-\infty) = x$, $g_{xy}(\infty) = y$, and $g_{xy}(0) = \Pi_{g_{xy}}(o)$. To show that $g_{xy}$ varies continuously in $(x,y) \in \Minfty \times \Minfty \smallsetminus \Delta$, it suffices to show 
\begin{lemma}\label{lemma:uni-convergence-of-geo}
    If $(x_n, y_n) \to (x, y)$ in $\Minfty \times \Minfty \smallsetminus \Delta$, then $g_{x_n y_n}(t) \to g_{xy}(t)$ locally uniformly, that is, $g_{x_n y_n}(t)$ converges uniformly to $g_{xy}(t)$ on any compact subset of $\mathbb{R}$.
\end{lemma}

\begin{proof}[Proof of \cref{lemma:uni-convergence-of-geo}]
Set $p_n = g_{x_n y_n}(0) = \Pi_{g_{x_n y_n}}(o)$. We first show that $\{p_n\}$ is bounded. Since $x \neq y$, the geodesic $g_{xy}$ is a well defined and closed subset at finite distance $d_0 = d(o, g_{xy})<\infty$ from $o$. For large $n$, $(x_n, y_n)$ is close to $(x,y)$ in the product (cone) topology, hence $g_{x_n y_n}$ passes within a uniformly bounded neighborhood of $o$ (see the ``fellow-traveler property'' in \cite[Chapter III.H, p.~457]{Bridson-Haefliger}). Hence $d(o, p_n)$ is uniformly bounded and $\{p_n\}$ lies in a compact set. Therefore, every subsequence of $\{p_n\}$ has a convergent subsequence with limit $p_\infty \in \M$. We will see that $p_\infty$ is independent of the subsequence, and in fact, $p_\infty=\Pi_{g_{xy}}(o)$.

Since $\Sigma$ has negative curvature, geodesic rays from a fixed point depend continuously on their forward endpoint in the cone topology. Hence, $g_{x_n y_n}|_{[0,\infty)}$ with $g_{x_n y_n}(0)=p_n$ converges locally uniformly to $g_{p_\infty y}|_{[0,\infty)}$, and similarly, $g_{x_n y_n}|_{(-\infty,0]}$ converges locally uniformly to $g_{xp_\infty}|_{(-\infty,0]}$. 
The limiting bi-infinite geodesic therefore has endpoints $x$ and $y$ and passes through $p_\infty$. By uniqueness of the complete geodesic with endpoints $x$ and $y$, this limiting geodesic is $g_{xy}$, and hence $p_\infty = \Pi_{g_{xy}}(o)$. Since  $p_\infty = \Pi_{g_{xy}}(o)$ for every subsequence of $\{p_n\}$, the full sequence satisfies $p_n\to \Pi_{g_{xy}}(o)$, and   $g_{x_n y_n} \to g_{xy}$ locally uniformly.
\end{proof}

By the triangle inequality,
\begin{equation}\label{eqn:orth-proj-joint-continuous}
d\big(\Pi_{g_{x_n y_n}}(p_n),\, \Pi_{g_{xy}}(p)\big)\leq
\underbracket{d\big(\Pi_{g_{x_n y_n}}(p_n),\, \Pi_{g_{x_n y_n}}(p)\big)}_{\mathclap{\leq d(p_n, p) \to 0\text{ since }\Pi_g\text{ is 1-Lipschitz}}}
+
d\big(\Pi_{g_{x_n y_n}}(p),\, \Pi_{g_{xy}}(p)\big).
\end{equation}
To study $d\big(\Pi_{g_{x_n y_n}}(p),\, \Pi_{g_{xy}}(p)\big)$, set $q_n = \Pi_{g_{x_n y_n}}(p) = g_{x_n y_n}(t_n)$ with $t_n \in \mathbb{R}$, and $q = \Pi_{g_{xy}}(p)$. Since $g_{x_n y_n} \to g_{xy}$ locally uniformly, we have $d(q, g_{x_n y_n}) \to 0$. Since $d(p, g_{x_n y_n}) \leq d(p, q) + d(q, g_{x_n y_n})$ is uniformly bounded and $g_{x_n y_n}$ is parametrized by arc length with $g_{x_n y_n}(0)$ bounded, the sequence $\{t_n\}$ is bounded. Taking a convergent subsequence $t_{n_k} \to t_*$, we have
\[
q_{n_k} = g_{x_{n_k} y_{n_k}}(t_{n_k}) \to g_{xy}(t_*)
\]
by locally uniform convergence of the geodesics. The limit $g_{xy}(t_*)$ lies on $g_{xy}$ and satisfies
\[
d\big(p,\, g_{xy}(t_*)\big) = \lim_{k} d\big(p,\, q_{n_k}\big) = \lim_k d\big(p,\, g_{x_{n_k} y_{n_k}}\big) = d(p, g_{xy}).
\]
Hence, $g_{xy}(t_*) = q$ by the uniqueness of orthogonal projection. Since every subsequence has the same limit, we have $q_n \to q$, and $d\big(\Pi_{g_{x_n y_n}}(p),\, \Pi_{g_{xy}}(p)\big)$ converges to $0$. Therefore, $d(\mathcal{P}(p_n, x_n, y_n), \mathcal{P}(p,x,y)) \to 0$ in \eqref{eqn:orth-proj-joint-continuous}. The sequence was arbitrary, so $\mathcal{P}$ is continuous.
\end{proof}

\begin{proof}[Proof of \cref{prop:orbit-equivlence-partial}]
This is vacuous when $\mu=0$. We consider $\mu>0$; the proof for $\mu<0$ is similar. Define \[P\colon \smm\to \smm, \; v\mapsto\dot g(0),\] where $g$ is the boundary chord (\cref{def:angle-for-pqtilde-end-points}) of $\gav$ with $g(0)=\Pi_g(\pi v)$  (\cref{def:ortho-projection}). $P$ sends orbits of $\f$ to those of the geodesic flow. We will prove that $P$ is a homeomorphism by showing $P$ is bijective, continuous and closed. Note that \ref{H3} is only needed to guarantee the injectivity of $P$, but the rest of the proof also works assuming only \ref{H2} (although this does not yield uniqueness of $\eta$ in the proof of surjectivity; see \cref{remark:orbit-factor}).

Injectivity: First assume $v_1\neq v_2\in\smm$ are tangent to distinct \m\ geodesics $\ga_{v_1}$ and $\ga_{v_2}$. \ref{H3} implies that either $\ga_{v_1}(\infty)\neq \ga_{v_2}(\infty)$ or $\ga_{v_1}(-\infty)\neq \ga_{v_2}(-\infty)$, so by \cref{prop:boundary-bijection}, $P(v_1)$ and $P(v_2)$ are tangent to different boundary chords, hence $P(v_1)\neq P(v_2)$. Next we prove that if $\dot \ga(t_1)\eqqcolon v_1\neq v_2 \coloneqq \dot\ga(t_2)$ for some $\ga$ and $t_1 \neq t_2$, then $P(v_1)\neq P(v_2)$. Suppose the contrary, that is, points $\ga(t_1)$ and $\ga(t_2)$ project onto the same point on the boundary chord $g$. Assuming $t_1<t_2$ without loss of generality, the geodesic $g_1\dfn g_{\ga(t_1)\ga(t_2)}$ is therefore such that $x' \dfn g_1(-\infty)\in\widehat{xy}$ and $y' \dfn g_1(\infty)\in\widehat{yx}$ where $x\dfn \ga(-\infty)$ and $y\dfn \ga(\infty)$. By \cref{rmk:general-behaviors-easy-observation}(\ref{LeftRightTangency}), the chord of $\ga([t_1,t_2])$ is to the left of $\ga([t_1,t_2])$ with both intersections transverse. Since $x'\in\widehat{xy}$ and $y'\in\widehat{yx}$, there exists $t_0\notin\{t_1,t_2\}$ such that $\ga(\re)\cap g_1(\re)\supset\{\ga(t_0),\ga(t_1),\ga(t_2)\}$, which contradicts that a \m\ geodesic intersects any geodesic at most twice (\cref{rmk:general-behaviors-easy-observation}(\ref{LeftRightTangency})).



Surjectivity: Let $u\in \smm$, and let $\eta$ be the unique (up to reparametrization) \m\  geodesic with the boundary chord $g_u$. Consider the geodesic $g_{u,\perp}$ orthogonal to $g_u$ at $\pi u$. By \cref{lemma:ortho-proj-any-point}, $g_{u,\perp}$ intersects $\eta$ at some point $p$. Let $\eta_0$ be the reparametrization of $\eta$ such that $\eta_0(0)=p$, and therefore, $P^{-1}(u)=\dot \eta_0(0)$.

Continuity: Let $v_n\in\smm$ and $v_n \xrightarrow{n\to\infty} v$. Then $\pi P(v_n)\to \pi P(v)$ by \cref{prop:joint-continuous}, and $P(v_n)\to P(v)$ once we show that $\measuredangle (\sigma P(v_n), P(v))\to 0$ where $\sigma P(v_n)$ denotes the parallel transport of $P(v_n)$ along the minimizing geodesic from $\pi v_n$ to $\pi v$. To see this, note that with respect to the cone topology, $g_{P(v_n)}(\pm \infty)=\ga_{v_n}(\pm\infty)\to \ga_{v}(\pm\infty)=g_{P(v)}(\pm\infty)$. 

The proof that $P^{-1}$ is continuous is analogous: given $u_n \to u$ in $\smm$, let $\eta_n$ be the unique \m\ geodesic such that $g_{u_n}$ is the boundary chord of $\eta_n$ with $P^{-1}(u_n) = \dot\eta_n(0)$. 
Since the endpoints of $g_{u_n}$ converge to those
of $g_u$, the aforementioned fellow-traveler property implies that the
$g_{u_n}(0)$ are bounded, and the subsequence argument
from \cref{lemma:uni-convergence-of-geo} gives $g_{u_n}
\to g_u$ locally uniformly.  
Hence, the footpoints $\pi P^{-1}(u_n)$ are bounded, and any subsequential limit must equal to $\pi P^{-1}(u)$ by the uniqueness of orthogonal projection and locally uniform convergence $g_{u_n} \to g_u$. Hence, $\pi P^{-1}(u_n) \to \pi P^{-1}(u)$. The tangent directions converge since the endpoints of $\eta_n$ coincide with those of $g_{u_n}$, which converge to those of $g_u$.
\end{proof}

\subsection{Magnetic flat strips and axes}\label{section:mag-axes-and-flat-strip}

We next study magnetic flat strips (\cref{def:magnetic-flat-strip}), which arise when a pair $(x,y) \in\Minfty\times\Minfty$ admits more than one magnetic geodesic from $x$ to $y$. Magnetic flat strips give rise to singular vectors (but $\Sing$ can be nonempty even in their absence), and these are not hyperbolic (\cref{prop:dim-of-subbundle}, so a magnetic flow with a singular vector is not uniformly hyperbolic). We also study axial isometries of $\M$ and their action on magnetic geodesics, which gives rise to magnetic axes and periodic orbits.
\begin{proposition}\label{prop:magnetic-flat-strip-is-connected}
    Suppose $(\Sigma,\mu)$ satisfies \ref{H1} and $(x,y) \in\Minfty\times\Minfty$. If there are two \m\  geodesics $\ga_1,\ga_2$ with $\ga_i(-\infty)=x$ and $\ga_i(\infty)=y$ for $i\in\{1,2\}$, then $\ga_{xp}(\infty)=y$ and $\ga_{py}(-\infty)=x$ for any $p$ in the closed subset $S$ of $\M$ bounded by $\ga_1$ and $\ga_2$.
\end{proposition} 

\begin{proof}
We prove $\ga_{xp}(\infty)=y$ by contradiction; the proof of $\ga_{py}(-\infty)=x$ is similar.
    If $\ga_{xp}(\infty)\neq y$ for some $p\in S$, then $\ga_{xp}\neq \ga_i$, $i\in\{1,2\}$. By \cref{cor:magnetic-geodesic-is-unbounded}, $\ga_{xp}$ intersects either $\ga_1$ or $\ga_2$. If 
    $\ga_{xp}$ intersects $\ga_2$ at $q\in\M$, then $\ga_{xp}$ and $\ga_2$ are distinct \m\  geodesics
    from $x\in\Minfty$ to $q\in \M$, contrary to \cref{thm:connectivity-combo}(\ref{l:nonflatstrip}). 
    Likewise when $\ga_{xp}$ intersects $\ga_1$.
\end{proof}

Combining \cref{prop:stable-mag-jf-same-endpoint}, \cref{lemma:sing-vector-has-constant-y,lemma:sing-zero-0-magnetic-curvature}, the following result is immediate.

\begin{corollary}\label{cor:magnetic-flat-strip-has-0-magnetic-curvature-and-magnetic-flat-strip-only-singular}
    If $(\Sigma,\mu)$ satisfies \ref{H2}, then
    \begin{enumerate}
        \item If  $p\in S$ as in \cref{prop:magnetic-flat-strip-is-connected}, then $\kmu(p)=0$ and $\dot\ga_{xp}(t)\in\Sing$ for all $t\in\re$;\label{cor:magnetic-flat-strip-has-0-magnetic-curvature}
        \item \label{Cor:magnetic-flat-strip-only-singular}
    for $v\in \Reg$ (\cref{def:sing-reg-mag-vectors}), $\ga_v$ is the unique \m\  geodesic from $\ga_v(-\infty)$ to $\ga_v(\infty)$ up to reparametrization.
    \end{enumerate}
\end{corollary}


Recall that the (nonuniformly hyperbolic) geodesic flow of a surface $N$ with nonpositive Gaussian curvature may have flat strips, which are defined as an isometric immersion $I\colon[0, C] \times \re \to N$ such that $g_c(t)\coloneqq I(c, t)$ is a geodesic for any $c\in[0,C]$ (and $C$ is then called the width of the strip). The geodesics in $I([0,C]\times \re)$ behave as those on a flat surface. For example, they are parallel (in fact, equidistant) and  carry no hyperbolicity.

The following lemma demonstrates a similar phenomenon in the magnetic setting and explains the motivation for the name ``magnetic flat strip" (\cref{def:magnetic-flat-strip}).

\begin{lemma}\label{lemma:magnetic-flat-strip-parallel}
    If \ref{H2} and two \m\  geodesics $\ga_1$ and $\ga_2$ are such that 
    $\ga_1(\infty)=\ga_2(\infty)$ and $\ga_1(-\infty)=\ga_2(-\infty)$, then $\ga_1$ and $\ga_2$ are parallel, that is, $t\mapsto d(\ga_1(t),\ga_2(\re))$ is constant.
\end{lemma}
The proof uses the following observation.
\begin{lemma}\label{lemma:finite-distance-to-point-with-negative-kmu}
    If \ref{H2}, then for any $q\in \M$ and any \m\  geodesic $\ga$, there is a point $p\in\M$ with $\kmu(p)<0$ such that the geodesic segment from $p$ to $q$ intersects $\ga$.
\end{lemma}
\begin{proof}
    Otherwise, there is a \m\  geodesic that bounds a half-space of $\M$ on which $\kmu\equiv0$, so there is  a fundamental domain on which $\kmu\equiv0$, contrary to \ref{H2}.
\end{proof}

\begin{proof}[Proof of \cref{lemma:magnetic-flat-strip-parallel}]
Without loss of generality, assume 
$\ga_1(\re)$ is in the subset of $\M$ bounded by $\ga_2(\re)$ and the arc $\widehat{\xi_1\xi_2}$ in $\Minfty$ introduced before the proof of 
\cref{thm:ptildeqtilde-magnetic-visibility-property},
where $\xi_1\dfn \ga_1(-\infty)=\ga_2(-\infty)$ and $\xi_2\dfn \ga_1(\infty)=\ga_2(\infty)$. 
{\color{black}
Let $C\colon[0,a]\to\M$ be a geodesic segment such that $C(0)=\ga_1(0)$, $\dot C(0)\perp \dot\ga_1(0)$  and $C(a)\in\ga_2(\re)$. \cref{lemma:finite-distance-to-point-with-negative-kmu} guaranties that $0<a<\infty$. Extend $C$ to $\tilde C\colon (-\eps,a+\eps)\to \M$, and define a smooth \m\  variation $\Gamma\colon (-\eps,a+\eps)\times\re\to \M$ by $\Gamma(s,\cdot)=\ga_{x\tilde C(s)}(\cdot)$. Note that $\Gamma([0,a],\re)=S$, where $S$ is the closed subset of $\M$ bounded by $\ga_1(\re)$ and $\ga_2(\re)$.
By \cref{cor:magnetic-flat-strip-has-0-magnetic-curvature-and-magnetic-flat-strip-only-singular}\eqref{cor:magnetic-flat-strip-has-0-magnetic-curvature}, $\dot\ga_{x\tilde C(s)}(t)\in \Sing$ for all $s\in[0,a]$ and $t\in\re$. Moreover, since $C\perp \ga_1(\re)$,  by \cref{remark:riccati-equation}, the geodesic segment $C$ is orthogonal to $\ga_{xC(s)}(\re)$ for any $s\in[0,a]$. 

For any $t\in \re$, (uniquely) define a curve $C_t\colon [0,a]\to \M$ such that $C_t(0)=\ga_1(t)$, $C_t(s)\in\Gamma(s,\re)$ and $C_t$ is orthogonal to $\Gamma(s,\cdot)$ for any $s\in[0,a]$. In particular, $C_0=C$. 
Again by \cref{remark:riccati-equation} and the fact that $S$ is a \m\ flat strip, $\{C_t([0,a])\}_{t\in\re}$ is a family of geodesic segments. It is unknown so far for any $t\in \re$, whether $C_t([0,a])$ is parametrized by arc length as $C$ is. We show that the answer is positive. That is, $l(C_t([0,a]))$ is, in fact, constant for $t\in \re$.

    
    Consider the \m\  Jacobi field $J_s$ along each \m\  geodesic $\Gamma(s,\cdot)$ corresponding to $\Gamma$ with orthogonal component $y_s$.
    Since $\kmu(\Gamma(s,t))=0$ for all $s\in [0,a]$ and $t\in \re$, \cref{lemma:sing-vector-has-constant-y} implies that $t\mapsto y_s(t)$ is constant for all $s\in[0,a]$.    
    Define $r\colon [0,a]\times \re\to \re$ such that $\Gamma(s,r(s,t))=C_t(s)$ for all $s\in[0,a]$ and $t\in\re$. Then $l(C_t([0,a]))=\int_0^a \lvert y_s(r(s,t))\rvert ds$ for any $t\in\re$.  
    This completes the proof.
    }
\end{proof}

\begin{definition}[Magnetic flat strip]\label{def:magnetic-flat-strip}
{\color{black}
If $(\Sigma,\mu)$ satisfies \ref{H1} and $(x,y)\in\Minfty\times \Minfty$, then a \m\ flat strip of width $C > 0$ with endpoints $(x,y)$ is an embedding $S \colon [0,C] \times \re \to \M$ such that each $S(c, \cdot)$ is a \m\ geodesic with $S(c, -\infty) = x$ and $S(c, \infty) = y$. The curves $S(0,\cdot)$ and $S(C,\cdot)$ are called the \emph{boundary \m\ geodesics} of $S$. If \ref{H2}, then the \emph{maximal \m\ flat strip} $S_{xy}$ with endpoints $(x,y)$ is the \m\ flat strip of greatest width among all \m\ flat strips with endpoints $(x,y)$.}
\end{definition}

\begin{remark}[Basic properties of magnetic flat strips]\label{rmk:mag-flat-strip}
If \ref{H2}, then the following hold.
\begin{enumerate}
    \item By \cref{cor:magnetic-flat-strip-has-0-magnetic-curvature-and-magnetic-flat-strip-only-singular,lemma:magnetic-flat-strip-parallel,lemma:finite-distance-to-point-with-negative-kmu}, a \m\  flat strip $S$ has finite width, consists of singular \m\  geodesics, and $\kmu=0$ on $S$.
    \item In \cref{def:magnetic-flat-strip}, the $\{S(c,\cdot)\}_{c\in[0,C]}$ are pairwise parallel  by \cref{lemma:magnetic-flat-strip-parallel}. In particular, 
    the boundary \m\ geodesics satisfy $d_H(S(0,\cdot),S(C,\cdot))=C$.
    \item \m\ geodesics in a \m\ flat strip have the same boundary chord (up to reparametrization).\label{same-boundary-chord}
\end{enumerate}
When $\Sigma$ is $\mu$-magnetically flat, $\M$ is isometric to $\mathbb{H}^2(-\mu^2)$, and a 
\m\ flat strip of width $C$ is isometric to a family of parallel horizontal lines of width $C$ in $\mathbb{H}^2(-\mu^2)$.
\end{remark}
Axial isometries yield further properties of magnetic flat strips and  $\Sing$, as well as abundance of periodic and nonwandering points (\cref{prop:periodic-orbits-vectors}, \cref{cor:nw-is-TM}).

\begin{definition}[Magnetic axes]\label{def:mag-axis}
A \m\  geodesic $\ga$ is said to be a \emph{\m\  axis} of an axial isometry $\phi$ if $\phi$ preserves $\ga$, that is, $\phi(\ga(t))=\ga(t+\omega)$ for some $\omega>0$.
\end{definition}
\begin{remark}\label{REMTranslationLength}
$\omega$ could naturally be called the \emph{\m\ translation length of $\phi$}, and it is natural to ask how it depends on $\mu$ and whether it is constant on flat strips. This would be basic for any study of the \m\ (marked) length spectrum \cite{MarkedLength,MR23,MR24Corr}.
\end{remark}
\begin{lemma}\label{Lemma:mag_strip_iso_action}
If \ref{H2}, $\phi$ is an isometry and $S$ is a \m\  flat strip of width $C$ with endpoints $(x, y)\in \Minfty\times \Minfty$, then
\begin{enumerate}
   \item $\phi(S)$ is a \m\  flat strip of width $C$ with endpoints $(\phi(x),\phi(y))\in\Minfty\times \Minfty$;\label{item1magstrip}
   \item if $\ga_L$ and $\ga_R$ are respectively the \emph{left and right boundary \m\  geodesic} of $S$, that is, the interior of $S$ is to the right of $\dot\ga_L(t)$ and to the left of $\dot\ga_R(t)$ for all $t\in\re$, then $\phi(\ga_L)$ and $\phi(\ga_R)$ are respectively the left and right boundary \m\  geodesic of $\phi(S)$; \label{item_boudnary-to-boundary}
   \item if the geodesic $g_{xy}$ is an axis of $\phi$, then $\phi$ preserves each \m\  geodesic $\ga$ in $S$ setwise, acting on each $\ga$ as a translation
   .\label{item3}
\end{enumerate}
\end{lemma}

\begin{proof}
  (\ref{item1magstrip}) holds because $\phi$ preserves the Hausdorff distance. 
  For (\ref{item_boudnary-to-boundary}), we first show by contradiction that $\phi$ maps the boundary \m\ geodesics of $S$ to the boundary \m\ geodesics of $\phi(S)$. Suppose that $\phi(\ga_L)$ is in the interior of $\phi(S)$. Then the width of $\phi(S)$ is larger than $d_H(\phi(\ga_L),\phi(\ga_R))=d_H(\ga_L,\ga_R)=C$, which contradicts that $\phi(S)$ is of width $C$ from (\ref{item1magstrip}). The proof in the $\ga_R$ case is similar.
  
  Next we show that $\phi(\ga_L)$ (resp. $\phi(\ga_R)$) is still the left (resp. right) boundary \m\ geodesic of $\phi(S)$.  We prove the case $\mu>0$, and the proof in the $\mu<0$ case is similar (note that there does not exist magnetic flat strips for $\mu=0$ since $\Sigma$ is with negative curvature; see \cref{prop:connectivity-geoflow-MtoMbdry}). Suppose the contrary, that is, $\phi(\ga_L)$ (resp. $\phi(\ga_R)$) is the right (resp. left) boundary \m\ geodesic of  $\phi(S)$. Denote by $g_{xy}$ the geodesic from $x$ to $y$ and consider the geodesic segment $G$ from $g_{xy}(0)$ to $\ga_R(0)$. Note that $G$ intersects $\ga_L$ at some $p\in \ga_L(\re)$. Hence $\phi(p)\in\phi(\ga_L)$. On the other hand, the geodesic segment $\phi(G)$ connecting $\phi(g_{xy}(0))$ to $\phi(\ga_R(0))$ does not intersect $\phi(\ga_L)$, since $\phi(\ga_L)$ is assumed to the right of $\phi(\ga_R)$, which contradicts that $\phi(p)\in\phi(\ga_L)$.
  
  To prove (\ref{item3}), we have $\phi(\ga)\in S$ from (\ref{item1magstrip}). It remains to show that the action of $\phi$ is the translation along each \m\ geodesic $\ga$ in $S$, or equivalently, to disprove the possibility that the action of $\phi$ is a permutation among \m\  geodesics in $S$. 
    This is obvious for boundary \m\  geodesics $\ga_L$ and $\ga_R$ of $S$ by (\ref{item_boudnary-to-boundary}).
    For any point $\ga(t_0)$ with some $t_0\in\re$, by \cref{lemma:sing-zero-0-magnetic-curvature} and \cref{remark:riccati-equation}, there exists $t_1\in \re$ such that the geodesic segment $G$ connecting $\ga_R(t_1)$ and $\ga(t_0)$ is orthogonal to both $\ga_R$ and $\ga$. Since $\phi$ translates $\ga_R$ along itself 
    and that $\ga_R$ and $\ga$ are parallel by \cref{lemma:magnetic-flat-strip-parallel}, we have $\phi(\ga(t_0))\in\ga(\re)$ and therefore, $\phi(\ga(\re))=\ga(\re)$.
\end{proof}

\begin{corollary}\label{cor:relation-between-magnetic-axes}
Suppose $(\Sigma,\mu_1)$ and $(\Sigma,\mu_2)$ satisfy \ref{H2}.
\begin{enumerate}
    \item A $\mu_1$-magnetic geodesic $\ga$ is a $\mu_1$-magnetic axis of an axial isometry $\phi$ if and only if $\phi$ fixes the endpoints of $\ga$.
    \item Let $\phi$ be an axial isometry and let $\ga_{i}$ be a $\mu_i$-magnetic geodesic for $i\in \{1, 2\}$ such that $\ga_1( \infty)=\ga_2(\infty)$ and $\ga_1(-\infty)=\ga_2(-\infty)$. Then $\ga_1$ is a $\mu_1$-magnetic axis of $\phi$ if and only if $\ga_2$ is a $\mu_2$-magnetic axis of $\phi$.
\end{enumerate}   
\end{corollary}

By the density of pairs $(x,y)\in \Minfty\times\Minfty$ that determine (geodesic or equivalently, magnetic by \cref{cor:relation-between-magnetic-axes}) axes in $\Minfty\times\Minfty$, we have the following results for the periodic orbits of hyperbolic magnetic flows.

\begin{proposition} \label{prop:periodic-orbits-vectors}If \ref{H2}, then
    \begin{enumerate}
        \item the set of pairs that determine \m\  axes is dense in $\Minfty\times\Minfty$; \label{cor:periodic-orbits-dense}
        \item the periodic vectors of $\fmu$ are dense in $\sm$; \label{cor:periodic-orbits-dense2}
        \item \label{lemma:regular-vector-open-dense} $\Reg$ is open and dense in $\sm$;
        \item the regular periodic vectors are dense in $\sm$.\label{cor:reg-periodic-orbit-dense}
    \end{enumerate}
\end{proposition}
Density of periodic points implies that every point is nonwandering---though this also follows from the Poincar\'e recurrence theorem, independently of \ref{H1}:
\begin{definition}[Nonwandering set]\label{def:nw-set}
    A point $x\in X$ is \emph{nonwandering} for a flow $F=\{f_t\}_{t\in\re}$ on $X$ if for any neighborhood $U$ of $x$ and $T_0>0$ there is a $t>T_0$ with $f_t(U)\cap U\neq \emptyset$. The set of nonwandering points is denoted by $NW(F)$.
\end{definition}
$NW(F)$ is closed and contains all periodic points, so \cref{prop:periodic-orbits-vectors}(\ref{cor:periodic-orbits-dense2}) gives
\begin{corollary}\label{cor:nw-is-TM}
    If \ref{H2}, then $NW(\f)=\sm$.
\end{corollary}

\begin{proof}[Proof of \cref{prop:periodic-orbits-vectors}]
\eqref{cor:periodic-orbits-dense}: By \cref{cor:relation-between-magnetic-axes}, this is the set of pairs in $\Minfty\times\Minfty$ that determine geodesic axes, which is dense in $\Minfty\times\Minfty$ by \cref{prop:axial-endpoints-pairs-are-dense}.

\eqref{cor:periodic-orbits-dense2}: If $V\subset T^1\M$ is open, then by \cref{lemma:open-set-in-sm-has-boundary-with-intervals} both $V_\infty\coloneqq \{\gav(\infty)\mid v\in V\}$ and $V_{-\infty}\coloneqq \{\gav(-\infty)\mid v\in V\}$ contain open arcs in $\Minfty$, and $\{\ga_v(\re)\mid v\in V\}\subset\M$ is open. By (\ref{cor:periodic-orbits-dense}), there exists a periodic \m\  geodesic $\ga$ in $\{\ga_v(\re)\mid v\in V\}$ with $\ga(\pm \infty)\in V_{\pm \infty}$. Hence, $V$ contains periodic vectors tangent to $\ga$.

(\ref{lemma:regular-vector-open-dense}): $\Reg$ is open because $\Sing$ is closed (\cref{prop:singular-vector-orbit-closed-invariant}). $\Reg$ is dense because otherwise there is a nonempty open $U\subset \Sing$, so by \cref{lemma:open-set-in-sm-has-boundary-with-intervals} both $\{\ga_u(\infty)\mid u\in U\}$ and $\{\ga_u(-\infty)\mid u\in U\}$ contain open arcs on $\Minfty$, and $\{\ga_u(\re)\mid u\in U\}\subset \M$ is open.
By \cref{prop:cone-nbhd-contains-fundamental-domain}, any neighborhood of $x\in \{\ga_u(\infty)\mid u\in U\}$ with respect to the cone topology contains a fundamental domain of $\Sigma$. Therefore, $\{\ga_u(\re)\mid  u\in U\}$ contains a fundamental domain $D$ of $\Sigma$, and by \cref{lemma:sing-zero-0-magnetic-curvature}, $\kmu(p)= 0$ for all $p\in D$ hence for all $p\in \Sigma$, contrary to \ref{H2}.

\eqref{cor:reg-periodic-orbit-dense}: Given any open set $U\subset\sm$, since $\Reg$ is open and dense, $U\cap \Reg$ is open and nonempty. Since periodic vectors are dense in $\sm$, there exists a periodic vector $v\in U\cap\Reg$. Therefore, $v$ is a regular periodic vector in $U$.
\end{proof}

\begin{definition}[Recurrent, uniformly recurrent, {\cite[Definition 3.1]{Ballmann-Brin-Eberlein}}]\label{def:recurrent}
$v\in\sm$ is \emph{recurrent} for the flow $F$ if for every open neighborhood $V$ of $v$, there exist times $t_n\to\infty$ such that $f_{t_n}v\in V$. $v$ is said to be \emph{uniformly recurrent} if for any neighborhood $V$ of $v$,
    \[\liminf_{T\to\infty}\frac{1}{T}\int_0^T\chi_V(f_tv)\,dt>0,\]
    where $\chi_V$ is the characteristic function of $V$. $v\in \smm$ is said to be \emph{uniformly recurrent} if it is a lift of a uniformly recurrent vector in $\sm$. 
\end{definition}
Since any periodic orbit is uniformly recurrent, uniformly recurrent vectors are dense. \ref{H2} implies this by \cref{prop:periodic-orbits-vectors} and in the magnetically flat case, this follows from unique ergodicity of the horocycle flow of surfaces with constant negative curvature \cite{Furstenburg}, which implies that \emph{every} vector is uniformly recurrent. 
Thus:
\begin{corollary}
    If $(\Sigma,\mu)$ satisfies \ref{H1}, then the set of uniformly recurrent vectors for $\f$ is dense in $\smm$ (hence so is the set of recurrent vectors).\label{uni-reitem1}
\end{corollary}
\begin{remark}
The density of (not necessarily uniformly) recurrent points also follows from topological transitivity (\cref{thm:topologically-transitivity}) or from the Poincar\'e Recurrence Theorem (almost every point with respect to the Liouville measure (\cref{rmk:observation}) is recurrent, and the Liouville measure has full support).
\end{remark}

\subsection{Topological transitivity}\label{sec:topo-transitivity}
In this section, we prove the topological transitivity of magnetic flows $\f$ by following \cite[Proposition 3.4, Lemma 3.5]{Eberlein-non-conjugate}
. Note that topological transitivity also implies the conclusion of \cref{cor:nw-is-TM}, giving an alternate proof.
\begin{definition}\label{def:topo_transitivity}Given a \m\  flow $\f$ of $\Sigma$,
\begin{enumerate}
    \item $\f$ is said to be \emph{topologically transitive} if there exists a vector $v \in \sm$ such that $\{\f_t(v) \mid t >0\}$ is dense in $\sm$;
    
    \item for points $p\in \M$ and $x\in \om$, define $V(p,x)\dfn \dot \ga_{px}(0)\in T^1_p\M$.
\end{enumerate}
\end{definition}


\begin{theorem}[Topological transitivity]\label{thm:topologically-transitivity}\label{prop:dense-forward-orbit}
If $(\Sigma,\mu)$ satisfies \ref{H1}, then $\f$ is topologically transitive. Moreover, the following are equivalent.
    \begin{enumerate}
        \item $(\Sigma,\mu)$ satisfies \ref{H2}\label{1tt};
        \item $\f$ has a regular dense orbit\label{2tt};
        \item any dense orbit of $\f$ is regular\label{3tt}.
    \end{enumerate}
\end{theorem}
\begin{proof}
    For $\mu$-magnetically flat surfaces, the main statement is topological transitivity of the horocycle flow \cite{Hedlund-transitivity-of-magnetic-flat}. We now consider hyperbolic magnetic flows. 

     For $v,w\in T^1\Sigma$ with lifts $v^*,w^*\in T^1\M$, by \cref{rmk:nonwandering-set-is-entire-sm,prop:duality-equivalence}, there exists a sequence $\{\phi_n\}\subset D$ such that for any $m\in \M$, $\phi_n^{-1}m\to\ga_{v^*}(\infty)$ and $\phi_nm\to \ga_{w^*}(-\infty)$; see \cref{fig:topological-transitivity}.  
Let $p\dfn \pi v^*$ and $q\dfn \pi w^*$ be the footpoints, $t_n\dfn d_\mu(\phi_np,q)$, and $\bar v_n \dfn V(\phi_n p,q)$. Since $\phi_n p \to \ga_{w^*}(-\infty)$, we have $\f_{t_n} \bar v_n\to w^*$; and similarly, since $\phi_n^{-1}q\to\ga_{v^*}(\infty)$, we have $(\phi_n^{-1})_*\bar v_n=V(p,\phi_n^{-1}q)\to v^*$. 
    Then $v_n\dfn \operatorname{proj}_* \bar v_n=\operatorname{proj}_* (\phi_n^{-1})_*\bar v_n\to \operatorname{proj}_* v^*=v$, and $\f_{t_n}(v_n)=\f_{t_n}(\operatorname{proj}_*\bar v_n)=\operatorname{proj}_* (\f_{t_n}\bar v_n) \to \operatorname{proj}_* w^*=w$, where $\operatorname{proj}\colon \M\to \Sigma$ denotes the natural projection as in \cref{def:deck-trans}. This implies that for any open sets $V$ and $U$ in $\sm$ containing $v$ and $w$ respectively, there exists $t>0$ such that $\f_t(V)\cap U\neq \emptyset$. Thus, $\f$ has a dense positive semiorbit \cite[Proposition 1.6.9]{Fisher-Hasselblatt}, so $\f$ is topologically transitive.

\begin{figure}[hbt]
  \centering
  \begin{tikzpicture}[scale=3, line cap=round, line join=round]

\draw[thick] (0,0) circle (1);

\coordinate (p)  at (0,0.48);
\coordinate (p1) at (0.4,0.4);
\coordinate (p2) at (0.55,0.5);

\coordinate (q)  at (0,-0.33);
\coordinate (q1) at (-0.4,0.2);
\coordinate (q2) at (-0.6,0.1);

\draw[thick]
(-0.6,0.8)
.. controls (-0.35,0.58) and (-0.15,0.48)
.. (p)
.. controls (0.15,0.48) and (0.35,0.58)
.. (0.6,0.8);

\draw[->, thick] (p) -- ++(0.2,0);
\node[above] at (0.28,0.42) {$v^*$};

\draw[thick, blue]
(p)
.. controls (0.10,0.4) and (0.25,0.38)
.. (p1);

\draw[->, thick, magenta] (p) -- ++(0.16,-0.12);
\node[below right, magenta] at (0.14,0.37) {$v$};

\draw[thick]
(-1,0)
.. controls (-0.55,-0.22) and (-0.18,-0.34)
.. (q)
.. controls (0.18,-0.34) and (0.55,-0.22)
.. (1,0);

\draw[->, thick] (q) -- ++(0.2,0);
\node[above] at (0.28,-0.42) {$w^*$};

\draw[thick, blue]
(q1)
.. controls (-0.32,-0.1) and (-0.15,-0.25)
.. (q);

\draw[->, thick, magenta] (q1) -- ++(0.03,-0.18);
\node[left, magenta] at (-0.4,0) {$w$};

\fill (p) circle (0.015);
\node[below] at (p) {$p$};

\fill (p1) circle (0.015);
\node[below right] at (p1) {$\phi_n^{-1}q$};

\fill (p2) circle (0.015);
\node[right] at (p2) {$\phi_{n+1}^{-1}q$};

\fill (q) circle (0.015);
\node[below] at (q) {$q$};

\fill (q1) circle (0.015);
\node[left] at (q1) {$\phi_n p$};

\fill (q2) circle (0.015);
\node[left] at (q2) {$\phi_{n+1}p$};

\fill (-1,0) circle (0.015);
\node[left] at (-1,0) {$\ga_{w^*}(-\infty)$};

\fill (0.6,0.8) circle (0.015);
\node[right] at (0.6,0.8) {$\ga_{v^*}(\infty)$};

\end{tikzpicture}
  \caption{Topological transitivity}
  \label{fig:topological-transitivity}
\end{figure}
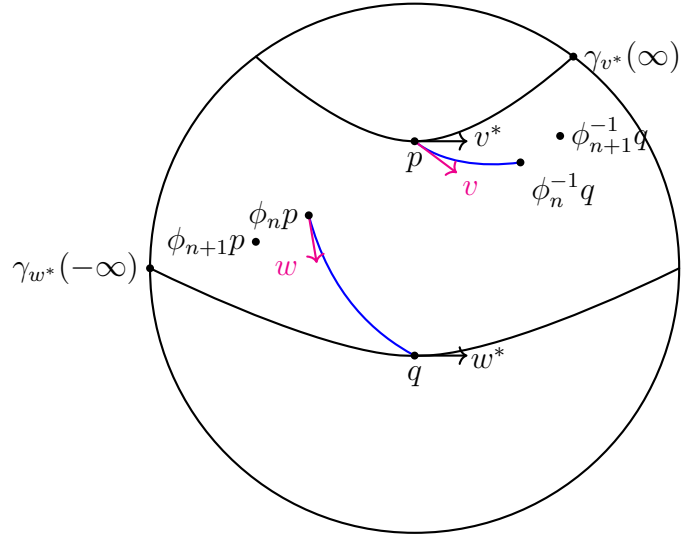

$\eqref{2tt}\implies\eqref{1tt}$ and $\eqref{3tt}\implies\eqref{1tt}$ are immediate by contraposition since $\Reg=\emptyset$ for any $\mu$-magnetically flat surface.
Next we prove the contrapositive of $\eqref{1tt}\implies\eqref{2tt}$ and of $\eqref{1tt}\implies\eqref{3tt}$: If $v\in\Sing$ has a dense orbit, then $\sm=\overline{\mathcal O(v)}\subset\Sing$ by \cref{prop:singular-vector-orbit-closed-invariant}, so \ref{H2} fails.
\end{proof}

{\color{black}
\subsection{Expansivity and applications}\label{ssection:expansivity}

In this section, we capture the universal orbit instability of magnetic flows with notions of expansivity. For more examples of dynamical phenomena associated with expansivity, see \cite[p.~89]{Fisher-Hasselblatt}.

Let $d$ be the distance on $\M$ induced by the Riemannian metric, and let $d_S$ the distance function of 
the Sasaki metric on $\smm$.  
Define the \emph{Knieper metric} \cite[p.\ 294]{Knieper} on $\smm$ by
\begin{equation}\label{def:metric_dk}
d_K(v,w)\coloneqq\max_{t\in[0,1]} d(\gav(t),\ga_w(t)).
\end{equation}
(Notations from \cref{def:mag flow}.) The same definition can be made on $\sm$.
\begin{lemma}\label{lemma:dk-ds-equiv}
If $(\Sigma,\mu)$ satisfies \ref{H1}, then $d_K$ is a metric and uniformly equivalent to $d_S$.
\end{lemma}


\begin{proof} Nonnegativity and symmetry of $d_K$ are clear. For the triangle inequality, note that \[d_K(u,v)+d_K(v,w)\geq d(\ga_u(t),\gav(t))+d(\gav(t),\ga_w(t))\geq d(\ga_u(t),\ga_w(t))\] for all $t\in[0,1]$. For $t=t_0$ such that $d(\ga_u(t_0),\ga_w(t_0))=d_K(u,w)$, this is the triangle inequality. Therefore, $d_K$ is a metric. 

$d_K\le Cd_S$ since $\{D\f_t\mid t\in[0,1]\}$ is equicontinuous and $\Gamma$-equivariant. Conversely, suppose $d_K(v,w)\le\epsilon$. Taking $t=0$ in the definition gives $d(\pi v,\pi w)\le\epsilon$. Consider the magnetic variation $\Gamma$ with $\Gamma(s,0)=g_{\pi v,\pi w}(s)$ and $\frac{\partial}{\partial t}\Gamma(s,0)$ parallel along $g_{\pi v,\pi w}$. At $\pi w$, let $v'\dfn\frac{\partial}{\partial t}\Gamma$ to get $d_S(v,v')=d(\pi v,\pi v')\le\epsilon$ (since $g_{\pi v,\pi w}$ is minimizing). By the preceding, the endpoints of $\Gamma(\cdot,1)$ are at most $C\epsilon$ apart, so $d(\emu(w),\emu(v'))\le(C+1)\epsilon$, hence $d_S(w,v')\le(C+1)\epsilon$ by \cref{prop:rauch-app} below. This implies $d_S(v,w)\le(C+2)\epsilon$.
\end{proof}}
\subsubsection{Expansivity, orbit-equivalence}
The Bowen--Walters notion of expansivity turns out to preclude orbit-equivalence to the geodesic flow in the presence of a flat strip.
\begin{definition}[Expansivity, {\cite[Definition 1.7.2]{Fisher-Hasselblatt}}]\label{def:expansivity}
A flow $f$ on a compact metric space $X$ is \emph{expansive} if for all $\epsilon>0$, there is a $\de>0$ such that if $x$, $y\in X$, $s\colon\re\to\re$ is continuous, $s(0)=0$ and $d(f_t(x),f_{s(t)}(y))\leq\de$ for all $t\in\re$, then $y=f_t(x)$ for some $t\in(-\epsilon,\epsilon)$.  
\end{definition}

\begin{remark}[{\cite[Remark 1.7.3]{Fisher-Hasselblatt}}]
By contraposition, this says that any two orbits will separate by $\de$ at some time, no matter how we reparametrize them.
\end{remark}

It is a classic result by Bowen in \cite{Bowen75} that uniformly hyperbolic systems are expansive. For magnetic flows, we have the following result.

\begin{proposition}\label{prop:expansivity-of-mag-flows}
    The \m\ flow $\f$ is expansive if and only if $(\Sigma,\mu)$ satisfies \ref{H3}.
\end{proposition}

\begin{proof}
We first show by contraposition that expansivity of $\f$ implies \ref{H3}. If $\Sigma$ is $\mu$-magnetically flat, then $\f$ is the horocycle flow of $\Sigma$ with $K\equiv -\mu^2$. The \m\  geodesics are horizontal lines on the upper half-plane of constant curvature $-\mu^2$, and therefore, $\f$ is not expansive. If $(\Sigma,\mu)$ satisfies \ref{H2} and there exists a \m\  flat strip, then $\f$ is not expansive by \cref{lemma:magnetic-flat-strip-parallel}.

If \ref{H3}, then by \cref{thm:ptildeqtilde-magnetic-visibility-property}, for any two distinct points on $\Minfty$ there exists exactly one \m\  geodesic from one point to the other. This implies that backward asymptotic \m\  geodesics cannot be forward asymptotic, so $\f$ is expansive.
\end{proof}
Since expansivity is preserved by orbit-equivalence \cite[Theorem 1.7.8]{Fisher-Hasselblatt}, this gives the converse of \cref{prop:orbit-equivlence-partial} and hence
\begin{theorem}[Orbit-equivalence]\label{thm:orbit-equivalence}
    The \m\ flow $\f$ is orbit-equivalent to the geodesic flow of $\Sigma$ if and only if \ref{H3}. 
\end{theorem}
\begin{remark}\label{remark:orbit-factor}
If \ref{H2}, then the geodesic flow of $\Sigma$ is an orbit factor (see \cref{def:orbit-equiv}) of $\f$ with a continuous surjection $P\colon \smm\to\smm$ in \cref{prop:orbit-equivlence-partial}. In particular, singular vectors corresponding to \m\ geodesics in the same \m\ flat strip $S$ are mapped to tangent vectors of the boundary chord of $S$.
\end{remark}

\subsubsection{Kinematic-expansivity, equilibrium states}

Kinematic-expansivity (\cref{def:k-expansivity}), see 
\cite[\S 4.2]{Fisher-Hasselblatt}, is a weakening of expansivity (\cref{def:expansivity}): any expansive flow is kinematic-expansive. It implies entropy-expansivity (\cref{def:h-expansivity}) of the time-1 map (\cite[Proposition 3.3]{Climenhaga-Thompson-advances}), and hence the existence of equilibrium states by \cref{thm:FH19-existence-ES-construction} (\cref{thm:existence-of-es-equidistribution}). Magnetic flows are kinematic-expansive (\cref{thm:kinematic-expansivity}). 

\begin{definition}[Kinematic-expansivity, {\cite[Definition 1.8.3]{Fisher-Hasselblatt}}]\label{def:k-expansivity}
A continuous flow $\mathcal{F}=\{f_t\}_{t\in \re}$ on a metric space $(X,d)$ is said to be \emph{kinematic-expansive} if for any $\eps>0$, there exists $\de>0$ such that $d(f_tx,f_ty)\leq \de$ for all $t\in \re$ implies that $y\in f_{(-\eps,\eps)}(x)\dfn\{f_tx\mid t\in (-\eps,\eps)\}$.
\end{definition}
Unlike expansivity, kinematic-expansivity is not preserved by orbit-equivalence since time changes can affect  separation of points by a flow. For instance, consider Euclidean rotations in $\mathbb{R}^2$. Being isometric, these are not expansive in any way. If one reparametrizes to unit speed, then points on distinct orbits drift apart, yielding kinematic-expansivity.

Shearing (\cref{rmk:shearing-effect,eqn:magJac1,sentence:shearing-effect}) causes kinematic-expansivity:
\begin{theorem}\label{thm:kinematic-expansivity}
If \ref{H1}, then $\f$ is kinematic-expansive with respect to  $d_K$ from \eqref{def:metric_dk}.
\end{theorem}


\begin{proof}
By continuity of $\f$ and \cref{lemma:pq-connectivity-and-no-self-intersect},
\begin{equation}\label{function-distance-conversion}
\beta\colon [0,\infty)\to [0, \infty),\quad  l\mapsto\inf\{d(p,q)\mid p,q\in\M, \td(p,q)=l\},
\end{equation}
is continuous and increasing, where $d$ is the distance induced by the lifted Riemannian metric $\tilde{g} = \pi^* g$ on $\M$ and $\td$ is the magnetic distance from \cref{def:mag-distance}.

To prove kinematic-expansivity, we will show that if $\eps>0$ and $d_K(\f_tv,\f_tw)\leq \de \dfn \frac{1}{2}\beta^{-1}(\eps)$ for all $t\in \re$, then $w\in \f_{(-\eps,\eps)}v$. It suffices to show that  $w=\f_tv$ for some $t$ because then $|t|<\epsilon$ by contraposition: otherwise, $d_K(w,v)\ge d(\pi v,\pi w)\ge\beta^{-1}(\eps)>\de$.

To that end, note first that if $d_{K}(\f_tv,\f_tw)\le \de$ (hence $d(\gav(t),\ga_w(t))\leq\de$) for all $t\in\re$, then $\gav(\pm\infty)= \ga_w(\pm\infty)$ by definition of ``asymptotic.'' Therefore, either $\gamma_v(\mathbb R)=\gamma_w(\mathbb R)$ (in which case we are done) or $\gav,\ga_w$ define a \m\ flat strip parametrized by an injective \m\  variation $\Gamma\colon (r,t)\in[0,R]\times \re\to \M$ such that $\Gamma(0,t)=\gav(t)$, $\Gamma(R,t)=\ga_w(t)$, and $\Gamma(r,\pm \infty)=\gav(\pm \infty)$ for any $r\in[0,R]$. But this is impossible: the corresponding \m\  Jacobi field is parallel by \cref{prop:stable-mag-jf-same-endpoint,lem:parallel-iff-stable-and-unstable} with constant nonzero orthogonal component by \cref{lemma:sing-vector-has-constant-y}.
Then for the tangential component $x_r$, \eqref{eqn:magJac1} gives a $T>0$ such that for any $t>T$ we have 
\begin{equation}\label{sentence:shearing-effect}
    \lvert x_r(t)\rvert>{\delta}/{R} \text{ for all } r\in[0,R],
\end{equation}hence 
\[d_{K}(\f_tv,\f_tw)\geq d(\gav(t),\ga_w(t))\geq \int_0^R \lvert x_r(t)\rvert dr> \frac{\delta}{R}\cdot R\ge\delta.\qedhere\]
\end{proof}
\begin{proposition}[Entropy-expansivity]\label{prop:mag-flow-is-h-expansive}
    If $(\Sigma,\mu)$ satisfies \ref{H1}, then the time-1 map $f=\f_1$ of $\f$ is entropy-expansive (\cref{def:h-expansivity}).
\end{proposition}
\begin{proof}
To find a $\delta>0$ such that $h^*(f,\delta)=0$ use that \cref{thm:kinematic-expansivity} gives kinematic expansivity, so for $\eps=1>0$, there exists $\de>0$ such that for any $v$
    \[Z_{\de}(v)\dfn\{w\mid d_1(f^nw,f^nv)\leq \de\text{ for } n\in \mathbb{Z}\}=\{w\mid d(\ga_w(t),\gav(t))\leq \de\text{ for } t\in\re\}\subset \f_{[-1,1]}v,\]
where $d_1$ is as in \eqref{def:metric_dk}.    For $\eps=2/m>0$ and \emph{any} $n\in\mathbb{N}$, \[N_m\dfn \{\f_{-1+k\eps}(v) \mid k=0,1,\dots,m\}\subset Z_{\de}(v)\]
    is an $(n,\eps)$-spanning set of $Z_{\de}(v)$ for $f$, so
$r_n(Z_{\de}(v),\eps)=r_1(Z_{\de}(v),\eps)$ from \eqref{def:r_n} is independent of $n$. 
Therefore (see \eqref{eqn:h-*}) $h^*(f,\de)=0$.
\end{proof}

}
\cref{prop:mag-flow-is-h-expansive,thm:FH19-existence-ES-construction} yield equilibrium states for magnetic flows:

\begin{theorem}[Existence of equilibrium states]\label{thm:existence-of-es-equidistribution}
If \ref{H1}, then every Bowen-bounded potential function with finite pressure (\cref{def:topological-pressue-potential-function-bowen-bounded}) has an equilibrium state for $\f$.
\end{theorem}

\section{Horospheres versus orthospheres}\label{SHorospheres}
Horospheres are crucial for understanding hyperbolic geodesic flows. This turns out to rest on two simultaneous features: they are orthogonal to geodesics, and they ``synchronize'' rays asymptotic to a common ideal point (i.e., they are level sets of Busemann functions). We now explore the difficulties that arise for magnetic flows, where these two features are instead mutually exclusive (\cref{fig:magtriineqfailed} encapsulates the issue, and \cref{rmk:obstacles-defining-magnetic-busemann-in-obvious-way} describes it).

These difficulties notwithstanding, magnetic flows exhibit further properties analogous to those of geodesic flows of surfaces with nonpositive curvature, but there is a major exception: 
the natural ``distance'' function (\cref{def:mag-distance}) only ``partially'' satisfies the triangle inequality (\cref{lemma:partial-magnetic-triangle-inequality,lemma:magnetic-triangle-inequality-fails}) and magnetic analogs of the law of cosines fail (\cref{,prop:mag-cosine-law-fails}). However, at least the partial magnetic triangle inequality guarantees that the magnetic B-function and orthospheres (\cref{def:F-F-t-magnetic-busemann-function})---the focus of \cref{ssection:mag-B-function,sec:mag-radial-C2-reg}---are well defined, and their $C^2$-regularity is proved in \cref{prop:c2-regularity-of-magnetic-horosphere}, which later allows us to define strong (un)stable spaces at regular vectors (\cref{def:strong-stable-space,prop:strong_stable_spaces}).

\subsection{Magnetic triangle inequality}\label{sec:Magnetic-Triangle}
The ``distance'' $\td=\td_\mu$ from \cref{def:mag-distance} has notable shortcomings. Irreversibility of the magnetic flow suggests that it may not be symmetric, and we will show that indeed, it is not (\cref{cor:td-is-not-metric}). And the triangle inequality fails. The problem is that $\tilde d$-spheres are not orthogonal to ``radial'' magnetic geodesics. Indeed, for a magnetic geodesic arc from $p$ to $q$, the forward and backward $\tilde d$-spheres around $p$ and $q$, respectively, that meet a point on the arc overlap, and the overlap is the locus of intermediate points $r$ for which the triangle inequality with $p,q,r$ fails (\cref{fig:magtriineqfailed,lemma:magnetic-triangle-inequality-fails}). 
\begin{figure}[h]
    \centering
\begin{tikzpicture}[scale=.6]

  \coordinate (C1) at (0,0);
  \coordinate (C2) at (6,0);

  \begin{scope}
    \clip (C1) circle (4);
    \fill[blue!12] (C2) circle (3);
  \end{scope}

  \draw[blue!80!black, thin] (C1) circle (4) node [above] {$p$};
  \draw[blue!80!black, thin] (C2) circle (3) node [above] {$q$};

  \fill[blue!80!black] (C1) circle (0.08);
  \fill[blue!80!black] (C2) circle (0.08);

  \draw[black, very thick]
    (C1) arc (-152.72:-27.28:3.3754);
  \draw[red, dashed, very thick]
    (C1) arc (-140:-79:3.3754) node {\textbullet\rlap{\!$x$}};
\node[blue!80!black] at (6, 3.5){$B^{\mu-}_{q}(\tilde{d}(p,q) - t)$};
\node[blue!80!black, font=\normalsize] at (0, 3.5)
  {$B^{\mu+}_p(t)$};
\node[black] at (3.7, -2.2)
  {$\gamma_{pq}(t)$};
\end{tikzpicture}
    \caption{Failure of the triangle inequality for intermediate points in the shaded region for any $t$}
    \label{fig:magtriineqfailed}
\end{figure}

The good news is that these points of failure all lie on one side of the magnetic arc, and indeed, there is a partial ``magnetic triangle inequality.''
\begin{proposition}[Partial magnetic triangle inequality]\label{lemma:partial-magnetic-triangle-inequality}
If \ref{H1}, $p,q\in \M$ and $r\in\M$ is to the right of $\ga_{pq}(\re)$, then $\td(p,r)+\td(r,q)\geq \td(p,q)$.
\end{proposition}

\begin{proof}
The case $\mu=0$ is the triangle inequality on Riemannian surfaces, so we assume $\mu>0$, and the proof is similar for $\mu<0$. Suppose the contrary, that is, there exists a point $r$ to the right of $\ga_{pq}(\re)$ such that $\td(p,r)+\td(r,q)< \td(p,q)$ or $\td(p,r)+\td(r,q)= \td(p,q)$. The first case implies that $\td(p,r)<\td(p,q)$, and from the continuity of the function $t\mapsto \td (p,\ga_{rq}(t))$ and the intermediate-value theorem, there exists $t\in [0,\td (r,q))$ such that $\td(p,r)+\td(r,q')= \td(p,q')$ for $q'=\ga_{rq}(t)$. Thus we have reduced the first case to the second, which we now rule out.

Let $R$ be the area enclosed by $\ga_{pq}$, $\ga_{rq}$ and $\ga_{pr}$. Define 
\begin{enumerate}
    \item a smooth curve $C\colon[0,1]\to \M$ in $R$ such that $C(0)=r$, $C(1)\in \ga_{pq}(\re)$, and $\td(p,C(s))=\td(p,r)$ for all $s\in[0,1]$ (such a curve $C$ always exists in $R$: if $C$ intersects $\ga_{pr}$ or $\ga_{rq}$, then just redefine $r$ as the intersection point);
    \item a \m\  variation $\Gamma_p\colon [0,1]\times[0,\td(p,r)]\to \M$ such that $\Gamma_p(s,\cdot)$ is the \m\  geodesic segment from $p$ to $C(s)$ for all $s\in[0,1]$;
    \item a \m\  variation $\Gamma_q\colon[0,1]\times [0,\td(r,q)]\to \M$ such that $\Gamma_q(0,t)=\ga_{rq}(t)$, $\Gamma_q(s,0)\in \Gamma_p(s,\re)$, and $\Gamma_q(s,\td (r,q))=q$ for any $s\in [0,1]$, as suggested by \cref{fig:mag tri ineq}. 
\end{enumerate}

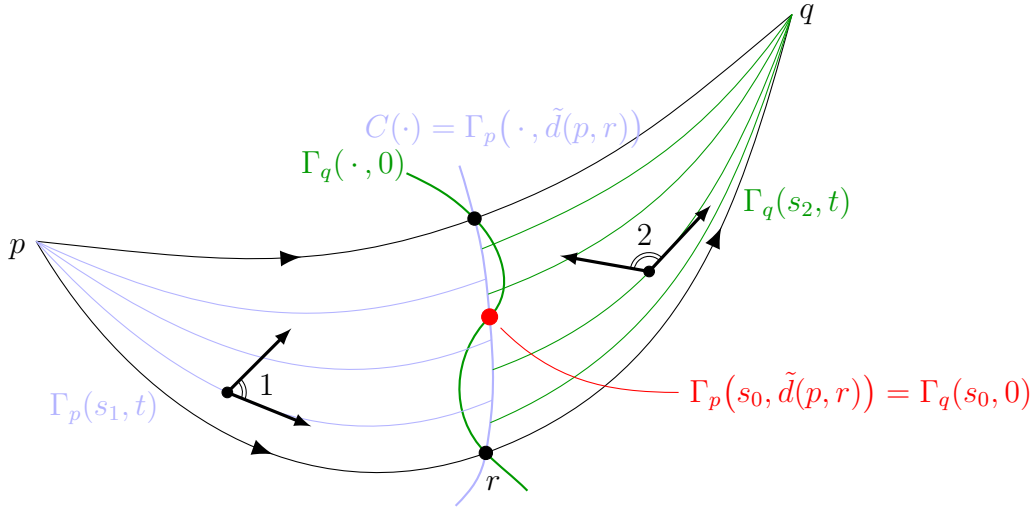
\begin{figure}[h]
        \centering
    \begin{tikzpicture}
    \begin{scope}[decoration={markings, mark=at position 0.6 with {\arrow[scale=2]{latex}} }]
        \node (A) at (-.2,1.3){};
        \node (At) at (-.4,2){};
        \node (O) at (0,0){};
        \node (B) at (-.05,-1.8){};
        \node (C) at (-.45,-2.5){};
        \node (P) at (-6,1){};
        \node (Q) at (4,4){};
      
        \draw[postaction={decorate}] (P.center) to [out=354,in=200] (A.center);
        \draw (Q.center) to [out=220,in=20] (A.center);
        \draw[thick,blue!30] (At.center) to [out=288,in=105] (A.center) to [out=285,in=95] (O.center) to [out=275,in=80] (B.center) to [out=260,in=45] (C.center);
        \draw[postaction={decorate}] (P.center) to [out=300,in=200] (B.center);
        \draw[postaction={decorate}] (B.center) to [out=20,in=255] (Q.center);
        \draw[thick,green!60!black] (.5,-2.3) to [out=135,in=315] (B.center) to [out=135,in=225] (O.center) to [out=45,in=315] (A.center) to [out=135,in=335] (-1.1,1.9);
        \node at (.2,2.5) {$\color{blue!30}C(\cdot)=\Gamma_p\big(\,\cdot\,,\td(p,r)\big)$};
        \node at (-1.8,2) {$\color{green!60!black}\Gamma_q(\,\cdot\,,0)$};
        \node at (-6.25,.9) {$p$};
        \node at (4.2,4) {$q$};
        \draw[blue!30] (P.center) to [out=335,in=200] (-.05,.5);
        \draw[blue!30] (P.center) to [out=325,in=203] (.03,-.3);
        \draw[blue!30] (P.center) to [out=312,in=203] (.04,-1.1);
        \draw[green!60!black] (Q.center) to [out=250,in=20] (.04,-.7);
        \draw[green!60!black] (Q.center) to [out=243,in=20] (-.02,.3);
        \draw[green!60!black] (Q.center) to [out=235,in=22] (-.1,.9);
        \draw[green!60!black] (Q.center) to [out=253,in=23] (.01,-1.4);
        \node at (-5.1,-1.2) {$\color{blue!30}\Gamma_p(s_1,t)$};
        \node at (4.07,1.5) {$\color{green!60!black}\Gamma_q(s_2,t)$};
        \draw[red] (O) to [out=315,in=180] (2.5,-1) node[anchor=west]{$\color{red}\Gamma_p\big(s_0,\td(p,r)\big)=\Gamma_q(s_0,0)$};
        \node at (0.05,-2.2) {$r$};
        
        \filldraw (-3.47,-1) circle (2pt);
        \draw[very thick,-latex] (-3.47,-1) to ({-3.47+1.2*cos(-22)},{-1+1.2*sin(-22)});
        \draw[very thick,-latex] (-3.47,-1) to ({-3.47+1.2*cos(45)},{-1+1.2*sin(45)});
        \draw ({-3.47+.2*cos(-22)},{-1+.2*sin(-22)}) arc (-22:45:.2);
        \draw ({-3.47+.25*cos(-22)},{-1+.25*sin(-22)}) arc (-22:45:.25);
        \node at (-2.95,-.9) {$1$};

        \filldraw (2.11,.6) circle (2pt);
        \draw[very thick,-latex] (2.11,.6) to ({2.11+1.2*cos(47)},{.6+1.2*sin(47)});
        \draw[very thick,-latex] (2.11,.6) to ({2.11+1.2*cos(170)},{.6+1.2*sin(170)});
        \draw ({2.11+.2*cos(47)},{.6+.2*sin(47)}) arc (47:170:.2);
        \draw ({2.11+.25*cos(47)},{.6+.25*sin(47)}) arc (47:170:.25);
        \node at (2.05,1.08) {$2$};
        \filldraw (A.center) circle (2.5pt);
        \filldraw (B.center) circle (2.5pt);
        \filldraw[red] (O.center) circle (3pt);
        \node at (-.5,-3.4) {$\measuredangle1 = \measuredangle_{\gamma_{s_1(t)}}(J_{s_1}(t),\dot{\gamma}_{s_1}(t)) < \frac{\pi}{2}$};
        \node at (-.5,-4.3) {$\color{red}\measuredangle2 = \measuredangle_{\eta_{s_2(t)}}(\Tilde{J}_{s_2}(t),\dot{\eta}_{s_2}(t)) > \frac{\pi}{2}$};
    \end{scope}
    \end{tikzpicture}
        \caption{\cref{lemma:partial-magnetic-triangle-inequality}}
        \label{fig:mag tri ineq}
    \end{figure}

    
Let $\ga_s(t)\dfn \Gamma_p(s,t)$ and $\eta_s(t)\dfn \Gamma_q(s,t)$. Denote by $J_s$ the \m\  Jacobi field of $\Gamma_p$ along $\ga_s$ and by $\tilde J_s$ the \m\  Jacobi field of $\Gamma_q$ along $\eta_s$ for all $s\in[0,1]$. It is obvious that $J_s(0)=\tilde J_s(\td(r,q))=0$. From \eqref{eqn:magJac1} and \eqref{MagneticJacobiEquation}, we have 
\begin{align}\label{eqn:angle-between-magjf-flow-direction}
& \measuredangle_{\gamma_{s_1(t)}}(J_{s_1}(t),\dot{\gamma}_{s_1}(t)) <\pi/2, \text{ for } t\in[0,\td (p,r)],\\
&\measuredangle_{\eta_{s_2(t)}}(\Tilde{J}_{s_2}(t),\dot{\eta}_{s_2}(t))>\pi/2, \text{ for } t\in [0,\td(r,q)),\label{eqn:angle-between-magjf-flow-direction-2}
\end{align}
 for any $s_1,s_2\in[0,1]$, as $\measuredangle1$ and $\measuredangle2$ suggest in \cref{fig:mag tri ineq}.

Now take $s_1=s_2\in\{0,1\}$ in \eqref{eqn:angle-between-magjf-flow-direction},\eqref{eqn:angle-between-magjf-flow-direction-2}. By the connectedness of both $\td$-circle arcs $\{m\in \M\mid \td(p,m)=\td(p,r)\}\cap R$ and $\{m\in \M\mid \td(m,q)=\td(r,q)\}\cap R$, the relation between the angles in \eqref{eqn:angle-between-magjf-flow-direction} and \eqref{eqn:angle-between-magjf-flow-direction-2} at $\ga_0(\td(p,r))$ and $\ga_1(\td(p,r))$ implies that there exists $s_0\in (0,1)$ such that
\[\Gamma_p(s_0,\td(p,r))=\Gamma_q(s_0,0)\] and
\[\measuredangle_{\eta_{s_0}(0)}(\tilde J_{s_0}(0),\dot \eta_{s_0}(0))\leq \measuredangle_{\ga_{s_0}(\td (p,r))}(J_{s_0}(\td (p,r)),\dot \ga_{s_0}(\td (p,r))) <\pi/2,\] which contradicts \eqref{eqn:angle-between-magjf-flow-direction-2}, ruling out the second case. The statement follows.
\end{proof}


We now make explicit the failure of the magnetic triangle inequality on the other side of the reference \m\  geodesic. 
\begin{proposition}[Failure of magnetic triangle inequality]
    \label{lemma:magnetic-triangle-inequality-fails}
    If $\mu\neq 0$, \ref{H1}, and $p,q\in \M$, then there exists $x\in \M$ to the left of $ \ga_{pq}(\re)$ such that $\td(p,x)+\td(x,q)< \td(p,q)$.
\end{proposition}

\begin{proof}

For $t\in\big(0,\td(p,q)\big)$ and $\tilde t\dfn\td(p,q)-t$ consider $B^{\mu+}_p(t)\coloneqq\{r\in\M\mid \td(p,r)=t\}$ and $ B^{\mu-}_q(\tilde{t})\coloneqq \{r\in\M\mid \td(r,q)=\tilde t\}$. 
Comparing the angle subtended by $B^{\mu+}_p(t)$ and $\dot\ga_{pq}(t)$ with the angle subtended by $B^{\mu-}_q(\tilde{t})$ and $\dot\ga_{pq}(t)$, via \eqref{eqn:angle-between-magjf-flow-direction} and \eqref{eqn:angle-between-magjf-flow-direction-2}, we obtain a point $x\in B^{\mu-}_q(\tilde t)$ such that $\td(p,x)<t$, as suggested in \cref{fig:magtriineqfailed}.
Then $\td(p,x)+\td(x,q)<t+\tilde t=\td(p,q)$.
\end{proof}

\begin{remark}
The set of points for which the magnetic triangle inequality fails in  \cref{lemma:magnetic-triangle-inequality-fails} is open in $\M$ with $\ga_{pq}$ in its boundary, specifically the union of the shaded regions for all $t\in[0,\td(p,q))]$ in \cref{fig:magtriineqfailed}.
\end{remark} 

This also implies something we expected from the start---$\td$ does not look symmetric, and we now show that it is not.
\begin{corollary}\label{cor:td-is-not-metric}
    If $(\Sigma,\mu)$ satisfies \ref{H1} and $\mu\neq0$, then $\td=\td_\mu$ from \cref{def:mag-distance} is not symmetric, that is, there exist points $p\neq q\in\M$ such that $\td_\mu(p,q)\neq \td_\mu(q,p)$, or equivalently (by  \cref{rmk:general-behaviors-easy-observation}\eqref{item:mu-mu-relation}), $\td_\mu(p,q)\neq \td_{(-\mu)}(p,q)$ with notations from \cref{def:mag-distance}.
\end{corollary}
\begin{proof}
    
For contradiction, suppose that $\td_\mu(p,q)=\td_\mu(q,p)$ for all $p,q\in\M$. Pick $p\neq q\neq r\in\M$ as  in \cref{lemma:magnetic-triangle-inequality-fails}, that is, $\td_\mu(p,r)+\td_\mu(r,q)<\td_\mu(p,q)$. Note that $p$, $q$ and $r$ determine a $(-\mu)$-magnetic triangle with $(-\mu)$-magnetic geodesic segments $\eta_{pr}$, $\eta_{rq}$ and $\eta_{pq}$ where $r$ is to the right of $\eta_{pq}$. Therefore,  by \cref{lemma:partial-magnetic-triangle-inequality} \[\td_{(-\mu)}(p,r)+\td_{(-\mu)}(r,q)\ge \td_{(-\mu)}(p,q).\] This contradicts \[\td_{(-\mu)}(p,r)+\td_{(-\mu)}(r,q)=\td_\mu(p,r)+\td_\mu(r,q)<\td_\mu(p,q) =\td_{(-\mu)}(p,q).\qedhere\]
\end{proof}

The law of cosines is used in the study of geodesic flows of nonpositive curvature, and we now show that it fails in the magnetic context. This is due to \cref{lemma:magnetic-triangle-inequality-fails} and the following length comparison result.

Given a curve $h$ and a point $p$ on $\M$, we compare in \cref{prop:rauch-app} the length $L(h)$ of $h$ and the length of its  preimage  in $T_p\M$ under the magnetic exponential map (\cref{def:magnetic-exponential-function}).
The counterpart result of \cref{prop:rauch-app} holds for geodesics of manifolds with nonpositive curvature, and is used to prove the law of cosines for geodesic triangles and then properties of Busemann functions, in particular, the continuity with respect to the boundary point; see \cite[Propositions 2.1 and 2.6]{Eberlein-II-1973}.

\begin{proposition}\label{prop:rauch-app}
If \ref{H1}, $p\in \M$, $c$ is a curve on $\M$, $\barc(s)\coloneqq (\emu^{-1})_p(c(s))\in T_p\M$, and 
    \begin{enumerate}
        \item $\mu=0$; or
        \item $\mu>0$ and $l(s)\dfn \dmu(p,c(s))$ is nondecreasing; or
        \item $\mu<0$ and $l(s)$ is nonincreasing,
    \end{enumerate} then $L(c)\geq L(\barc)$, where $L$ is the length function.
\end{proposition} 

\begin{proof}
We first derive a formula, \eqref{eqn:curve-speed}, for the speed of any smooth curve $h(s)$ on a Riemannian surface $(N,g)$ with magnetic structure $(g,\mu\Omega_{g})$. Consider $p \in N$,
    \begin{enumerate}
        \item[(1)] a $\mu$-magnetic variation $\Gamma(s,t)=\Gamma_{\mu}(s,t)$ of $\mu$-magnetic geodesics $\eta_s(\cdot)\dfn \Gamma (s,\cdot)$ such that $\Gamma(s,0)=p$, and $\eta_s(l(s))=h(s)$  for $l(s)=\td_{\mu}(p,h(s))$;
        \item[(2)] the orthonormal basis $s\mapsto E_i^s(t), i\in\{0,1\}$ parallel along $\eta_s$ (see \cite[Proposition 4.32, p.~108]{John-Lee-Riemannian-MFD}) with $E_0^s(t)=\frac{\partial}{\partial t}\eta_s(t)$ and 
        \begin{equation}\label{eqn:orthonormal-basis}
            \{E_0^s(t),E_1^s(t)\}
        \end{equation} 
        being positively oriented for all $t\in [0,l(s)]$;\label{jfitem2}
        \item[(3)] the $\mu$-magnetic Jacobi field $J_s(t)=\frac{\partial\Gamma}{\partial s}(s,t)=x_s(t)\dot \eta_s(t)+y_s(t)i\dot \eta_s(t)$ along $\eta_s$.
    \end{enumerate}
Then
\begin{equation*}
    \frac{dh}{ds}(s)=\big(x_s(l(s)+\frac{dl}{ds}(s)\big)E^s_0(l(s))+y_s(l(s))E^s_1(l(s)),
   \end{equation*}
   and 
\begin{equation}\label{eqn:curve-speed}
    \lvert h^\prime (s)\rvert^2= \big(x_s(l(s)+\frac{dl}{ds}(s)\big)^2+\lvert y_s(l(s))\rvert^2.
\end{equation}

To prove the proposition, it suffices to show that $\lvert c^\prime(s)\rvert\geq \lvert \bar c^\prime(s)\rvert$ for any $s$. The proof for $\mu=0$ is a classic application of the Rauch Comparison Theorem. We now consider $\mu>0$, and the proof in the case $\mu<0$ is similar after changing $\{E_0^s,E_1^s\}$  in \eqref{eqn:orthonormal-basis} to be negatively oriented, that is, $E_1^s=iE_0^s$. 

 Fix $s$. We first apply \eqref{eqn:curve-speed} to $h=c$ of the surface $(\Sigma,g,\Omega_g,\mu_1)$ satisfying \ref{H1}. Then \eqref{eqn:curve-speed} becomes 
\[
    \lvert c'(s)\rvert^2=\big(x_s(l(s)+\frac{dl}{ds}(s)\big)^2+\lvert y_s(l(s))\rvert^2,
\]
where $y_s$ satisfies \eqref{MagneticJacobiEquation} with $\kmu=K_{\mu_1}\leq0$. Note that $x_s(t)\ge0$ for $t\in[0,l(s)]$ since $x_s(0)=0$ and $y_s(t)\geq0$ for  $t\in[0,l(s)]$.

Next we apply \eqref{eqn:curve-speed} to $h=\bar c$ of the surface $(T_p\M,g_0,\Omega_{g_0},0)$ where $g_0$ represents the Euclidean metric. Then \eqref{eqn:curve-speed} becomes 
\[
    \lvert \barc'(s)\rvert^2=\frac{dl}{ds}(s)^2+\lvert \bar y_s(l(s))\rvert^2,
\] where $\bar y_s$ satisfies \eqref{MagneticJacobiEquation} with $\kmu=K_0\equiv0$.

Note that $y_s\ge \bar y_s\geq0$ by \cref{thm:sturm}, and $x_s(t)\geq0$ for $t\in[0,l(s)]$. Since $|x+\frac{dl}{ds}|\geq|\frac{dl}{ds}|$, $\lvert c^\prime(s)\rvert\geq \lvert \bar c^\prime(s)\rvert$ for any $s$, and the proof is complete.
\end{proof}

The magnetic analog of the law of cosines does not hold when $\mu\neq 0$:

\begin{proposition}\label{prop:mag-cosine-law-fails}
$\mu=0$ if \ref{H1} and
\begin{align}
    a_i^2+b_i^2-2a_ib_i\cos(\measuredangle_{a_ib_i})&\leq c_k^2,\label{eq:magcoslaw:1}\\
    b_i^2+c_i^2-2b_ic_i\cos(\measuredangle_{b_ic_i})&\leq a_k^2,\label{eq:magcoslaw:2}\\
    a_i^2+c_j^2-2a_ic_j\cos(\measuredangle_{a_ic_j})&\leq b_k^2,\label{eq:magcoslaw:3}
\end{align}for any $p\neq q\neq r\in \M$ and $i,j,k\in \{1,2\}$, where
\[
a_1=\td(p,r),\quad a_2=\td(r,p),\quad
b_1=\td(p,q),\quad b_2=\td(q,p),\quad
c_1=\td(r,q),\quad c_2=\td(q,r),
\]
 and the angle $\measuredangle_{e_if_j}$ is subtended at their intersection by the corresponding \m\ geodesics $e_i$ and $f_j$ for $e,f\in\{a,b,c\}$.
\end{proposition}
\begin{proof}
    When $\mu=0$, this is just the law of cosines of manifolds with negative curvature. 
    Suppose the statement is true for some $\mu\neq0$ such that $(\Sigma,\mu)$ satisfies \ref{H1}, that is, \eqref{eq:magcoslaw:1}, \eqref{eq:magcoslaw:2} and \eqref{eq:magcoslaw:3} hold for all $p\neq q\neq r\in\M$ and  $i,j,k\in \{1,2\}$.
    
    In \eqref{eq:magcoslaw:1} for $i=1$, $k=2$, let $r\to p$. Then $a_1\to 0$ and the left-hand side goes to $b_1^2$, while on the right-hand side, $c_2=\td(q,r)\to \td(q,p)=b_2$. Thus, $b_1\leq b_2$, 
    that is, $\td(p,q)\leq \td(q,p)$.
    In \eqref{eq:magcoslaw:3}, for $i=k=1$ and $j=2$, let $r \to p$. Then $a_1\to 0$ and the left-hand side goes to $b_2^2$. Thus, $b_2\leq b_1$, so $\td(q,p)\leq \td(p,q)$. Hence, $\td(p,q)=\td(q,p)$ for any $p$, $q\in\M$, which contradicts \cref{cor:td-is-not-metric}.
\end{proof}

\subsection{Magnetic B-function and orthosphere}
\label{ssection:mag-B-function}
For geodesic flows, stable and unstable leaves are obtained from the horospheres, which in turn arise as level sets of Busemann functions. The construction of the latter uses the Riemannian distance $d$ and its definition via lengths of geodesics in an essential way. Mimicking this construction for magnetic flows naturally suggests replacing $d$ with the magnetic distance $\td$, but this is not a metric (\cref{cor:td-is-not-metric,lemma:magnetic-triangle-inequality-fails}). This prevents us from directly defining magnetic Busemann functions and magnetic horospheres in an analogous way (\cref{rmk:obstacles-defining-magnetic-busemann-in-obvious-way,rmk:name-of-B-function}). 
Thus, we consider magnetic B-functions and magnetic orthospheres (\cref{def:F-F-t-magnetic-busemann-function}) using only the partial magnetic triangle inequality in \cref{lemma:partial-magnetic-triangle-inequality}; see \cref{rmk:perspective-on-horospheres}.

Throughout this section, we take $\mu>0$ to fix ideas.

\begin{definition}\label{def:F-F-t_0}
    Given $p\in\M$ and $z\in\om$, define $H_t=H_{t,p,z}\colon \M\to\re$ by 
    \[H_t(q)=
    \begin{cases} 
        \tilde d(q,z)-\tilde d(p,z), & z\in \M,\\
        \tilde d(q,\ga_{pz}(t))-t, & z\in \Minfty,
    \end{cases}
    \]
    where $\ga_{pz}$ is the \m\  geodesic from $p$ to $z$ (\cref{lemma:pq-connectivity-and-no-self-intersect} and \cref{thm:pqtilde}), and $\tilde{d}=\tilde d_\mu$ is from \cref{def:mag-distance}.
\end{definition}

\begin{remark}
    In \cref{def:F-F-t_0}, $\tilde d(q,\ga_{pz}(t))-t=\tilde d(q,\ga_{pz}(t))-\tilde d(p,\ga_{pz}(t))$ is the magnetic analog of the classical Busemann function, measuring the relative \m\ distance from $p$ and $q$ to a given point $\ga_{pz}(t)$.
\end{remark}

 The following fact will be used for \cref{prop:properties-of-F_t-F-B} to show that magnetic B-functions and orthospheres (\cref{def:F-F-t-magnetic-busemann-function}) are well defined.

\begin{proposition}\label{prop:monotone-decreasing-H_t} Suppose $(\Sigma,\mu)$ satisfies \ref{H1} and $\mu>0$. Given $p\in\M$ and  $z\in\om$, if $t_1<t_2$, then $H_{t_1}(q)\geq H_{t_2}(q)$ for any point $q\in\M$ to the left of $\ga_{pz}(\re)$.
\end{proposition}

\begin{proof}
If $z\in\Minfty$, then $H_{t_1}(q)-H_{t_2}(q)=\tilde d(q,\ga_{pz}(t_1))+(t_2-t_1)- \tilde d(q,\ga_{pz}(t_2))\ge0$ by \cref{lemma:partial-magnetic-triangle-inequality}. If $z\in\M$, then $H_{t_1}(q)-H_{t_2}(q)=0$. 
\end{proof}

\begin{remark}\label{rmk:H_t-monotone}
This argument fails for $q$ to the right of $\ga_{pz}(\re)$ by 
    \cref{lemma:magnetic-triangle-inequality-fails}.

\end{remark}

\begin{remark}[Obstacle to defining magnetic horosphere]\label{rmk:obstacles-defining-magnetic-busemann-in-obvious-way}For geodesic flows, the Busemann function, whose level sets are horospheres, is defined as $\lim_{t\to\infty}\bigl(d(q, g(t)) - t\bigr)$ which converges for geodesics on negatively curved manifolds \cite[p.~62]{Eberlein-II-1973}. It is natural to seek a magnetic analog via $\lim_{t\to\infty} H_t(q)$, but by the previous remark it is unclear whether this converges. This is a problem for defining ``magnetic Busemann functions'' and hence for the adaptation of proofs of topological mixing, existence of (un)stable leaves, and the uniqueness of the equilibrium state. 
Specifically, let $\ga\dfn \ga_{pz}$ in \cref{prop:monotone-decreasing-H_t}:
\begin{enumerate}
     \item For $q$ to the left of $\ga$, $t\mapsto H_t(q)$ is monotone by \cref{prop:monotone-decreasing-H_t}, but the failure of the magnetic triangle inequality for the magnetic triangle with edges of magnetic segments on $\ga$, $\ga_{q\ga(t)}$ and $\ga_{pq}$ precludes finding a lower bound for $H_t(q)$ for all $t\in \re$.
    \item For $q$ to the right of $\ga$, although $H_t(q)$ has a lower bound $-\td(p,q)$, it is not clear whether $H_t(q)$ is decreasing with respect to $t$ or has an upper bound.
    \end{enumerate}
     Hence, it is not clear whether $\lim_{t\to \infty}H_t(q)$ exists.
\end{remark}

\begin{remark}\label{rmk:perspective-on-horospheres}
    (Un)stable horospheres for  geodesic flows admit two equivalent descriptions: they are level sets of the Busemann function, and they are curves orthogonal to all geodesics asymptotic to a given geodesic. 
    In light of the preceding remarks, we adopt the latter for the magnetic setting to define magnetic B-functions and orthospheres. They carry the hyperbolicity information as \cref{remark:riccati-equation} suggests.
\end{remark}

\begin{definition}[Magnetic B-function; magnetic orthosphere]\label{def:F-F-t-magnetic-busemann-function}
  Suppose $(\Sigma,\mu)$ satisfies \ref{H1}. Given $p,q\in\M$ and $z\in\om$,
   let $\Gamma\colon[0,d(p,q)]\times \re\to \M$ be the \m\  variation such that $\Gamma(s,t)=\ga_{g_{pq}(s)z}(t)$. 
  As suggested in \cref{fig:mag buse}, we define

\begin{figure}[h]
    \centering
    \begin{tikzpicture}
    \begin{scope}[decoration={markings, mark=at position 0.3 with {\arrow[scale=1.4]{latex}} }]
        \draw (0,0) circle (4);
        \node (q) at ({.4*cos(170)},{.4*sin(170)}){};
        \node (p) at ({3.1*cos(235)},{3.1*sin(235)}){};
        \node (z) at ({4*cos(45)},{4*sin(45)}){};
        \node (h) at (-1,-2.635){};
        \node (mh) at (-.925,-1.5){};
        \node (T) at (2.9,2.4){};
        \draw[orange] (h.center) to [out=90,in=244] (q.center);
        \draw[blue, postaction={decorate}] (mh.center) to [out=357,in=273] (T.center);
        \draw[blue] (mh.center) to [out=177,in=347] (-1.2,-1.45);

        \filldraw (T) circle (2.5pt);
        \filldraw (mh) circle (1.5pt);
        \filldraw (q) circle (1.5pt);
        \filldraw (p) circle (1.5pt);
        \filldraw (z) circle (1.5pt);
        \filldraw (h) circle (1.5pt);
        \draw (q.center) to [out=340,in=285] (z.center);
        \draw[dashed] (q.center) to (p.center);
        \draw (p.center) to [out=345,in=285] (z.center);
        \filldraw[green!60!black] (-1.2,-1.45) circle (1.5pt);
        \draw (-.75,-2.635) to (-.75,-2.385) to (-1,-2.39); 
        \draw (-.725,-1.5) to (-.705,-1.3) to (-.88,-1.28);
        \node at (-.45,.35) {$q$};
        \node at (-2,-2.7) {$p$};
        \node at (0,-1.15) {$\color{blue}\gamma_s$};
        \node at (-1.75,-1.3) {$\color{green!60!black}\gamma_s(0)$};
        \node at (2.25,2.4) {$\gamma_{pz}(t)$};
        \node at (3.1,2.9) {$z$};
        \node at (-.28,-1.95) {$\color{orange}C_{pzq}(s)$};
        \node at (-1.9,-.25) {$g_{pq}(s)=g(s)$};\draw[decoration={brace,mirror,raise=5pt},decorate,red] ({3.1*cos(235)},{3.1*sin(235)}) to node[below=6pt]{} (-1,-2.635);
        \draw[decoration={brace,mirror,raise=3pt},decorate,green!60!black] (-1.2,-1.44) to node[below=6pt]{} (-.925,-1.5);
        \draw[red] (-1.424,-2.9) to [out=270,in=180] (-.6,-4.447);
        \node at (1.5,-4.5) {$\color{red}T_{p\gamma(t)q}(0)=-F_t(p,z,q)$};
        \draw[green!60!black] (-1.116,-1.716) to [out=230,in=350] (-4.5,-2.5);
        \node at (-5.3,-2.5) {$\color{green!60!black}T_{p\gamma(t)q}(s)$};
    \end{scope}
    \end{tikzpicture}
        \caption{\cref{def:F-F-t-magnetic-busemann-function}}
        \label{fig:mag buse}
    \end{figure}


  \begin{enumerate}
      \item the \emph{\m\  orthogonal leaf} $C_{pzq}$ as the curve through $q$ and orthogonal to $\ga_s(\cdot)\dfn \Gamma(s,\cdot)$ for all $s\in [0,d(p,q)]$. Define $C_{pzq}(s)\dfn C_{pzq}\cap \ga_s=\ga_s(T_{pzq}(s))$ for some $T_{pzq}(s)\in\re$;
      \item $F_t\colon \M\times\om\times \M\to \re$ by 
    \[F_t(p,z,q)=
    \begin{cases} 
        -T_{pzq}(0), & z\in \M,\\
        -T_{p\ga(t)q}(0), & z\in \Minfty;
    \end{cases}
    \] 
    \item $q\mapsto B_{pz}(q)=\bmu_{pz}(q)\dfn F(p,z,q)\dfn\lim_{t\to \infty}F_t(p,z,q)$ as the \emph{stable \m\  B-function determined by $p$ and $z$} ($B_{pz}(q)$ is well defined by \cref{prop:properties-of-F_t-F-B});
    \item $(B^\mu_{pz})^{-1}(c)$,  $c\in\re$ as the \emph{\m\  orthospheres} determined by $p$ and $z$;
    \item the \emph{stable orthospherical foliation $\mathcal{H}^s$} of $\smm$, where the leaf through $v\in T_p^1\M$ with $\gav(\infty)=z$ is 
    \[\mathcal{H}^s(v)=\{\dot\ga_{qz}(0)\mid q\in B^{-1}_{pz}(0)\}.\]


    
  \end{enumerate}
  The unstable \m\  B-function, unstable \m\  orthospheres, and unstable \m\  orthospherical foliations are defined in the obvious ways.
\end{definition}

\begin{remark}\label{rmk:name-of-B-function}
    We chose the name ``magnetic B-function'' because it can be viewed as the magnetic counterpart of the Busemann function, whose level sets are orthogonal to asymptotic geodesics  (see \cref{rmk:perspective-on-horospheres}). We avoid naming it ``magnetic Busemann function'', which would suggest that we mean $\lim_{t\to\infty} H_t$ in \cref{rmk:obstacles-defining-magnetic-busemann-in-obvious-way}.
\end{remark}

\begin{remark}\label{rmk:magnetic-busemann-function-orthogonality}

For the geodesic flow, given $z\in\Minfty$, we have $F(p,z,q)=-F(q,z,p)$. This follows from the fact that, for any $v\in T^1\M$ with the induced geodesic $g$ such that $g_v(\infty)=z$, the orthospherical (horospherical) foliations are preserved under the geodesic flow $f^0$. However, this is not true for \m\  flow $\f$ that satisfies \ref{H1}. Horizontal lines in the upper half-plane model of $\mathbb H^2$ as 1-magnetic geodesics provide an example: vertical lines (geodesics) are orthospheres. 
\end{remark}


Magnetic B-functions are well defined and continuous:
\begin{proposition}\label{prop:properties-of-F_t-F-B}
Suppose $(\Sigma,\mu)$ satisfies \ref{H1}. Let $p,q\in\M$, $z\in\om$ and $\ga\dfn \ga_{pz}$. Then
    \begin{enumerate}
        \item for $t_2> t_1>0$, $F_{t_2}(p,z,q)\leq F_{t_1}(p,z,q)$;\label{F_t-F-prop:1}
        \item $F_t\to F$ pointwise as $t\to \infty$, and therefore, $F$ and $\bmu$ are well defined;\label{F_t-F-prop:2}
        \item for given $t$, $F_t(p,z,q)$ is smooth with respect to $p$, $z$ and $q$;\label{F_t-F-prop:3}
        \item for given $z$, $F(p,z,q)$ is smooth with respect to $p$ and $q$.\label{F_t-F-prop:4}
    \end{enumerate}
\end{proposition}

\begin{remark}
    We do not know whether $z\mapsto F(p,z,q)$ is continuous on $\Minfty$. For the geodesic flow of manifolds with nonpositive curvature, this is proved in \cite[Proposition 2.3]{Eberlein-II-1973} using (Riemannian) triangle inequality and cosine laws, but we lack those tools (\cref{lemma:magnetic-triangle-inequality-fails,prop:mag-cosine-law-fails}).
\end{remark}

\begin{proof}[Proof of \cref{prop:properties-of-F_t-F-B}]
Let $d\dfn d(p,q)$.  Since $t_2>t_1$, by \cref{lemma:pq-connectivity-and-no-self-intersect} we have \[\measuredangle_q(\dot \ga_{q\ga(t_2)}(0),-\dot C_{p\ga(t_1)q}(d))>\measuredangle_q(\dot \ga_{q\ga(t_2)}(0),-\dot C_{p\ga(t_2)q}(d))=\frac{\pi}{2}.\] 
Suppose $F_{t_2}(p,z,q)> F_{t_1}(p,z,q)$. Then there exists $s_1, s_2\in [0,d]$ and $r\in\M$ such that \[r= \ga_{g(s_1)\ga(t_1)}\cap \ga_{g(s_2)\ga(t_2)}\in C_{p\ga(t_1)q}\cap C_{p\ga(t_2)q},\]
which implies that \[\measuredangle_r(\dot \ga_{r\ga(t_2)}(0),\dot C_{p\ga(t_2)q}(s_2)=\pi/2<\measuredangle_r(\dot \ga_{r\ga(t_1)}(0),\dot C_{p\ga(t_1)q}(s_1)).\] 
This contradicts that $C_{p\ga(t_1)q}$ is orthogonal to $\ga_{g(s_1)\ga(t_1)}=\ga_{r\ga(t_1)}$ at $r$. 

(\ref{F_t-F-prop:2}) is obvious in the $\mu$-magnetically flat case since orthospheres are geodesics. It now suffices to prove boundedness of $F_t$ from below when $z\in\Minfty$ for $(\Sigma,\mu)$ satisfying \ref{H2}. Write $g=g_{pq}$, and let $\eta\dfn \ga_{qz}$ be the \m\  geodesic from $q$ to $z$. 

Define $A_{pq}(z)\dfn \measuredangle_q(\dot\eta(0),\dot g_{qz}(0))$. By the continuity of $A_{pq}(z)$ to $z\in\Minfty$ with respect to the cone topology and orthogonality between the geodesic and its Riemannian horospheres, there exists $z'=z'(p,q,z)\in\Minfty$ such that $\measuredangle_q(\dot\ga_{qz'}(0),\dot g_{qz}(0))<A_{pq}(z)$ and $(B^0)^{-1}_{qz'}(0)\cap\ga_{pz}=\ga_{pz}(T)$ for some $T>0$. Repeating a similar argument in the proof of (\ref{F_t-F-prop:1}), we have $F_t(p,z,q)\geq -T$ for all $t$, and the desired result follows. 


(\ref{F_t-F-prop:3}) is immediate from the smoothness of \m\  flow $\f$. 

 For (\ref{F_t-F-prop:4}), the statement in the case of $z\in\M$ is obvious. For fixed $z\in\Minfty$ and $p\in\M$, consider the intersection point $I$ of $\ga_{pz}$ and the curve $C$ through $q$ and orthogonal to the stable \m\  Jacobi variation $\Gamma$ consisting of all \m\  geodesics forward asymptotic to $z$. In particular, $C$ is transverse to $\ga_{pz}$, and $C$ is smooth since $\Gamma$ is smooth.  By the implicit-function theorem, the point $I$ depends smoothly on $q$, which proves the smoothness of $F$ with respect to $q$. Similarly, $F$ is smooth with respect to $p$.
\end{proof}

We continue to study the regularity of \m\  B-function $B^\mu_{pz}\colon \M\to \re$. 
For geodesic flows of manifolds with nonpositive curvature, Eberlein shows that the horospheres are $C^2$, which implies that the individual horospherical leaves are $C^1$ 
\cite[\S3]{HI}.

\subsection{Magnetic radial fields and magnetic radial flows}\label{sec:mag-radial-C2-reg}
The geometry of horospheres was initially developed using 
radial fields and radial flows 
\cite{HI}. 
For $z\in\Minfty$, the corresponding radial field has geodesics towards $z$ as integral curves, and the corresponding radial flow moves $\M$ toward $z$. In this section, we define and study magnetic radial fields and magnetic radial flows in order to show the $C^2$-regularity of magnetic B-functions and magnetic orthospheres.

\begin{definition}[Magnetic radial field, magnetic radial flow]\label{def:radial-field}
Suppose $(\Sigma,\mu)$ satisfies \ref{H1}. Then for $p\in\M$ and $q\in\om$, the \emph{$\mu$-magnetic radial field} $Z^q$ in the direction of $q$ is defined by $Z^q(p)=\dot{\ga}_{pq}(0)$. For any $z\in\Minfty$, the flow $\psi$ generated by the vector field $Z$ in the direction of $z$ is called the \emph{\m\ radial flow} with respect to $z$ and is denoted by $\psi=\pi\circ \f\circ(\operatorname{id}_\re\times Z)\colon \re\times \Sigma\to \Sigma$, where $\pi$ is the canonical projection.

For $m\in\M$ different from $p,q\in \om$, $\measuredangle_p(m,q)$ denotes the angle in $[0,\pi]$ subtended by $\dot{\ga}_{pm}(0)$ and $\dot{\ga}_{pq}(0)$, where $\ga_{pq}$, $\ga_{pm}$ and $\ga_{pq}$ are all \m\  geodesics or geodesic rays as in \cref{lemma:pq-connectivity-and-no-self-intersect} or \cref{thm:pqtilde}.
\end{definition}

\begin{proposition}[Orthospheres are $C^2$]
\label{prop:c2-regularity-of-magnetic-horosphere}
Suppose $(\Sigma,\mu)$ satisfies \ref{H1}. Given $z\in\Minfty$ and $p\in\M$, let $Z$ be the \m\  radial field in the direction of $z\in \M(\infty)$ and $B\dfn B_{pz}$ the \m\  B-function determined by $p$ and $z$ (\cref{def:F-F-t-magnetic-busemann-function}).
Then $Z=-grad\, B$, $Z$ is $C^1$, and $\nabla_{v}Z=Y_v'(0)-\mu y\dot\ga_{pz}$ for all $v\in T_qM$, where $Y_v$ is the stable \m\  Jacobi field along $\ga_{qz}$ with $Y_v(0)=v$, and $y$ is the orthogonal component of $Y$ along $\dot\ga_{qz}$. In particular,
\begin{enumerate}
     \item\label{rf1} $Z$ is complete;
     \item if $u\in T_p\Sigma$ is collinear with $Z(p)$, then $\psi_{t*}(u)$ is collinear with $Z(\psi_t(p))$, and $\lVert \psi_{t*}(u)\rVert=\lVert u \rVert$;\label{rf2}
     \item $B$ is $C^2$, and $\psi$ is $C^1$.\label{rf3}
 \end{enumerate}
\end{proposition}

The proof follows \cite{HI}; the idea is to carry over statements for magnetic radial fields in the direction of finite points to the given field $Z$. 
While \cite{HI} uses Rauch's comparison theorem to estimate the change of the Jacobi field, we use Sturm's \cref{thm:sturm}, a counterpart of Rauch's comparison theorem in the 2-dimensional case.

\begin{proof}
We start with notations. For any vector $v=x\dot\ga+y i\dot\ga$ and vector field $V$ along a \m\  geodesic $\ga$, denote by $\bar v$ and $\bar V$ their components along $\dot \ga$, and by $v^\perp$ and $V^\perp$ their components orthogonal to $\dot \ga$. Write $\hat{v}=\hat{v}(\gamma)\dfn y\dot{\ga}+xi \dot{\ga}$. For any \m\  Jacobi field $Y=x\dot\ga+yi\dot\ga$ along $\ga$, by \eqref{eqn:mag-geo-flow}, \eqref{eqn:magJac1} and \eqref{eqn:J'}, we have 
\begin{align}
    &Y'=(\mu x+\dot y)i\dot\ga \label{eqn:Y-'},\\
    &\bar Y'=\mu y\dot\ga+\mu xi\dot\ga=\mu \hat Y. \label{eqn:Y-tangential-prime}
\end{align}
Let $\ga\dfn \ga_{pz}$, $p_n\dfn\ga(n)$ for $n\in\mathbb{N}$ and $B_n\colon q\mapsto B_n(q)\dfn F(p,p_n,q)$. Then 
\begin{enumerate}
    \item $Z_n=-grad\ B_n$ is the magnetic radial field in the direction of $p_n$,
    \item $Z_n$ is $C^\infty$ on $\M-\{p_n\}$, and 
    \item $B=\lim B_n$.
\end{enumerate}
We will show that 
\begin{enumerate}
    \item[(a)] the fields $Z_n$ converge uniformly on compact sets to $Z$, and\label{lemma:C2-lemma1}
    \item[(b)] for any vector field $V$ on $\M$, the covariant derivatives $\nabla_VZ_n$ converge uniformly on compact sets to $Y_V$, where $Y_V(p)=Y_v'(0)+v\cdot \dot\ga(0)\mu i\dot\ga(0)+\mu \hat{Y}_{qz}(0)$, and $Y_v$ is the stable $\mu$-magnetic Jacobi field along $\ga_{qz}$ with $Y_v(0)=v=V(q)$ and $Y_{qz}$ is the stable \m\  Jacobi field along $\ga_{qz}$ with $Y_{qz}(0)=v$. \label{lemma:C2-lemma2}
\end{enumerate}
This implies uniform convergence of the first and second derivatives of the functions $B_n$ on compact sets. Thus $B$ is $C^2$, $grad\ B=\lim grad\ B_n=-Z$, and $\nabla_vZ=-\nabla_v grad\ B=-\lim \nabla_vgrad\  B_n=Y_v'(0)-\mu y \dot\ga$.

Let $K\subset\M$ be compact and $n_0\in\mathbb{N}$ such that $p_n\notin K$ for all $n\geq n_0$.

Proof of \hyperref[lemma:C2-lemma1]{(a)} :
Let $q\in K$ and $n\geq n_0$. 
Then $\left\Vert Z_n-Z\right\Vert(q)=\left\Vert\dot{\ga}_{qp_n}(0)-\dot{\ga}_{qz}(0)\right\Vert$ goes to zero uniformly on $K$ 
by Dini's theorem.

Proof of \hyperref[lemma:C2-lemma1]{(b)}:
Let $q\in K$ and $n\geq n_0$. For any vector field $V$ and $r\in\M$, there exists a $\mu$-magnetic variation $\Gamma$ through \m\  geodesic $\ga_{qr}$ such that
\begin{itemize}
   \item $\Gamma(0,t)=\ga_{qr}(t)$,
   \item $\Gamma(s,H(s))=r, \forall s$ for some continuous function $H$ with $H(0)=\td(q,r)$, and 
   \item $\frac{\partial \Gamma}{\partial s}(0,0)=v=V(q)$.
\end{itemize}
Therefore, 
\begin{align}
\begin{split}
\nabla_VZ^r(q)&=\nabla_{V^\perp} Z^r(q)+\nabla_{\bar{V}} Z^r(q)\\
& =\frac{D}{d (it)}\frac{\partial\Gamma}{\partial t} (0,0)+\frac{D}{d t}\frac{\partial}{\partial t} \Gamma(0,0) \\
& =\frac{D}{d t}\frac{\partial\Gamma}{\partial (it)} (0,0)+ \langle v,\dot\ga_{qr}(0)\rangle\  \nabla_{\dot{\ga}_{qr}(0)}\frac{\partial}{\partial t}\Gamma(0,0)\\
& =(Y^\perp)^\prime(0)+ \langle v,\dot\ga_{qr}(0)\rangle\ \mu i \dot \ga_{qr}(0),\label{eqn:nablaVZ0}
\end{split}
\end{align}
where 
$Y(t)=Y^\perp(t)+\bar Y(t)$ is the \m\  Jacobi field along $\ga_{qr}$ with $Y(0)=v$ and $Y(\dmu(q, r))=0$, and $\bar Y(t)=\langle Y(t),\dot\ga_{qr}(t)\rangle$.

Taking $r=p_n$, the application of \eqref{eqn:nablaVZ0} to the vector field $Z_n$ returns
\begin{equation}
\left\Vert \nabla_VZ_n-Y_V\right\Vert(q)= \left\Vert (Y_{qp_n}^\perp)^\prime(0)+\langle v,\dot\ga_{qp_n}(0)\rangle\ \mu i\dot\ga_{qp_n}(0)-Y_V(p)\right\Vert,\label{eqn:difference0}
\end{equation}
where $Y_{qp_n}$ is the \m\  Jacobi field along $\ga_{qp_n}$ with $Y_{qp_n}(0)=v$ and $Y_{qp_n}(\tilde{d}(p,p_n))=0$. To prove \hyperref[lemma:C2-lemma1]{(b)} now reduces to showing that \eqref{eqn:difference0} uniformly converges to $0$ on a compact set $K$.

By \eqref{eqn:Y-'} and \eqref{eqn:Y-tangential-prime}, we have 
\[Y_{qp_n}^\prime(0)-(Y_{qp_n}^\perp)^\prime(0)=\mu \hat Y_{qp_n} (0),\]
and therefore, \eqref{eqn:difference0} becomes 

\begin{multline}\label{eqn:goal}
\lVert \nabla_VZ_n-Y_V\rVert(p)\\
\begin{aligned}
&=\left\Vert Y_{qp_n}^\prime(0)-\mu \hat Y_{qp_n} (0)+\langle v,\dot\ga_{qp_n}(0)\rangle \mu i\dot\ga_{qp_n}(0)-Y_V(p)\right\Vert\\
&= \lVert Y_{qp_n}^\prime(0)-\mu \hat Y_{qp_n}(0)+ \langle v,\dot\ga_{qp_n}(0)\rangle \mu i \dot\ga_{qp_n}(0)-Y^\prime_v(0)-\mu\hat Y_{qz}(0)-\langle v,\dot\ga_{qz}(0)\rangle\ \mu i \dot\ga_{qz}(0)\rVert\\
&\leq \lVert Y_{qp_n}^\prime(0)-Y^\prime_v(0) \rVert+  \mu \lVert \hat Y_{qp_n}(0) -\hat Y_{qz}(0)\rVert+ \lVert \langle v,\dot\ga_{qz}(0)\rangle\ \mu i \dot\ga_{qz}(0)- \langle v,\dot\ga_{qp_n}(0)\rangle \ \mu i \dot\ga_{qp_n}(0) \rVert.
\end{aligned}
\end{multline}

It is relatively quick to check that the latter two terms converge to zero uniformly:
\begin{multline}\label{eqn:spadesuit1}
\lVert \hat Y_{qp_n}(0) -\hat Y_{qz}(0)\rVert\\
\begin{aligned}
&= \lVert \langle v,\dot\ga_{qp_n}(0)\rangle \ i\dot\ga_{qp_n}(0)+\langle v,i\dot\ga_{qp_n}(0)\rangle \dot\ga_{qp_n}(0)- \langle v,\dot\ga_{qz}(0)\rangle i\dot\ga_{qz}(0)-\langle v,i\dot\ga_{qz}(0)\rangle \dot\ga_{qz}(0)\rVert\\
&\leq\lVert \langle v,\dot\ga_{qp_n}(0)\rangle i\dot\ga_{qp_n}(0)- \langle v,\dot\ga_{qz}(0)\rangle i\dot\ga_{qz}(0)\rVert\\&\quad +\lVert \langle v,i\dot\ga_{qp_n}(0)\rangle \dot\ga_{qp_n}(0)-\langle v,i\dot\ga_{qz}(0)\rangle \dot\ga_{qz}(0)\rVert\\
&\leq\lVert \langle v,\dot\ga_{qp_n}(0)\rangle i\dot\ga_{qp_n}(0)- \langle v,\dot\ga_{qp_n}(0)\rangle i\dot\ga_{qz}(0) \rVert\\&\quad+ \lVert \langle v,\dot\ga_{qp_n}(0)\rangle i\dot\ga_{qz}(0) - \langle v,\dot\ga_{qz}(0)\rangle i\dot\ga_{qz}(0)\rVert\\
&\quad+\lVert \langle v,i\dot\ga_{qp_n}(0)\rangle \dot\ga_{qp_n}(0)- \langle v,i\dot\ga_{qp_n}(0)\rangle \dot\ga_{qz}(0) \rVert\\&\quad+ \lVert \langle v,i\dot\ga_{qp_n}(0)\rangle \dot\ga_{qz}(0) - \langle v,i\dot\ga_{qz}(0)\rangle \dot\ga_{qz}(0)\rVert\\
&\leq\lvert \langle v,\dot\ga_{qp_n}(0)\rangle \rvert \lVert  i\dot\ga_{qp_n}(0)-i\dot\ga_{qz}(0)\rVert +\lvert \langle v,\dot\ga_{qp_n}(0)-\dot\ga_{qz}(0)\rangle \rvert \lVert i\dot\ga_{qz}(0)\rVert\\
&\quad+\lvert \langle v,i\dot\ga_{qp_n}(0)\rangle \rvert \lVert  \dot\ga_{qp_n}(0)-\dot\ga_{qz}(0)\rVert +\lvert \langle v,i\dot\ga_{qp_n}(0)-i\dot\ga_{qz}(0)\rangle \rvert \lVert \dot\ga_{qz}(0)\rVert,
\end{aligned}    
\end{multline}
and
\begin{multline}\label{eqn:diamondsuit11}
    \lVert \langle v,\dot\ga_{qz}(0)\rangle \mu i \dot\ga_{qz}(0)- \langle v,\dot\ga_{qp_n}(0)\rangle \mu i \dot\ga_{qp_n}(0) \rVert\\\begin{aligned}
&\leq \lVert \langle v,\dot\ga_{qz}(0)\rangle \mu i \dot\ga_{qz}(0)- \langle v,\dot\ga_{qz}(0)\rangle \mu i \dot\ga_{qp_n}(0)\rVert+ \lVert  \langle v,\dot\ga_{qz}(0)\rangle \mu i \dot\ga_{qp_n}(0)- \langle v,\dot\ga_{qp_n}(0)\rangle \mu i \dot\ga_{qp_n}(0) \rVert \\
&= \lvert\mu \langle v,\dot\ga_{qz}(0)\rangle\rvert \lVert  i \dot\ga_{qz}(0)- i \dot\ga_{qp_n}(0)\rVert+ \lvert \langle v,\dot\ga_{qz}(0)-\dot\ga_{qp_n}(0)\rangle\rvert\lVert \mu i \dot\ga_{qp_n}(0)\rVert .
    \end{aligned}
\end{multline}


The remainder of the proof establishes that $\lVert Y^\prime_{qp_n}(0)-Y^\prime_v(0)\rVert\rightarrow0$ uniformly for $q\in K$ as $n\rightarrow \infty$. For $T>0$, let $X_{T,qp_n}$ be the \m\  Jacobi field along $\ga_{qp_n}$ with $X_{T,qp_n}(0)=V(p)=v$ and $X_{T,qp_n}(T)=0$. Define $X_{T,qz}$ analogously. Since
\begin{equation}\label{eqn:triineq1}
\lVert Y^\prime_{qp_n}(0)-Y^\prime_v(0)\rVert\leq
\lVert Y^\prime_{qp_n}(0)-X^{\prime}_{T,qp_n}(0)\rVert+
\lVert X^{\prime}_{T,qz}(0)-Y^\prime_v(0)\rVert+
\lVert X^{\prime}_{T,qp_n}(0)-X^{\prime}_{T,qz}(0)\rVert,
\end{equation}
our goal now is to show that all three terms on the right-hand side uniformly converge to zero on $K$ as $T\to \infty$. We claim that
\begin{equation}\label{eqn:norm10}
\left\Vert Y^\prime_{qp_n}(0)-X^{\prime}_{T,qp_n}(0)\right\Vert\leq\frac{1}{T}\lVert V(p)\rVert\to 0.
\end{equation}

To show this, write $Y_{qp_n}(t)=x_1(t)\dot\ga_{qp_n}(t)+y_1(t)i\dot\ga_{qp_n}(t)$,
and $X_{T,qp_n}(t)=x_2(t)\dot\ga_{qp_n}(t)+y_2(t)i\dot\ga_{qp_n}(t)$. Note that 

\begin{enumerate}
   \item $Y_{qp_n}(0)-X_{T,qp_n}(0)=0$;
   \item $Y_{qp_n}(T)-X_{T,qp_n}(T)=Y_{qp_n}(T)$;
   \item $Y^\prime_{qp_n}(t)\perp\dot\ga_{qp_n}(t)$ and $X^{\prime}_{T,qp_n}(t)\perp\dot\ga_{qp_n}(t)$ since \eqref{eqn:Y-'}.
\end{enumerate}
Therefore, $\left\Vert Y^\prime_{qp_n}(0)-X^{\prime}_{T,qp_n}(0)\right\Vert=\lVert\mu (x_1-x_2)+(\dot y_1-\dot y_2)\rVert\leq \mu \lVert x_1-x_2\rVert+\lVert\dot y_1-\dot x_2\rVert$ where $\lVert x_1(0)-x_2(0)\rVert=0$ since $Y_{qp_n}(0)=X_{T,qp_n}(0)=v$. To prove \eqref{eqn:norm10}, it now suffices to show that $\lVert\dot y_1(0)-\dot y_2(0)\rVert$ is uniformly bounded in $n$.

Since both $Y_{qp_n}$ and $X_{T,qp_n}$ are \m\  Jacobi fields along $\ga_{qp_n}$, we have $\big(y_1-y_2\big)^{\prime\prime}(t)+K_\mu \big(y_1-y_2\big)(t)=0$. Applying \cref{thm:sturm} (to compare with the magnetic flat case) returns $\lVert\dot y_1(0)-\dot y_2(0)\rVert\leq\frac{1}{T}\lVert y_1(T)\rVert\leq\frac{1}{T}\lVert V(p)\rVert$ for sufficiently large $n$ so that $\dmu(p,p_n)>T$, where the latter inequality follows from convexity of $y_2$ and the fact that $t\mapsto \lVert y_2(t)\rVert$ is monotonically decreasing on $[0,\dmu(p,p_n)]$. This completes the proof of \eqref{eqn:norm10}, and the proof is similar for $\lVert X^{\prime}_{T,qz}(0)-Y^\prime_v(0)\rVert$. 

Finally, we will show below that $d(\ga_{qp_n}(T),\ga_{qz}(T))\to 0$ uniformly on $K$, which implies that the last term in \eqref{eqn:norm10}, $\lVert X^{\prime}_{T,qp_n}(0)-X^{\prime}_{T,qz}(0)\rVert\to0$ uniformly on $K$ for any fixed $T$, as $n\to\infty$. Then the quantities in \eqref{eqn:triineq1} and \eqref{eqn:goal} uniformly converge to zero on $K$, completing the proof of \hyperref[lemma:C2-lemma1]{(b)}.

Note that \m\ Jacobi fields, as solutions to a second-order differential equation, depend smoothly on their boundary values, in particular on the endpoint $\ga_{qp_n}(T)$ or $\ga_{qz}(T)$. Since the dependence of $X^{\prime}_{T,qp_n}(0)$ (resp. $X^{\prime}_{T,qz}(0)$) on the endpoint $\ga_{qp_n}(T)$ (resp. $\ga_{qz}(T)$) is uniformly continuous for $q\in K$, $d(\ga_{qp_n}(T),\ga_{qz}(T))\to 0$ uniformly on $K$ implies that $\lVert X^{\prime}_{T,qp_n}(0)-X^{\prime}_{T,qz}(0)\rVert\to0$ uniformly on $K$.

To show that $d(\ga_{qp_n}(T),\ga_{qz}(T))\to 0$ uniformly on $K$, we consider the \m\  circle $C^T(q)\dfn\{p\in\M\mid \dmu(q,p)=T\}$ centered at $q$ with radius $T$. Let $C_n^T(q)$ be the arc on $C^T(q)$ from $\ga_{qp_n}(T)$ to $\ga_{qz}(T)$ (counterclockwise when $\mu\ge0$ and clockwise when $\mu<0$), and let $l^T_n(q)\dfn l(C_n^T(q))$. 
Since $n\mapsto l^T_n(q)\to 0$ monotonically, Dini's theorem implies that $l^T_n(q)\to0$ uniformly on $K$ as $n\to\infty$. Since $l^T_n(q)>d(\ga_{qp_n}(T),\ga_{qz}(T))>0$, we thus have $d(\ga_{qp_n}(T),\ga_{qz}(T))\to 0$ uniformly on $K$, and the proof of \hyperref[lemma:C2-lemma1]{(b)} (hence (\ref{rf3})) is complete.

    The \m\  geodesics going to $z$ are integral curves of $Z$ implies (\ref{rf1}). Moreover, $\psi_{t^*}(Z(p))=Z(\psi_t(p))$ proves (\ref{rf2}). 
\end{proof}

\begin{remark}\label{rmk:radial-field-shearing}
    The \m\  radial flow $\psi$ does not map orthospheres to orthospheres due to the twist from \eqref{eqn:magJac1}, so we have no magnetic counterpart of \cite[Proposition 3.2(ii)]{HI}.
\end{remark}

\subsection{Hyperbolicity function}\label{sec:la-function}
We saw in \cref{remark:riccati-equation} the geometric interpretation of the solution of the Riccati equation corresponding to a given magnetic Jacobi field along a magnetic geodesic $\ga$. In particular, when we consider a stable magnetic Jacobi field, the solution of the Riccati equation reflects the geodesic curvature of magnetic orthospheres and, therefore, the hyperbolic behavior of adjacent magnetic geodesics asymptotic to $\ga$. This motivates us to define and study a hyperbolicity function $\la$ on $\sm$ (\cref{def:lau-las-la}) that measures the rate of hyperbolicity of a vector under the magnetic flow analogously to 
\cite[Definition 2.10]{BCFT}. 
\begin{definition}[Hyperbolicity function $\la$; $\Reg(\eta)$]\label{def:lau-las-la}
For $v\in \sm$, define \[\lau(v)=\min_{J^u}\lvert \dot y^u(0)/y^u(0)\rvert,\] where $J^u(t)=x^u(t)\dot\ga_v(t)+y^u(t)i\dot\ga_v(t)$ is any nontangential unstable \m\  Jacobi field along the \m\  geodesic $\ga_v$ (see \cref{def:stable-unstable-magJF});
similarly, define \[\la^s(v)\dfn \min_{J^s}|\dot y^s(0)/y
^s(0)|,\] where $J^s(t)=x^s(t)\dot\ga_v(t)+y^s(t)i\dot\ga_v(t)$ is any nontangential stable \m\  Jacobi field along $\ga_v$. Define $\lav \dfn \min\{\la^u(v),\la^s(v)\},$ 
and for $\eta> 0$, let \[\Reg(\eta) \dfn \{v\mid \lambda(v) \geq \eta\}.\]
\end{definition} 

\begin{remark}\label{remark:la-unknown-continuous}
We do not know whether $\lau$, $\las$ and $\la$ are (uniformly) continuous on $\sm$ because we have not established that the orthospherical leaves vary continuously with respect to $v$. Here is a partial result.
\end{remark}

\begin{lemma}\label{lemma:la-continuous-subspaces}
    Let $\operatorname{proj}\colon \M\to \Sigma$ be the covering map with derivative $\operatorname{proj}_*\colon \smm \to \sm$, and define $ \tilde\lambda^s =\las \circ \operatorname{proj}_*$, $\tilde\lambda^u =\lau \circ \operatorname{proj}_*$, and $\widetilde\la=\la\circ \operatorname{proj}_*$.
    For any vector $v\in\sm$ and its lift $\tilde v\in\smm$, the maps $t\mapsto\lau(\f_t(v))$, $t\mapsto\las(\f_t(v))$ and $t\mapsto\la(\f_t(v))$ are continuous in $t\in\re$. Moreover, $\tilde\lambda^s$ (resp. $\tilde\lambda^u$) is continuous on $\mathcal{H}^{s}(\tilde v)$ (resp. $\mathcal{H}^{u}(\tilde v)$; see \cref{def:F-F-t-magnetic-busemann-function}) 
    with respect to the Knieper metric in \eqref{def:metric_dk} (equivalently, the Sasaki metric by \cref{lemma:dk-ds-equiv}).
\end{lemma}
\begin{proof}
    Given $v\in\sm$, $\las(\f_t(v))$ 
    is a solution of \eqref{eqn:riccati} by \cref{remark:riccati-equation}, so $t\mapsto \las(\f_t(v))$ is continuous. Likewise, $\lau$ and hence $\la$.
    By \cref{remark:riccati-equation} and \cref{def:F-F-t-magnetic-busemann-function}, $\tilde\lambda^s(\tilde v)$ is the geodesic curvature of the stable orthosphere of $\tilde v$ at $\pi \tilde v$, hence continuous on  $\mathcal{H}^{s}(\tilde v)$ by \cref{prop:c2-regularity-of-magnetic-horosphere}. Likewise for $\tilde\lambda^u$.
\end{proof}

The rest of the section concerns properties and applications of $\la$ analogous to those in \cite[\S 3.2]{BCFT} for geodesic flows.
\begin{lemma}\label{lemma:exponential-y-component}
Given $v\in T^1\Sigma$, let $J^u$ and $J^s$ be unstable and stable \m\  Jacobi fields along the \m\  geodesic $\gav(t)$, and let $y^u$ and $y^s$ be orthogonal components of $J^u$ and $J^s$, respectively. Then for any $T>0$,
\begin{equation}\label{eqn:exp_y_comp}
|y^u(T)|\geq e^{\int_0^T \lau(\fmu_tv)\;dt} |y^u(0)| \text{ and } |y^s(T)|\leq e^{-\int_0^T \las(\fmu_tv)\;dt}|y^s(0)|.
\end{equation}
\end{lemma}

\begin{proof}
Let $\sigma \in\{u,s\}$. If $v\in \Sing$, then by \cref{def:sing-reg-mag-vectors}, there exists a \m\  Jacobi field $J^\sigma$ with orthogonal component $y^\sigma$ that is constant and nonzero by \cref{lemma:sing-vector-has-constant-y}. Therefore, $\la^\sigma(\fmu_tv)=0$ for all $t\in\re$ by \cref{prop:singular-vector-orbit-closed-invariant}, and the result follows.

If $v\in \Reg$, then \eqref{eqn:exp_y_comp} is immediate from $\lvert\la^\sigma(\fmu_tv)\rvert\leq\lvert\frac{\dot{y}^\sigma(t)}{y^\sigma(t)}\rvert=\lvert(\log{\lvert y^\sigma(t)\rvert})'\rvert$.
\end{proof}

Besides \cref{lemma:sing-zero-0-magnetic-curvature}, here we give another characterization of singular vectors using the function $\la$ and clarify the relation between Reg and Reg($\eta$).

\begin{lemma} \label{lemma:la-vanish-on-sing}
The following are equivalent for $v\in \sm$. 
  \begin{enumerate}
    \item $v\in \Sing$; \label{vanish:1}
    \item $\la^s(\f_t(v))=0$ for all $t\in \re$; \label{vanish:2}
    \item $\la^u(\f_t(v))=0$ for all $t\in \re$.\label{vanish:3}
  \end{enumerate}
In particular, $\Reg(\eta)\subset\Reg$, and if $v\in \Reg$, then there exist $t_1,t_2\in\re,\eta>0$ such that $\la^s(\f_{t_1}v)>\eta$ and $\la^u(\f_{t_2}v)>\eta$.
\end{lemma}

\begin{proof}
$\eqref{vanish:1}\implies \eqref{vanish:2}$ and $\eqref{vanish:1} \implies \eqref{vanish:3}$ follow from the proof of \cref{lemma:exponential-y-component}.
For $\eqref{vanish:2} \implies \eqref{vanish:1}$, if $\las(\f_tv)=0$ for all $t\in \re$, then for every $T\geq 0$ there is a stable \m\  Jacobi field $J_T=x_T\dot\ga_v+y_Ti\dot\ga_v$ along $\gav$ that is parallel (say, with $\lvert y_T(t)\rvert=1$) for $t\geq -T$ and $x_T(0)=0$ after reparametrization.
By compactness we get a sequence $T_k\to \infty$ for which $J_{T_k}(0)$ and $J'_{T_k}(0)$ converge to some $J(0)$, $J'(0)\in T_{\pi v}\Sigma$; 
the corresponding \m\  Jacobi field $J$ is parallel for all time, so $v\in \Sing$. 
Similarly, we have $\eqref{vanish:3})\implies \eqref{vanish:1}$.
\end{proof}

\begin{lemma}\label{lemma:la-0-onesided-parallel-JF}
For $v\in \sm$, the following are equivalent.
  \begin{enumerate}
    \item $\la^u(v)=0$;\label{onesided:1}
    \item $\la^u(\f_t(v))=0$ for all $t\leq 0$;\label{onesided:2}
    \item There is an unstable \m\  Jacobi field $J$ on $\ga_v$ such that $J(t)$ is parallel but not tangential for all $t\leq 0$.\label{onesided:3}
  \end{enumerate}
  The analogous result holds for $\la^s$ and $t\geq 0$.
\end{lemma}

\begin{proof}
It is immediate that $\eqref{onesided:3}\implies\eqref{onesided:2}\implies \eqref{onesided:1}$. 
To prove $\eqref{onesided:1}\implies\eqref{onesided:3}$, suppose $\lau(v)=0$. Then there is an unstable \m\  Jacobi field $J$ such that $\dot{y}(0)=0$. By boundedness and convexity of $y(t)$ on $(-\infty,0]$, we have $\dot{y}(t)=0$ for all $t \leq 0$. Therefore, $J(t)$ is parallel for $t\leq 0$.
\end{proof} 
We now produce vectors with strong stable Jacobi fields.
\begin{definition}[Strong (un)stable magnetic Jacobi fields]\label{def:strong-mag-jf}
    A \m\ Jacobi field $J^{ss}_w(t)$ along $\ga_w(t)$ is said to be \emph{strong stable} if $J^{ss}_w(t)\xrightarrow{t\to\infty}0$; a \m\ Jacobi field $J^{uu}_w(t)$ along $\ga_w(t)$ is said to be \emph{strong unstable} if $J^{uu}_w(t)\xrightarrow{t\to-\infty}0$.
\end{definition}
\begin{proposition}\label{lem:strong-stable-jf}
If $v \in \sm$ is such that $\varliminf_{T\to\infty}\frac1T\int_0^T \las(\fmu_tv)\;dt>0$, then there is a one-dimensional space of \m\ Jacobi fields $J_v^{ss}$ along $\ga_v$ such that $J_v^{ss}(t)\xrightarrow{t\to\infty} 0$; convergence is exponential. Likewise for $J_v^{uu}(t)\xrightarrow{t\to-\infty} 0$.
\end{proposition}
\begin{proof}
Note that nonvanishing tangential Jacobi fields do not go to 0 at infinity. For nontangential stable Jacobi fields, the assumptions and \cref{lemma:exponential-y-component} show that $y_v\to0$ exponentially, where $J_{v}=x_{v}\dot \ga_{v}+y_{v}i\dot\ga_{v}\in E^s(v)$.
Then
\begin{align}\label{eqn:X_0}
    X_0 \dfn-\int^{\infty}_0\dot x_v(t)dt=-\mu\int_0^{\infty}y_v(t)dt>-\infty,
    \end{align}
and $J^{ss}_v\dfn x^{ss}_v\dot \ga_v+y^{ss}_vi\dot \ga_v$ with \[y^{ss}_v=y_v, x_v^{ss}(0)=X_0, \text{ and } \dot x_v^{ss}=\mu y_v^{ss}\] is a stable \m\  Jacobi field along $\ga_v$ 
with $x^{ss}_v(t)=x^{ss}_v(0)+\int^t_0\dot x^{ss}_v(t)dt\xrightarrow{t\to\infty}0$.
\end{proof}
\section{Examples}
To show that this work is not vacuous, we produce a negatively curved surface with a singular magnetic geodesic or a magnetic flat strip. Briefly: start with a strip in the Poincar\'e disk that is saturated by horocycles; make them periodic and close up. Here are details. A surface of revolution 
\((z,\theta)\mapsto\bigl(r(z)\cos\theta,r(z)\sin\theta,z\bigr)\)
around the $z$-axis has 
Gauss curvature
\[
K(z)=-\frac{r''}{r(1+r'^2)^2},
\]
and a ``parallel'' $z=z_0$ (with a suitable choice of orientation) has geodesic curvature
\[
k(z)
=
\frac{r'}{r\sqrt{1+r'^2}}.
\]
A closed 1-magnetic geodesic $\{z=0\}$ arises when $k(0)=1$, i.e., \(r=\frac{|r'|}{\sqrt{1+r'^2}}\). It lies in a 1-magnetically flat region (and is hence singular) if $K(0)=-1$, i.e.,
\[
r''=r(1+r'^2)^2.
\]
Choosing here $r'=1$ gives $r=1/\sqrt2$ and then $r''=2\sqrt2$. These conditions hold for the quadratic polynomial $\frac1{\sqrt2}+z+\sqrt2z^2$, so we choose a polynomial (say)
\[
r(z)
\approx
\frac1{\sqrt2}+z+\sqrt2z^2
\]
with higher-order terms such that
\(
r''-r(1+r'^2)^2
\)
has a strict local minimum at 0, hence $K>-1$ near 0. (Brute-force computation suggests that 
\(r(z)=\frac1{\sqrt2}+z+\sqrt2\,z^2+\frac{10}{3}z^3+12z^4\)
will serve.) This gives a surface of revolution with two circle boundaries, with nonpositive 1-magnetic curvature, and with a singular closed 1-magnetic geodesic. It can be compactified to a closed surface with nonpositive 1-magnetic curvature by attaching handles as in \cite[p.\ 305]{Burns-Paternain}.

An example with a 1-magnetically flat strip (and nonpositive 1-magnetic curvature) is obtained by replacing the distinguished parallel here by a piece of the pseudospere, which has Gauss curvature $-1$ and for which parallels are horocycles, hence closed 1-magnetic geodesics.
\subsection*{Acknowledgments}
The support of Vaughn Climenhaga and Daniel Thompson were motivating and helpful in this project. We used AI in two specific ways. After proving \cref{THMStableComp}, it occurred to us that \cref{thm:sturm2} would much shorten our one-page proof of its item \eqref{item:comparison_1}, and ChatGPT suggested a proof of \cref{thm:sturm2} along the lines we give here. ChatGPT also suggested the polynomials $r(z)$ in the discussion of the examples above.
\bibliographystyle{amsalpha}
\bibliography{Z-bibs}

@article {Adachi1,
    AUTHOR = {Adachi, Toshiaki},
     TITLE = {Curvature bound and trajectories for magnetic fields on a
              {H}adamard surface},
   JOURNAL = {Tsukuba J. Math.},
  FJOURNAL = {Tsukuba Journal of Mathematics},
    VOLUME = {20},
      YEAR = {1996},
    NUMBER = {1},
     PAGES = {225--230},
      ISSN = {0387-4982,2423-821X},
   MRCLASS = {58E10 (53C22)},
  MRNUMBER = {1406045},
MRREVIEWER = {Edoh\ Amiran},
       DOI = {10.21099/tkbjm/1496162994},
       URL = {https://doi.org/10.21099/tkbjm/1496162994},
}

@article {Adachi2,
    AUTHOR = {Adachi, Toshiaki},
     TITLE = {A comparison theorem on magnetic {J}acobi fields},
   JOURNAL = {Proc. Edinburgh Math. Soc. (2)},
  FJOURNAL = {Proceedings of the Edinburgh Mathematical Society. Series II},
    VOLUME = {40},
      YEAR = {1997},
    NUMBER = {2},
     PAGES = {293--308},
      ISSN = {0013-0915,1464-3839},
   MRCLASS = {53C22 (53C55 78A35)},
  MRNUMBER = {1454024},
MRREVIEWER = {Edoh\ Amiran},
       DOI = {10.1017/S0013091500023737},
       URL = {https://doi.org/10.1017/S0013091500023737},
}

@article {Grognet,
    AUTHOR = {Grognet, St\'ephane},
     TITLE = {Flots magn\'etiques en courbure n\'egative},
   JOURNAL = {Ergodic Theory Dynam. Systems},
  FJOURNAL = {Ergodic Theory and Dynamical Systems},
    VOLUME = {19},
      YEAR = {1999},
    NUMBER = {2},
     PAGES = {413--436},
      ISSN = {0143-3857,1469-4417},
   MRCLASS = {37D40 (53D25)},
  MRNUMBER = {1685401},
MRREVIEWER = {Rafael\ Oswaldo\ Ruggiero},
       DOI = {10.1017/S0143385799126634},
       URL = {https://doi.org/10.1017/S0143385799126634},
}

@book {Anosov,
    AUTHOR = {Anosov, Dmitri},
     TITLE = {Geodesic flows on closed {R}iemann manifolds with negative
              curvature},
    SERIES = {Proceedings of the Steklov Institute of Mathematics},
    VOLUME = {No. 90 (1967)},
      NOTE = {Translated from the Russian by S. Feder},
 PUBLISHER = {American Mathematical Society, Providence, RI},
      YEAR = {1969},
     PAGES = {iv+235},
   MRCLASS = {57.50 (53.00)},
  MRNUMBER = {242194},
}

@incollection {Paternain-Paternain,
    AUTHOR = {Paternain, Gabriel P. and Paternain, Miguel},
     TITLE = {Anosov geodesic flows and twisted symplectic structures},
 BOOKTITLE = {International {C}onference on {D}ynamical {S}ystems
              ({M}ontevideo, 1995)},
    SERIES = {Pitman Res. Notes Math. Ser.},
    VOLUME = {362},
     PAGES = {132--145},
 PUBLISHER = {Longman, Harlow},
      YEAR = {1996},
      ISBN = {0-582-30296-X},
   MRCLASS = {58F17 (58F15)},
  MRNUMBER = {1460801},
MRREVIEWER = {Gerhard\ Knieper},
}

@book {Burns-Gidea,
    AUTHOR = {Burns, Keith and Gidea, Marian},
     TITLE = {Differential geometry and topology},
    SERIES = {Studies in Advanced Mathematics},
      NOTE = {With a view to dynamical systems},
 PUBLISHER = {Chapman \& Hall/CRC, Boca Raton, FL},
      YEAR = {2005},
     PAGES = {x+389},
      ISBN = {978-1-58488-253-4; 1-58488-253-0},
   MRCLASS = {53-01 (37-01 54H20 57-01 58-01)},
  MRNUMBER = {2148643},
MRREVIEWER = {Andrew\ Bucki},
}

@article {Burns-Paternain,
    AUTHOR = {Burns, Keith and Paternain, Gabriel P.},
     TITLE = {Anosov magnetic flows, critical values and topological
              entropy},
   JOURNAL = {Nonlinearity},
  FJOURNAL = {Nonlinearity},
    VOLUME = {15},
      YEAR = {2002},
    NUMBER = {2},
     PAGES = {281--314},
      ISSN = {0951-7715,1361-6544},
   MRCLASS = {37J05 (37B40 37D20 53D25)},
  MRNUMBER = {1888853},
MRREVIEWER = {Maciej\ P.\ Wojtkowski},
       DOI = {10.1088/0951-7715/15/2/305},
       URL = {https://doi.org/10.1088/0951-7715/15/2/305},
}

@book {Eberlein-book,
    AUTHOR = {Eberlein, Patrick B.},
     TITLE = {Geometry of nonpositively curved manifolds},
    SERIES = {Chicago Lectures in Mathematics},
 PUBLISHER = {University of Chicago Press, Chicago, IL},
      YEAR = {1996},
     PAGES = {vii+449},
      ISBN = {0-226-18197-9; 0-226-18198-7},
   MRCLASS = {53-02 (53C20 53C21 53C35)},
  MRNUMBER = {1441541},
MRREVIEWER = {Mar\'ia\ J.\ Druetta},
}

@article {HI,
    AUTHOR = {Heintze, Ernst and Im Hof, Hans-Christoph},
     TITLE = {Geometry of horospheres},
   JOURNAL = {J. Differential Geometry},
  FJOURNAL = {Journal of Differential Geometry},
    VOLUME = {12},
      YEAR = {1977},
    NUMBER = {4},
     PAGES = {481--491},
      ISSN = {0022-040X,1945-743X},
   MRCLASS = {53C20},
  MRNUMBER = {512919},
MRREVIEWER = {Midori\ Goto},
       URL = {http://projecteuclid.org/euclid.jdg/1214434219},
}

@article {Eberlein-non-conjugate,
    AUTHOR = {Eberlein, Patrick},
     TITLE = {Geodesic flow in certain manifolds without conjugate points},
   JOURNAL = {Trans. Amer. Math. Soc.},
  FJOURNAL = {Transactions of the American Mathematical Society},
    VOLUME = {167},
      YEAR = {1972},
     PAGES = {151--170},
      ISSN = {0002-9947,1088-6850},
   MRCLASS = {58E10 (53C20)},
  MRNUMBER = {295387},
MRREVIEWER = {W.\ Klingenberg},
       DOI = {10.2307/1996132},
       URL = {https://doi.org/10.2307/1996132},
}

@article {Ballmann,
    AUTHOR = {Ballmann, Werner},
     TITLE = {Axial isometries of manifolds of nonpositive curvature},
   JOURNAL = {Math. Ann.},
  FJOURNAL = {Mathematische Annalen},
    VOLUME = {259},
      YEAR = {1982},
    NUMBER = {1},
     PAGES = {131--144},
      ISSN = {0025-5831,1432-1807},
   MRCLASS = {53C22},
  MRNUMBER = {656659},
MRREVIEWER = {John\ Bolton},
       DOI = {10.1007/BF01456836},
       URL = {https://doi.org/10.1007/BF01456836},
}

@article {Eberlein-II-1973,
    AUTHOR = {Eberlein, Patrick},
     TITLE = {Geodesic flows on negatively curved manifolds. {II}},
   JOURNAL = {Trans. Amer. Math. Soc.},
  FJOURNAL = {Transactions of the American Mathematical Society},
    VOLUME = {178},
      YEAR = {1973},
     PAGES = {57--82},
      ISSN = {0002-9947,1088-6850},
   MRCLASS = {58F10 (53C20 53C70)},
  MRNUMBER = {314084},
MRREVIEWER = {Leon\ W.\ Green},
       DOI = {10.2307/1996689},
       URL = {https://doi.org/10.2307/1996689},
}

@article {BCFT,
    AUTHOR = {Burns, Keith and Climenhaga, Vaughn and Fisher, Todd and Thompson, Daniel J.},
     TITLE = {Unique equilibrium states for geodesic flows in nonpositive
              curvature},
   JOURNAL = {Geom. Funct. Anal.},
  FJOURNAL = {Geometric and Functional Analysis},
    VOLUME = {28},
      YEAR = {2018},
    NUMBER = {5},
     PAGES = {1209--1259},
      ISSN = {1016-443X,1420-8970},
   MRCLASS = {37D35 (37C40 37D25 37D40)},
  MRNUMBER = {3856792},
MRREVIEWER = {Boris\ Hasselblatt},
       DOI = {10.1007/s00039-018-0465-8},
       URL = {https://doi.org/10.1007/s00039-018-0465-8},
}

@article {Knieper,
    AUTHOR = {Knieper, Gerhard},
     TITLE = {The uniqueness of the measure of maximal entropy for geodesic
              flows on rank {$1$} manifolds},
   JOURNAL = {Ann. of Math. (2)},
  FJOURNAL = {Annals of Mathematics. Second Series},
    VOLUME = {148},
      YEAR = {1998},
    NUMBER = {1},
     PAGES = {291--314},
      ISSN = {0003-486X,1939-8980},
   MRCLASS = {37C40 (37B40 37D40 53D25)},
  MRNUMBER = {1652924},
MRREVIEWER = {Vadim\ A.\ Ka\u imanovich},
       DOI = {10.2307/120995},
       URL = {https://doi.org/10.2307/120995},
}

@book {Walters,
    AUTHOR = {Walters, Peter},
     TITLE = {An introduction to ergodic theory},
    SERIES = {Graduate Texts in Mathematics},
    VOLUME = {79},
 PUBLISHER = {Springer-Verlag, New York-Berlin},
      YEAR = {1982},
     PAGES = {ix+250},
      ISBN = {0-387-90599-5},
   MRCLASS = {28Dxx (54H20 58F11)},
  MRNUMBER = {648108},
MRREVIEWER = {M.\ A.\ Akcoglu},
}

@book {Katok-Hasselblatt,
    AUTHOR = {Katok, Anatole and Hasselblatt, Boris},
     TITLE = {Introduction to the modern theory of dynamical systems},
    SERIES = {Encyclopedia of Mathematics and its Applications},
    VOLUME = {54},
      NOTE = {With a supplementary chapter by Katok and Leonardo Mendoza},
 PUBLISHER = {Cambridge University Press, Cambridge},
      YEAR = {1995},
     PAGES = {xviii+802},
      ISBN = {0-521-34187-6},
   MRCLASS = {58Fxx (34Cxx 34Dxx 58-01 58F11 58F15)},
  MRNUMBER = {1326374},
MRREVIEWER = {Edoh\ Amiran},
       DOI = {10.1017/CBO9780511809187},
       URL = {https://doi.org/10.1017/CBO9780511809187},
}

@book {Fisher-Hasselblatt,
    AUTHOR = {Fisher, Todd and Hasselblatt, Boris},
     TITLE = {Hyperbolic flows},
    SERIES = {Zurich Lectures in Advanced Mathematics},
 PUBLISHER = {EMS Publishing House, Berlin},
      YEAR = {[2019] \copyright 2019},
     PAGES = {xiv+723},
      ISBN = {978-3-03719-200-9},
   MRCLASS = {37-02 (37A30 37A35 37D20 37D40)},
  MRNUMBER = {3972204},
MRREVIEWER = {Miguel\ Paternain},
       DOI = {10.4171/200},
       URL = {https://doi.org/10.4171/200},
}

@article {Gouda,
    AUTHOR = {Gouda, Norio},
     TITLE = {Magnetic flows of {A}nosov type},
   JOURNAL = {Tohoku Math. J. (2)},
  FJOURNAL = {The Tohoku Mathematical Journal. Second Series},
    VOLUME = {49},
      YEAR = {1997},
    NUMBER = {2},
     PAGES = {165--183},
      ISSN = {0040-8735,2186-585X},
   MRCLASS = {58F15 (53C22 58F17 78A35)},
  MRNUMBER = {1447180},
MRREVIEWER = {Gerhard\ Knieper},
       DOI = {10.2748/tmj/1178225145},
       URL = {https://doi.org/10.2748/tmj/1178225145},
}

@article {visibility,
    AUTHOR = {Eberlein, Patrick and O'Neill, Barrett},
     TITLE = {Visibility manifolds},
   JOURNAL = {Pacific J. Math.},
  FJOURNAL = {Pacific Journal of Mathematics},
    VOLUME = {46},
      YEAR = {1973},
     PAGES = {45--109},
      ISSN = {0030-8730,1945-5844},
   MRCLASS = {53C20},
  MRNUMBER = {336648},
MRREVIEWER = {J.\ Ferrand},
       URL = {http://projecteuclid.org/euclid.pjm/1102946601},
}

@article {Wojtkowski,
    AUTHOR = {Wojtkowski, Maciej P.},
     TITLE = {Magnetic flows and {G}aussian thermostats on manifolds of
              negative curvature},
   JOURNAL = {Fund. Math.},
  FJOURNAL = {Fundamenta Mathematicae},
    VOLUME = {163},
      YEAR = {2000},
    NUMBER = {2},
     PAGES = {177--191},
      ISSN = {0016-2736,1730-6329},
   MRCLASS = {37D20 (37D40 53D25)},
  MRNUMBER = {1752103},
MRREVIEWER = {Gerhard\ Knieper},
       DOI = {10.4064/fm-163-2-177-191},
       URL = {https://doi.org/10.4064/fm-163-2-177-191},
}

@article {Artin,
    AUTHOR = {Artin, Emil},
     TITLE = {Ein mechanisches {S}ystem mit quasiergodischen {B}ahnen},
   JOURNAL = {Abh. Math. Sem. Univ. Hamburg},
  FJOURNAL = {Abhandlungen aus dem Mathematischen Seminar der Universit\"at
              Hamburg},
    VOLUME = {3},
      YEAR = {1924},
    NUMBER = {1},
     PAGES = {170--175},
      ISSN = {0025-5858,1865-8784},
   MRCLASS = {99-04},
  MRNUMBER = {3069425},
       DOI = {10.1007/BF02954622},
       URL = {https://doi.org/10.1007/BF02954622},
}

@article {Hopf,
    AUTHOR = {Hopf, Eberhard},
     TITLE = {Statistik der geod\"atischen {L}inien in {M}annigfaltigkeiten
              negativer {K}r\"ummung},
   JOURNAL = {Ber. Verh. S\"achs. Akad. Wiss. Leipzig Math.-Phys. Kl.},
  FJOURNAL = {Berichte \"uber die Verhandlungen der S\"achsischen Akademie
              der Wissenschaften zu Leipzig. Mathematisch-Physische Klasse},
    VOLUME = {91},
      YEAR = {1939},
     PAGES = {261--304},
      ISSN = {0366-0036},
   MRCLASS = {46.3X},
  MRNUMBER = {1464},
MRREVIEWER = {Gustav\ A.\ Hedlund},
}

@article {Pesin1,
    AUTHOR = {Pesin, Yakov B.},
     TITLE = {Characteristic {L}japunov exponents, and smooth ergodic
              theory},
   JOURNAL = {Uspehi Mat. Nauk},
  FJOURNAL = {Akademija Nauk SSSR i Moskovskoe Matemati\v ceskoe Ob\v s\v
              cestvo. Uspehi Matemati\v ceskih Nauk},
    VOLUME = {32},
      YEAR = {1977},
    NUMBER = {4(196)},
     PAGES = {55--112, 287},
      ISSN = {0042-1316},
   MRCLASS = {34D05 (28A65 58F10 58F15)},
  MRNUMBER = {466791},
MRREVIEWER = {A.\ Morimoto},
}

@article {Pesin2,
    AUTHOR = {Pesin, Yakov B.},
     TITLE = {Geodesic flows in closed {R}iemannian manifolds without focal
              points},
   JOURNAL = {Izv. Akad. Nauk SSSR Ser. Mat.},
  FJOURNAL = {Izvestiya Akademii Nauk SSSR. Seriya Matematicheskaya},
    VOLUME = {41},
      YEAR = {1977},
    NUMBER = {6},
     PAGES = {1252--1288, 1447},
      ISSN = {0373-2436},
   MRCLASS = {58F15},
  MRNUMBER = {488169},
MRREVIEWER = {A.\ Morimoto},
}

@article {AnosovSinai,
    AUTHOR = {Anosov, Dmitri and Sinai, Yakov},
     TITLE = {Certain smooth ergodic systems},
   JOURNAL = {Uspehi Mat. Nauk},
  FJOURNAL = {Akademija Nauk SSSR i Moskovskoe Matemati\v ceskoe Ob\v s\v
              cestvo. Uspehi Matemati\v ceskih Nauk},
    VOLUME = {22},
      YEAR = {1967},
    NUMBER = {5(137)},
     PAGES = {107--172},
      ISSN = {0042-1316},
   MRCLASS = {28.70},
  MRNUMBER = {224771},
}

@incollection {Humsan,
    AUTHOR = {Katok, Anatole},
     TITLE = {The 1965 {H}umsan school in ergodic theory},
 BOOKTITLE = {The collected works of {A}natole {K}atok. {V}ol. 1},
     PAGES = {1205--1207},
 PUBLISHER = {World Sci. Publishing, Singapore},
      YEAR = {[2024] \copyright 2024},
      ISBN = {[9789811237751]; [9789811237768]; [9789811238062]},
   MRCLASS = {01A72 (37-03 37Axx)},
  MRNUMBER = {4893901},
}

@article {Comtet,
    AUTHOR = {Comtet, Alain},
     TITLE = {On the {L}andau levels on the hyperbolic plane},
   JOURNAL = {Ann. Physics},
  FJOURNAL = {Annals of Physics},
    VOLUME = {173},
      YEAR = {1987},
    NUMBER = {1},
     PAGES = {185--209},
      ISSN = {0003-4916,1096-035X},
   MRCLASS = {81G10 (78A35)},
  MRNUMBER = {870891},
       DOI = {10.1016/0003-4916(87)90098-4},
       URL = {https://doi.org/10.1016/0003-4916(87)90098-4},
}

@article {Sunada,
    AUTHOR = {Sunada, Toshikazu},
     TITLE = {Magnetic Flows on a {R}iemann Surface},
   JOURNAL = {Proc. KAIST Math. Ann},
  FJOURNAL = {Proceedings of KAIST Mathematics Workshop},
    VOLUME = {8},
      YEAR = {1993},
    NUMBER = {1},
     PAGES = {93--108},
 PUBLISHER ="Analysis and Geometry",
       URL = {https://cir.nii.ac.jp/crid/1574231873874473728},
}

@book {Paternainbook,
    AUTHOR = {Paternain, Gabriel P.},
     TITLE = {Geodesic flows},
    SERIES = {Progress in Mathematics},
    VOLUME = {180},
 PUBLISHER = {Birkh\"auser Boston, Inc., Boston, MA},
      YEAR = {1999},
     PAGES = {xiv+149},
      ISBN = {0-8176-4144-0},
   MRCLASS = {53D25 (37D40 37J99)},
  MRNUMBER = {1712465},
MRREVIEWER = {Boris\ Hasselblatt},
       DOI = {10.1007/978-1-4612-1600-1},
       URL = {https://doi.org/10.1007/978-1-4612-1600-1},
}

@article {PaternainRigidity,
    AUTHOR = {Paternain, Gabriel P.},
     TITLE = {Magnetic rigidity of horocycle flows},
   JOURNAL = {Pacific J. Math.},
  FJOURNAL = {Pacific Journal of Mathematics},
    VOLUME = {225},
      YEAR = {2006},
    NUMBER = {2},
     PAGES = {301--323},
      ISSN = {0030-8730,1945-5844},
   MRCLASS = {37D40 (37C27 37J05 53D25)},
  MRNUMBER = {2233738},
MRREVIEWER = {Gerhard\ Knieper},
       DOI = {10.2140/pjm.2006.225.301},
       URL = {https://doi.org/10.2140/pjm.2006.225.301},
}

@article {PeyerimhoffSiburg,
    AUTHOR = {Peyerimhoff, Norbert and Siburg, Karl Friedrich},
     TITLE = {The dynamics of magnetic flows for energies above {M}a\~n\'e's
              critical value},
   JOURNAL = {Israel J. Math.},
  FJOURNAL = {Israel Journal of Mathematics},
    VOLUME = {135},
      YEAR = {2003},
     PAGES = {269--298},
      ISSN = {0021-2172,1565-8511},
   MRCLASS = {37J50 (37D40 53D25)},
  MRNUMBER = {1997047},
MRREVIEWER = {Gabriel\ P.\ Paternain},
       DOI = {10.1007/BF02776061},
       URL = {https://doi.org/10.1007/BF02776061},
}

@article {BurnsMatveev,
    AUTHOR = {Burns, Keith and Matveev, Vladimir S.},
     TITLE = {On the rigidity of magnetic systems with the same magnetic
              geodesics},
   JOURNAL = {Proc. Amer. Math. Soc.},
  FJOURNAL = {Proceedings of the American Mathematical Society},
    VOLUME = {134},
      YEAR = {2006},
    NUMBER = {2},
     PAGES = {427--434},
      ISSN = {0002-9939,1088-6826},
   MRCLASS = {37D40 (37J05 53C24 53D25)},
  MRNUMBER = {2176011},
MRREVIEWER = {Norbert\ Peyerimhoff},
       DOI = {10.1090/S0002-9939-05-08196-7},
       URL = {https://doi.org/10.1090/S0002-9939-05-08196-7},
}

@article{ComtetHouston,
  title={Effective action on the hyperbolic plane in a constant external field},
  author={Alain Comtet and P. J. Houston},
  journal={Journal of Mathematical Physics},
  year={1985},
  volume={26},
  pages={185-191},
  url={https://api.semanticscholar.org/CorpusID:119496252}
}

@article {GomesRuggiero,
    AUTHOR = {Gomes, Jos\'e{} Barbosa and Ruggiero, Rafael O.},
     TITLE = {Rigidity of magnetic flows for compact surfaces},
   JOURNAL = {C. R. Math. Acad. Sci. Paris},
  FJOURNAL = {Comptes Rendus Math\'ematique. Acad\'emie des Sciences. Paris},
    VOLUME = {346},
      YEAR = {2008},
    NUMBER = {5-6},
     PAGES = {313--316},
      ISSN = {1631-073X,1778-3569},
   MRCLASS = {37D40 (37C15 37E35 53C12)},
  MRNUMBER = {2414176},
MRREVIEWER = {Donal\ Hurley},
       DOI = {10.1016/j.crma.2008.01.011},
       URL = {https://doi.org/10.1016/j.crma.2008.01.011},
}

@article {Paternain-Paternain-1997,
    AUTHOR = {Paternain, Gabriel P. and Paternain, Miguel},
     TITLE = {First derivative of topological entropy for {A}nosov geodesic
              flows in the presence of magnetic fields},
   JOURNAL = {Nonlinearity},
  FJOURNAL = {Nonlinearity},
    VOLUME = {10},
      YEAR = {1997},
    NUMBER = {1},
     PAGES = {121--131},
      ISSN = {0951-7715,1361-6544},
   MRCLASS = {58F15 (58F05 58F17)},
  MRNUMBER = {1430743},
MRREVIEWER = {Gerhard\ Knieper},
       DOI = {10.1088/0951-7715/10/1/008},
       URL = {https://doi.org/10.1088/0951-7715/10/1/008},
}

@article {Paternain-regularity,
    AUTHOR = {Paternain, Gabriel P.},
     TITLE = {On the regularity of the {A}nosov splitting for twisted
              geodesic flows},
   JOURNAL = {Math. Res. Lett.},
  FJOURNAL = {Mathematical Research Letters},
    VOLUME = {4},
      YEAR = {1997},
    NUMBER = {6},
     PAGES = {871--888},
      ISSN = {1073-2780},
   MRCLASS = {58F15 (58F17)},
  MRNUMBER = {1492126},
MRREVIEWER = {Boris\ Hasselblatt},
       DOI = {10.4310/MRL.1997.v4.n6.a7},
       URL = {https://doi.org/10.4310/MRL.1997.v4.n6.a7},
}

@article {Grognet-entropy,
    AUTHOR = {Grognet, St\'ephane},
     TITLE = {Entropies des flots magn\'etiques},
   JOURNAL = {Ann. Inst. H. Poincar\'e{} Phys. Th\'eor.},
  FJOURNAL = {Annales de l'Institut Henri Poincar\'e. Physique Th\'eorique},
    VOLUME = {71},
      YEAR = {1999},
    NUMBER = {4},
     PAGES = {395--424},
      ISSN = {0246-0211},
   MRCLASS = {37D40 (37B40 53D25)},
  MRNUMBER = {1721559},
MRREVIEWER = {Boris\ Hasselblatt},
       URL = {http://www.numdam.org/item?id=AIHPA_1999__71_4_395_0},
}

@book {John-Lee-Riemannian-MFD,
    AUTHOR = {Lee, John M.},
     TITLE = {Introduction to {R}iemannian manifolds},
    SERIES = {Graduate Texts in Mathematics},
    VOLUME = {176},
   EDITION = {Second},
 PUBLISHER = {Springer, Cham},
      YEAR = {2018},
     PAGES = {xiii+437},
      ISBN = {978-3-319-91754-2; 978-3-319-91755-9},
   MRCLASS = {53-01 (53B20 53B30 53C20 53C21)},
  MRNUMBER = {3887684},
MRREVIEWER = {Robert\ J.\ Low},
}

@book {Handbook-Hasselblatt-Katok,
     TITLE = {Handbook of dynamical systems. {V}ol. 1{A}},
    EDITOR = {Hasselblatt, B. and Katok, A.},
 PUBLISHER = {North-Holland, Amsterdam},
      YEAR = {2002},
     PAGES = {xii+1220},
      ISBN = {0-444-82669-6},
   MRCLASS = {37-06},
  MRNUMBER = {1928517},
}

@article {Myers-Steenrod-theorem,
    AUTHOR = {Myers, Sumner Byron and Steenrod, Norman Earl},
     TITLE = {The group of isometries of a {R}iemannian manifold},
   JOURNAL = {Ann. of Math. (2)},
  FJOURNAL = {Annals of Mathematics. Second Series},
    VOLUME = {40},
      YEAR = {1939},
    NUMBER = {2},
     PAGES = {400--416},
      ISSN = {0003-486X,1939-8980},
   MRCLASS = {99-04},
  MRNUMBER = {1503467},
       DOI = {10.2307/1968928},
       URL = {https://doi.org/10.2307/1968928},
}

@article {Hedlund-transitivity-of-magnetic-flat,
    AUTHOR = {Hedlund, Gustav A.},
     TITLE = {On the metrical transitivity of the geodesics on closed
              surfaces of constant negative curvature},
   JOURNAL = {Ann. of Math. (2)},
  FJOURNAL = {Annals of Mathematics. Second Series},
    VOLUME = {35},
      YEAR = {1934},
    NUMBER = {4},
     PAGES = {787--808},
      ISSN = {0003-486X,1939-8980},
   MRCLASS = {99-04},
  MRNUMBER = {1503197},
       DOI = {10.2307/1968495},
       URL = {https://doi.org/10.2307/1968495},
}

@article {Eberlein-geodesic-flow-negative-curvature-I,
    AUTHOR = {Eberlein, Patrick},
     TITLE = {Geodesic flows on negatively curved manifolds. {I}},
   JOURNAL = {Ann. of Math. (2)},
  FJOURNAL = {Annals of Mathematics. Second Series},
    VOLUME = {95},
      YEAR = {1972},
     PAGES = {492--510},
      ISSN = {0003-486X},
   MRCLASS = {58F15},
  MRNUMBER = {310926},
MRREVIEWER = {Leon\ W.\ Green},
       DOI = {10.2307/1970869},
       URL = {https://doi.org/10.2307/1970869},
}

@article {SSChen-Eberlein,
    AUTHOR = {Chen, Su Shing and Eberlein, Patrick},
     TITLE = {Isometry groups of simply connected manifolds of nonpositive
              curvature},
   JOURNAL = {Illinois J. Math.},
  FJOURNAL = {Illinois Journal of Mathematics},
    VOLUME = {24},
      YEAR = {1980},
    NUMBER = {1},
     PAGES = {73--103},
      ISSN = {0019-2082},
   MRCLASS = {53C10 (53C20)},
  MRNUMBER = {550653},
MRREVIEWER = {Witold\ Mozgawa},
       URL = {http://projecteuclid.org/euclid.ijm/1256047798},
}

@article {Bishop-Oneil,
    AUTHOR = {Bishop, Richard L. and O'Neill, Barrett},
     TITLE = {Manifolds of negative curvature},
   JOURNAL = {Trans. Amer. Math. Soc.},
  FJOURNAL = {Transactions of the American Mathematical Society},
    VOLUME = {145},
      YEAR = {1969},
     PAGES = {1--49},
      ISSN = {0002-9947,1088-6850},
   MRCLASS = {53.72 (57.00)},
  MRNUMBER = {251664},
MRREVIEWER = {T.\ Nagano},
       DOI = {10.2307/1995057},
       URL = {https://doi.org/10.2307/1995057},
}

@article {Bowen-h-expansive,
    AUTHOR = {Bowen, Rufus},
     TITLE = {Entropy-expansive maps},
   JOURNAL = {Trans. Amer. Math. Soc.},
  FJOURNAL = {Transactions of the American Mathematical Society},
    VOLUME = {164},
      YEAR = {1972},
     PAGES = {323--331},
      ISSN = {0002-9947,1088-6850},
   MRCLASS = {28.70 (54.00)},
  MRNUMBER = {285689},
MRREVIEWER = {Karl\ Sigmund},
       DOI = {10.2307/1995978},
       URL = {https://doi.org/10.2307/1995978},
}

@book {Bridson-Haefliger,
    AUTHOR = {Bridson, Martin R. and Haefliger, Andr\'e},
     TITLE = {Metric spaces of non-positive curvature},
    SERIES = {Grundlehren der mathematischen Wissenschaften [Fundamental
              Principles of Mathematical Sciences]},
    VOLUME = {319},
 PUBLISHER = {Springer-Verlag, Berlin},
      YEAR = {1999},
     PAGES = {xxii+643},
      ISBN = {3-540-64324-9},
   MRCLASS = {53C23 (20F65 53C70 57M07)},
  MRNUMBER = {1744486},
MRREVIEWER = {Athanase\ Papadopoulos},
       DOI = {10.1007/978-3-662-12494-9},
       URL = {https://doi.org/10.1007/978-3-662-12494-9},
}

@incollection {Furstenburg,
    AUTHOR = {Furstenberg, Harry},
     TITLE = {The unique ergodicity of the horocycle flow},
 BOOKTITLE = {Recent advances in topological dynamics ({P}roc. {C}onf.
              {T}opological {D}ynamics, {Y}ale {U}niv., {N}ew {H}aven,
              {C}onn., 1972; in honor of {G}ustav {A}rnold {H}edlund)},
    SERIES = {Lecture Notes in Math.},
    VOLUME = {Vol. 318},
     PAGES = {95--115},
 PUBLISHER = {Springer, Berlin-New York},
      YEAR = {1973},
   MRCLASS = {22D40 (58F15)},
  MRNUMBER = {393339},
MRREVIEWER = {Leonard\ F.\ Richardson},
}

@article {Ballmann-Brin-Eberlein,
    AUTHOR = {Ballmann, Werner and Brin, Misha and Eberlein, Patrick},
     TITLE = {Structure of manifolds of nonpositive curvature. {I}},
   JOURNAL = {Ann. of Math. (2)},
  FJOURNAL = {Annals of Mathematics. Second Series},
    VOLUME = {122},
      YEAR = {1985},
    NUMBER = {1},
     PAGES = {171--203},
      ISSN = {0003-486X,1939-8980},
   MRCLASS = {58F17 (53C20)},
  MRNUMBER = {799256},
MRREVIEWER = {Midori\ Goto},
       DOI = {10.2307/1971373},
       URL = {https://doi.org/10.2307/1971373},
}

@article {Bowen75,
    AUTHOR = {Bowen, Rufus},
     TITLE = {Some systems with unique equilibrium states},
   JOURNAL = {Math. Systems Theory},
  FJOURNAL = {Mathematical Systems Theory. An International Journal on
              Mathematical Computing Theory},
    VOLUME = {8},
      YEAR = {1974/75},
    NUMBER = {3},
     PAGES = {193--202},
      ISSN = {0025-5661},
   MRCLASS = {28A65},
  MRNUMBER = {399413},
MRREVIEWER = {B.\ Weiss},
       DOI = {10.1007/BF01762666},
       URL = {https://doi.org/10.1007/BF01762666},
}

@misc{Ballmann_lecture,
  author       = {Ballmann, Werner},
  title        = {Lectures on Spaces of Nonpositive Curvature},
  note         = {Lecture notes, Max Planck Institute for Mathematics, Bonn},
  year         = {1995},
  url          = {https://people.mpim-bonn.mpg.de/hwbllmnn/archiv/NPC0606.pdf}
}

@article {Climenhaga-Thompson-advances,
    AUTHOR = {Climenhaga, Vaughn and Thompson, Daniel J.},
     TITLE = {Unique equilibrium states for flows and homeomorphisms with
              non-uniform structure},
   JOURNAL = {Adv. Math.},
  FJOURNAL = {Advances in Mathematics},
    VOLUME = {303},
      YEAR = {2016},
     PAGES = {745--799},
      ISSN = {0001-8708,1090-2082},
   MRCLASS = {37D35},
  MRNUMBER = {3552538},
MRREVIEWER = {Mike\ Todd},
       DOI = {10.1016/j.aim.2016.07.029},
       URL = {https://doi.org/10.1016/j.aim.2016.07.029},
}

@book {DoCarmo-Riem-Geo,
    AUTHOR = {do Carmo, Manfredo Perdig\~ao},
     TITLE = {Riemannian geometry},
    SERIES = {Mathematics: Theory \& Applications},
   EDITION = {Portuguese},
 PUBLISHER = {Birkh\"auser Boston, Inc., Boston, MA},
      YEAR = {1992},
     PAGES = {xiv+300},
      ISBN = {0-8176-3490-8},
   MRCLASS = {53-01},
  MRNUMBER = {1138207},
MRREVIEWER = {Bang-yen\ Chen},
       DOI = {10.1007/978-1-4757-2201-7},
       URL = {https://doi.org/10.1007/978-1-4757-2201-7},
}

@article {HurderKatok,
    AUTHOR = {Hurder, S. and Katok, A.},
     TITLE = {Differentiability, rigidity and {G}odbillon-{V}ey classes for
              {A}nosov flows},
   JOURNAL = {Inst. Hautes \'Etudes Sci. Publ. Math.},
  FJOURNAL = {Institut des Hautes \'Etudes Scientifiques. Publications
              Math\'ematiques},
    NUMBER = {72},
      YEAR = {1990},
     PAGES = {5--61},
      ISSN = {0073-8301,1618-1913},
   MRCLASS = {58F18 (57R30 58F15 58F17)},
  MRNUMBER = {1087392},
MRREVIEWER = {Takashi\ Tsuboi},
       URL = {http://www.numdam.org/item?id=PMIHES_1990__72__5_0},
}

@article {MarkedLength,
    AUTHOR = {Assenza, Valerio and De Simoi, Jacopo and Marshall Reber,
              James and Terek, Ivo},
     TITLE = {Marked length spectrum rigidity for {A}nosov magnetic
              surfaces},
   JOURNAL = {Adv. Math.},
  FJOURNAL = {Advances in Mathematics},
    VOLUME = {500},
      YEAR = {2026},
     PAGES = {Paper No. 111091, 25},
      ISSN = {0001-8708,1090-2082},
   MRCLASS = {37D40 (37C10 37C27 53C24)},
  MRNUMBER = {5083058},
       DOI = {10.1016/j.aim.2026.111091},
       URL = {https://doi.org/10.1016/j.aim.2026.111091},
}

@article{GLP25,
  author  = {Guillarmou, Colin and Lefeuvre, Thibault and Paternain, Gabriel P.},
  title   = {Marked length spectrum rigidity for {A}nosov surfaces},
  journal = {Duke Math. J.},
  volume  = {174},
  number  = {1},
  pages   = {131--157},
  year    = {2025}
}

@article{MRS24,
  author  = {Marshall Reber, James and Shen, Yumin},
  title   = {Anosov magnetic flows on surfaces},
  journal = {arXiv preprint arXiv:2406.18735},
  year    = {2024}
}

@article{MR23,
  author  = {Marshall Reber, James},
  title   = {Deformative magnetic marked length spectrum rigidity},
  journal = {Bull. Lond. Math. Soc.},
  volume  = {55},
  number  = {6},
  pages   = {3077--3096},
  year    = {2023},
  doi     = {10.1112/blms.12911}
}

@article{MR24Corr,
  author  = {Marshall Reber, James},
  title   = {Corrigendum: {D}eformative magnetic marked length spectrum rigidity},
  journal = {Bull. Lond. Math. Soc.},
  volume  = {56},
  number  = {12},
  pages   = {3920--3923},
  year    = {2024},
  doi     = {10.1112/blms.13195}
}

@misc{beaufort2026markedmagneticactionrigidity,
      title={Marked magnetic action rigidity}, 
      author={Louis-Brahim Beaufort and Sebastián Muñoz-Thon and Sean Richardson},
      year={2026},
      eprint={2604.08262},
      archivePrefix={arXiv},
      primaryClass={math.DS},
      url={https://arxiv.org/abs/2604.08262}, 
}

@article {Echevarria,
    AUTHOR = {Echevarr\'ia Cuesta, Javier},
     TITLE = {Smooth orbit equivalence rigidity for dissipative geodesic
              flows},
   JOURNAL = {Ergodic Theory Dynam. Systems},
  FJOURNAL = {Ergodic Theory and Dynamical Systems},
    VOLUME = {45},
      YEAR = {2025},
    NUMBER = {12},
     PAGES = {3698--3727},
      ISSN = {0143-3857,1469-4417},
   MRCLASS = {37D40 (37C15 53C24)},
  MRNUMBER = {4986555},
MRREVIEWER = {Sergei\ V.\ Agapov},
       DOI = {10.1017/etds.2025.20},
       URL = {https://doi.org/10.1017/etds.2025.20},
}

@misc{beaufort2026tensortomographyframeflow,
      title={Tensor tomography and frame flow ergodicity for magnetic flows in higher dimensions}, 
      author={Louis-Brahim Beaufort},
      year={2026},
      eprint={2604.12495},
      archivePrefix={arXiv},
      primaryClass={math.DG},
      url={https://arxiv.org/abs/2604.12495}, 
}

@article {Franco,
    AUTHOR = {Franco, Ernesto},
     TITLE = {Flows with unique equilibrium states},
   JOURNAL = {Amer. J. Math.},
  FJOURNAL = {American Journal of Mathematics},
    VOLUME = {99},
      YEAR = {1977},
    NUMBER = {3},
     PAGES = {486--514},
      ISSN = {0002-9327,1080-6377},
   MRCLASS = {28A65},
  MRNUMBER = {442193},
MRREVIEWER = {Peter\ Walters},
       DOI = {10.2307/2373927},
       URL = {https://doi.org/10.2307/2373927},
}

@article (ColesHasselblattWu,
AUTHOR = {Solly Coles and Boris Hasselblatt and Weisheng Wu},
TITLE = {Equilibrium states for magnetic flows on surfaces with nonpositive magnetic curvature},
NOTE = {in preparation},
)
\end{document}